\documentclass[11pt,reqno]{amsart}
\title{Perverse schobers, stability conditions and quadratic differentials I}

\author{Merlin Christ}
\address{MC: Mathematisches Institut, Universit\"at Bonn, Endenicher Allee 60, 53115 Bonn, Germany}
\email{christ@math.uni-bonn.de}
\thanks{
M.C. was funded by the Deutsche Forschungsgemeinschaft
(DFG, German Research Foundation) under Germany’s Excellence Strategy–EXC-2047/1–390685813. M.C. acknowledges support by the Deutsche Forschungsgemeinschaft under Germany’s Excellence Strategy–EXC 2121 “Quantum Universe” – 390833306.}

\author{Fabian Haiden}
\address{FH: Centre for Quantum Mathematics, Department of Mathematics and Computer Science, University of Southern Denmark, Campusvej 55, 5230 Odense, Denmark}
\email{fab@sdu.dk}
\thanks{
F.H. is supported by the VILLUM FONDEN, VILLUM Investigator grant 37814 and the Sapere Aude grant 3120-00076B from the Independent Research Fund Denmark (DFF).
This paper is partly a result of the ERC-SyG project Recursive and Exact New Quantum Theory (ReNewQuantum) which received funding from the European Research Council (ERC) under the European Union's Horizon 2020 research and innovation programme under grant agreement No 810573.
}

\author{Yu Qiu}
\address{Qy:
	Yau Mathematical Sciences Center and Department of Mathematical Sciences,
	Tsinghua University,
    100084 Beijing,
    China.
    \&
    Beijing Institute of Mathematical Sciences and Applications, Yanqi Lake, Beijing, China}
\email{yu.qiu@bath.edu}
\thanks{
Qy is supported by
National Natural Science Foundation of China (No. 12425104, 12031007 and 12271279) and
National Key R\&D Program of China (No. 2020YFA0713000).}

\usepackage[margin=3cm]{geometry}

\usepackage{float}

\usepackage[colorlinks, linkcolor=blue!50,anchorcolor=Periwinkle,
    citecolor=blue!72,urlcolor=cyan, bookmarksopen,bookmarksdepth=2]{hyperref}
\usepackage[usenames,dvipsnames]{xcolor}
\usepackage{enumerate}
\usepackage{subfigure}
\usepackage{bookmark}
\usepackage{url}
\usepackage{ifthen}
\usepackage{cleveref}
\usepackage{enumitem}
\usepackage{array}   
\newcolumntype{L}{>{$}l<{$}} 
\usepackage{verbatim}

\usepackage{tikz}\usetikzlibrary{matrix}
\usepackage{tikz-cd}

\usetikzlibrary{matrix, calc, arrows, shapes,
decorations.pathreplacing, decorations.markings, decorations.pathmorphing,
backgrounds,fit,positioning,shapes.symbols,chains,shadings,fadings}

\tikzset{->-/.style={decoration={  markings,  mark=at position #1 with
    {\arrow{>}}},postaction={decorate}}}
\tikzset{-<-/.style={decoration={  markings,  mark=at position #1 with
    {\arrow{<}}},postaction={decorate}}}

\usepackage{amsfonts}
\usepackage{amsmath,amssymb,amsthm,amsxtra}
\usepackage{mathtools} 
\usepackage{extarrows} 

\theoremstyle{plain}
\newtheorem{theorem}{Theorem}[section]
\newtheorem{lemma}[theorem]{Lemma}
\newtheorem{corollary}[theorem]{Corollary}
\newtheorem{proposition}[theorem]{Proposition}

\theoremstyle{definition}
\newtheorem{definition}[theorem]{Definition}
\newtheorem{example}[theorem]{Example}
\newtheorem{remark}[theorem]{Remark}
\newtheorem{construction}[theorem]{Construction}

\numberwithin{equation}{section}
\numberwithin{figure}{section}

\def\be{\begin{equation}}
\def\ee{\end{equation}}

\newcommand\hua{\mathcal}

\newcommand\NN{\mathbb{N}}
\newcommand\ZZ{\mathbb{Z}}

\newcommand\RR{\mathbb{R}}
\newcommand\CC{\mathbb{C}}

\newcommand\<{\langle}
\renewcommand\>{\rangle}

\renewcommand{\setminus}{\smallsetminus}
\renewcommand{\emptyset}{\varnothing}
\newcommand{\isom}{\cong}
\newcommand\noloc{
  \nobreak
  \mspace{6mu plus 1mu}
  {:}
  \nonscript\mkern-\thinmuskip
  \mathpunct{}
  \mspace{2mu}
}

\def\on{\operatorname} 
\mathchardef\mhyphen="2D

\newcommand\Sim{\operatorname{Sim}}
\newcommand\Hom{\operatorname{Hom}}

\newcommand\Ext{\operatorname{Ext}}

\newcommand{\h}{\operatorname{\hua{H}}} 
\newcommand{\C}{\operatorname{\hua{C}}} 
\newcommand{\D}{\operatorname{\hua{D}}} 

\newcommand\Stab{\operatorname{Stab}} 
\newcommand\Stap{\Stab^\circ} 

\newcommand{\EG}{\operatorname{EG}} 
\newcommand{\EGp}{\EG^\circ}       
\newcommand{\EGb}{\EG^\bullet}       

\newcommand{\tilt}[3]{{#1}^{#2}_{#3}}
\newcommand{\Cone}{\operatorname{Cone}}

\usepackage[all]{xy}

\def\wt{\mathbf{w}}
\newcommand\surf{\mathbf{S}}  
\newcommand\sow{\surf_\wt}  

\newcommand\T{\mathbb{T}} 
\newcommand\M{\mathbf{M}} 
\newcommand\W{\Delta} 
\newcommand\A{\Gamma} 

\newcommand{\FQuad}[2]{\operatorname{FQuad}^{#1}(#2)}

\newcommand{\UHP}{\mathbf{H}} 
\newcommand{\skel}{\wp} 
\newcommand{\cub}{\operatorname{U}} 

\tikzcdset{arrow style=tikz, diagrams={>=stealth}}

\newcommand*\cocolon{
        \nobreak
        \mskip6mu plus1mu
        \mathpunct{}%
        \nonscript
        \mkern-\thinmuskip
        {:}%
        \mskip2mu
        \relax
}

\def\hX{\hua{X}}

\def\hT{\hua{T}}
\def\hF{\hua{F}}
\def\hP{\hua{P}}

\def\hC{\hua{C}}

\def\Hone{\operatorname{H}_1}

\newcommand\AS{\mathbb{A}} 
\newcommand\dAS{\AS^*} 
\newcommand\Sgh{\mathbb{S}}
\newcommand\eS{\widehat{\Sgh}}
\def\nn{node{$\bullet$}}
\def\ww{node[white]{$\bullet$}node[red]{$\circ$}}

\newcommand\rgraph{{\bf G}} 
\newcommand\glsec{{\Gamma}} 
\newcommand\losec{\Gamma_{\rm{lax}}} 
\newcommand\CS{\hC({\Sgh,\hF})} 
\newcommand\CSh{\CS^\heartsuit} 
\newcommand\Ch{\hC^\heartsuit} 
\def\ASG{\A_\Sgh} 
\newcommand\rg{{\widehat{\rgraph}}} 

\def\eprime{h}
\def\aprime{b}
\def\vprime{u}
\def\spider{\on{\mathbf{Sp}}}
\def\NUM{\Lambda}
\def\NUMS{\Lambda_\Sgh}
\def\prodS{\Pi_{\mathrm{sph}}}     

\begin{document}
\begin{abstract}
We develop a unified approach for identifying spaces of stability conditions of triangulated categories arising from weighted marked surfaces with moduli spaces of quadratic differentials.
This identification is based on the use of perverse schobers (perverse sheaves of triangulated categories) and a notion of \textit{positive arc system kit} on a perverse schober $\mathcal F$, which provides a systematic way of assigning to a graded curve on the surface a global section of $\mathcal F$. This assignment allows us to identify mixed-angulations and their flips with finite-length hearts and their tilts. As an application we obtain a generalization of the results of Bridgeland--Smith to quadratic differentials with arbitrary singularity type (zero/pole/exponential).

\end{abstract}
\maketitle
\vspace{-1.4em}
\tableofcontents\addtocontents{toc}{\protect\setcounter{tocdepth}{1}}

\section{Introduction}

In the seminal work~\cite{B1}, Bridgeland introduced the notion of a stability condition on a triangulated category
based on results on slope stability of vector bundles over curves and $\Pi$-stability of D-branes in string theory.
The main result of~\cite{B1} is that the space of stability conditions on a fixed triangulated category is a complex (in fact affine) manifold.
It was suggested independently by Kontsevich and Seidel that
moduli spaces of quadratic differentials could be closely related to spaces of stability conditions on Fukaya-type categories of surfaces.
Several results of this kind for various classes of quadratic differentials and categories associated with surfaces are available by now \cite{BS,Qiu16,HKK17,Ike17,IQ18,KQ2,BQS,Hai21,BMQS}.
The main novelty of the present work is that
it unifies and generalizes many of these results using the notion of perverse schobers on a surface (due to Kapranov--Schechtman~\cite{KS14}).

The present work can also be seen as part of a general program, initially suggested by Kontsevich, of studying stability conditions on Fukaya categories of surfaces with coefficients in a perverse schober, see also \cite{HKS}.

\subsection{Weighted marked surfaces and mixed-angulations}
To link the geometry of quadratic differentials and the algebra of tilting theory we use a combinatorial structure called a \textit{weighted marked surface} $\sow$ together with a \textit{mixed-angulation}. This is an oriented surface obtained by gluing polygons of varying numbers $m\geq 1$ of edges along pairs of edges. We also allow the limiting case $m=\infty$, i.e.\ $\infty$-gons. The underlying surface $\bf S$ with boundary inherits some additional markings from the mixed-angulation which make it a \textit{weighted marked surface}: the set $\M \subset\bf S$ of vertices of the polygons, the \textit{marked points}, and a set $\W\subset\bf S$ of \textit{singular points}, each equipped with a degree $1\leq m\leq \infty$, such that each $m$-gon contains a unique singular point (lying on its boundary if and only if $m=\infty$). The mixed-angulation also induces a grading structure on $\bf S$. We refer to \Cref{subsec:WMS} for the precise definitions. The purpose of the notion of a weighted marked surface is to collect the common data contained in different mixed-angulations of the same surface that are related by sequences of flips.

Mixed-angulations with their polygons with varying numbers of edges, including $\infty$-gons, generalize the well-known triangulations and $n$-angulations of marked surfaces with finitely many marked points: a triangulation or $n$-angulation corresponds to a mixed-angulation where all polygons are triangles, respectively, $n$-gons.

The dual graph, $\Sgh$, of a mixed-angulation, has as vertices the set $\W$ of singular points and edges dual to each of the interior edges of the mixed-angulation. $\Sgh$ has the structure of an \textit{S-graph}, which is a ribbon graph with a positive number assigned to each corner keeping track of the number ($+1$) of boundary edges between a given pair of interior edges of a polygon of the corresponding mixed-angulation. Some examples are depicted in \Cref{fig:S-graphs}. As can be seen in the second example in \Cref{fig:S-graphs}, the difference between the valency $r$ and the degree $m$ of a singular point counts the number of boundary edges in the polygon.

\begin{figure}[ht]\centering
\makebox[\textwidth][c]{
\begin{tikzpicture}[scale=.4,rotate=0,arrow/.style={->-=.5,>=stealth,thick}]
\draw[thick](0,0) circle (5);
\foreach \j in {0,...,4}{
    \draw[dashed,white,very thick]
    (54+72*\j:5)arc (54+72*\j:54+72*\j+18:5)
    (54+72*\j:5)arc (54+72*\j:54+72*\j-18:5);}
\foreach \j in {0,...,4}{
    \draw[font=\small] (90+72*\j:5) coordinate  (w\j) edge[blue, thick] (18+72*\j:5)
    (72*\j+5:5)\nn(72*\j-5:5)\nn;}
\draw[blue, thick](w3)to(w0)to(w2);
\foreach \j in {0,...,4}{\draw(w\j)\nn;}
   \draw[red,thick,font=\tiny](0,-5) node[below]{$\infty$}\ww to (0,0) node[above]{$3$}\ww to (90-72:3.2) node[right]{$3$}\ww to (90-36:5) node[above]{$\infty$}\ww (90-72:3.2)\ww to (90-108:5) node[right]{$\infty$}\ww;
    \draw[red,thick,font=\tiny](0,0) \ww to (90+72:3.2) node[left]{$3$}\ww to (90+36:5) node[above]{$\infty$}\ww (90+72:3.2)\ww to(90+108:5)node[left]{$\infty$}\ww;
\begin{scope}[shift={(-24,0)}]
\draw[red,thick](0:2.5)to(90:2.5)to(180:2.5)to(-90:2.5)to(0:2.5);
    \draw[red,thick,font=\tiny] (0:2.5) node[left]{$3$}\ww to (0:4.5);
    \draw[red,thick,font=\tiny] (90:2.5) node[below]{$3$}\ww to (90:4.5);
    \draw[red,thick,font=\tiny] (180:2.5) node[right]{$3$}\ww to (180:4.5);
    \draw[red,thick,font=\tiny] (270:2.5) node[above]{$3$}\ww to (270:4.5);

\begin{scope}[thick,blue,decoration={
    markings,
    mark=at position 0.5 with {\arrow{stealth}}}
    ]
    \draw[postaction={decorate}] (-4.5,-4.5) to (-4.5,4.5);
    \draw[postaction={decorate}] (4.5,-4.5) to (4.5,4.5);
\end{scope}
\begin{scope}[thick,blue,decoration={
    markings,
    mark=at position 0.5 with {\arrow{{stealth}{stealth}}}}
    ]
    \draw[postaction={decorate}] (-4.5,4.5) to (4.5,4.5);
    \draw[postaction={decorate}] (-4.5,-4.5) to (4.5,-4.5);
\end{scope}
\draw[blue,thick](4.5,-4.5) to (-4.5,4.5) (-4.5,-4.5) to (4.5,4.5);
\draw(0,0)\nn(4.5,-4.5)\nn(-4.5,4.5)\nn(-4.5,-4.5)\nn(4.5,4.5)\nn;
\end{scope}
\begin{scope}[shift={(-12,0)}]
\foreach \j in {0,...,8}{\draw (40*\j:5) coordinate  (w\j) edge[very thick] (40+40*\j:5);}
\draw[blue,thick](w0)to(w0)to[bend left=-15](w3)to[bend left=-15](w6)to[bend left=-15](w0)
    (5,0).. controls +(160:3) and +(90:.5) ..(1.25,0).. controls +(-90:.5) and +(200:3) ..(5,0)
    (5,0).. controls +(140:5) and +(90:1) ..(-1.25,0).. controls +(-90:1) and +(220:5) ..(5,0);
\foreach \j in {0,...,8}{\draw(w\j)\nn;}
\draw[red,thick,font=\tiny](-4,0)\ww node[above]{$4$}to(0,0)\ww node[above]{$2$}to(2.5,0)\ww node[right]{$1$}
    (45:4)\ww node[above]{$4$}to[bend left=-30](-2.5,0)\ww node[above]{$4$}to[bend left=-30](-45:4)\ww node[below]{$4$};
\end{scope}
\end{tikzpicture}}
\caption{Three examples of weighted marked surfaces with a mixed-angulation and dual S-graph. The left one is an ideal triangulation of the torus (with no boundary) with two interior marked points. The central and right examples are mixed-angulations of the disk. The central one has singular points with degrees $1,2$, and $4$ which are the centers of $1$-, $2$-, and $4$-gons, respectively. The right one includes boundary singular points and infinitely many boundary marked points arising from $\infty$-gons.}
\label{fig:S-graphs}
\end{figure}

\subsection{Perverse schobers and the tilting theory of arcs}\label{introsec:schobers}
Perverse schobers are a, in general conjectural, categorification of perverse sheaves proposed by Kapranov--Schechtman \cite{KS14}.
Perverse sheaves on surfaces with boundary can be described in terms of
constructible (co)sheaves on ribbon graphs embedded in the surfaces, see \cite{KS16}. To categorify this description, we consider constructible sheaves of stable $\infty$-categories on ribbon graphs, subject to certain local conditions, see \Cref{sec:backgroundschobers} and \cite{Chr22} for discussions. Such a constructible sheaf is called a perverse schober parametrized by the ribbon graph. Concretely, given a ribbon graph ${\bf G}$, we can describe constructible sheaves on ${\bf G}$ as functors $\on{Exit}({\bf G})\to \on{St}$ from the exit path category of the ribbon graph to the $\infty$-category of stable $\infty$-categories.

Given a perverse schober $\mathcal{F}$ parametrized by a ribbon graph ${\bf G}$, we can consider its stable $\infty$-category of global sections, which is simply defined as the $\infty$-categorical limit of the functor $\on{Exit}({\bf G})\to \on{St}$.  This category can be thought of as a topological Fukaya category of the surface \textit{with coefficients} in $\mathcal{F}$, analogous to how sheaf cohomology generalizes ordinary cohomology. By choosing suitable perverse schobers, one obtains as global sections many derived categories naturally constructed from surfaces with markings, see also \Cref{introsec:examples} below for examples.

Consider now a weighted marked surface $\sow$ with a mixed-angulation and dual S-graph $\Sgh$.  Each graded arc in $\sow$, meaning an embedded curve with endpoints at the singular points together with a grading relative to the grading structure of $\sow$, defines an object in the topological Fukaya category. Given an arbitrary perverse schober $\mathcal{F}$ parametrized by the S-graph $\Sgh$, we can ask if we can similarly associate global sections with graded arcs. For this, we intersect the graded arc with the mixed-angulation, yielding local graded arc segments that decompose the arc. A natural way to obtain a global section from the arc is to associate a local section of the perverse schober to each such arc segment, such that these local sections are compatible on their overlaps, and thus glue to a global section\footnote{More precisely, we will formally implement the gluing of local sections by gluing the corresponding lax sections of $\hF$. Such lax sections are an $(\infty,2)$-categorical notion and the stable $\infty$-category of lax sections of $\hF\colon \on{Exit}(\Sgh)\to \on{St}$, see \Cref{def:sections}, describes the lax limit of $\hF$. For $x\in \on{Exit}(\Sgh)$ a vertex or edge of the S-graph, a local section at $x$ refers to an object in the stalk $\mathcal{F}(x)$ of the perverse schober $\mathcal{F}$ at $x$. To a local section, we universally associate a lax section via a relative Kan extension. Their gluing can then conveniently be formulated as taking a colimit in the $\infty$-category of lax sections.}. Since we are considering an arbitrary perverse schober, we need to assume that there exists a recipe that allows us to associate suitable local sections with graded arc segments.

We introduce in \Cref{sec:arcsystemkit} the notion of an \textit{arc system kit} for an $\Sgh$-parametrized perverse schober, which is a collection of the data necessary to associate the desired local sections with graded arc segments. Using the ansatz explained above, given an arc system kit for $\mathcal{F}$, we can thus associate a global section of $\mathcal{F}$ to each graded arc.

The simplest examples of graded arcs are the edges of the S-graph $\Sgh$ itself, so that each such edge has an associated global section. We denote the collection of the so-obtained global sections by $\A_{\Sgh}$. We compute the derived Homs between the objects in $\A_{\Sgh}$, and relate them with counts of intersections of the edges in the S-graph, see \Cref{prop:inthom}. Under certain local criteria on the arc system kit, referred to as positivity, we find that $\A_{\Sgh}$ forms a simple-minded collection, meaning that 
\begin{itemize}
    \item for $\A_e\in \A_{\Sgh}$ we have 
    \[ 
    \on{Ext}^{-i}(\A_e,\A_e)\simeq \begin{cases} k & \text{ for }i=0\\ 0 & \text{ for }i>0\,,\end{cases}
    \]
    \item and for $\A_e,\A_f\in \A_{\Sgh}$, with $\A_e\not\simeq \A_f$, we have
    \[ \on{Ext}^{-i}(\A_e,\A_f)\simeq 0 \qquad \text{for }i\geq 0\,.\]
\end{itemize}
In this case, $\A_{\Sgh}$ is thus the set of simple objects in a heart of a bounded $t$-structure, defined on the stable subcategory generated by the objects in $\A_{\Sgh}$, denoted $\CS$. Note that this is the motivation for the name S-graph: its edges are in bijection with these simple objects. 

We next relate the tilting theory of the finite hearts with the geometry of the mixed-angulations and their dual S-graphs. Given an edge $e$ of the S-graph $\Sgh$, there are two ways of flipping the edge $e$ to obtain a new S-graph, namely flipping backward and forward. These two operations on S-graphs are inverse to each other. The edges of the flipped S-graph again describe a set of graded arcs, whose corresponding simple-minded collection is the backward or forward tilt of the simple-minded collection $\A_{\Sgh}$ at the simple object associated with $e$. 


We summarize our results on simple-minded collections as follows, see also \Cref{constr:glsec}, \Cref{prop:inthom}, \Cref{thm:tilt=flip}, and \Cref{cor:iso}:

\begin{theorem}\label{introthm:tilt}
    Let $\sow$ be a weighted marked surface with non-empty set of marked points. Let $\mathcal{F}$ be a perverse schober parametrized by an extended S-graph of $\sow$. Suppose that $\mathcal{F}$ is equipped with a positive arc system kit. Then there exists a stable subcategory $\CS$ of the stable $\infty$-category of global sections of $\mathcal{F}$ satisfying the following:
    \begin{enumerate}
    \item[(i)]
     Gluing of sections of $\mathcal{F}$ yields an injective arc-to-object assignment 
        \[ 
        \textit{graded arc }\gamma\textit{ in }\sow~\mapsto~\A_\gamma\in \CS\,.
        \] 
    \item[(ii)] For each S-graph $\Sgh$ of $\sow$, the collection $\A_\Sgh=\{\A_e\}$ of global sections associated with the edges of $\Sgh$ forms a generating simple-minded collection in $\CS$, with corresponding heart $\mathcal{H}_{\Sgh}\subset \CS$. 
    \item[(iii)] If $\Sgh\to \Sgh^\sharp$ is a forward flip of S-graphs at an edge $e$, then $\mathcal{H}_{\Sgh^\sharp}$ is the forward simple tilt of $\mathcal{H}_{\Sgh}$ at $\A_e$, in the sense of \Cref{def:forward_tilt}.
    \end{enumerate}
\end{theorem}

Parts (i) and (iii) in \Cref{introthm:tilt} lead to an embedding of the  exchange graph of S-graphs of $\sow$ into the exchange graph of hearts of $\CS$, see \Cref{cor:iso}. The latter exchange graph can be considered as a kind of skeleton of the space of stability conditions on $\CS$.  

The innovation in \Cref{introthm:tilt} is that it does consider any specific $\Sgh$-parametrized perverse schober $\mathcal{F}$ and corresponding category $\CS$, but instead introduces the minimal data for $\mathcal{F}$ that leads to the correspondence between S-graphs and hearts, namely a positive arc system kit. This makes the theorem applicable to many different perverse schobers and triangulated categories of interest. 

In the proof of part (iii) of \Cref{introthm:tilt}, a central step is the reduction to the case that the S-graph $\Sgh$ is the ribbon graph parameterizing the perverse schober $\mathcal{F}$. To perform this reduction, we show that the perverse schober $\mathcal{F}$ can be ``transported'' along a forward flip $\Sgh\to \Sgh^\sharp$ to a $\Sgh^\sharp$-parametrized perverse schober $\mathcal{F}^\sharp$ together with a positive arc system kit, see \Cref{prop:arcsystemkitflip}, and such that $\CS\simeq \C(\Sgh^\sharp,\mathcal{F}^\sharp)$. The possibility of canonically transporting perverse schobers along changes in the ribbon graph is a central aspect of their theory, see \Cref{lem:contr}, and this would not be possible for arbitrary sheaves of categories.

\subsection{Stability conditions and spaces of quadratic differentials}

To relate the space of stability conditions with a space of quadratic differentials, we follow the basic strategy of \cite{BS,BMQS}, but in the much more general setting of mixed-angulations and perverse schobers with an arc system kit.

The first thing to notice is that each quadratic differential $\varphi$ that is generic (i.e.\ without saddle trajectories) gives rise to a mixed-angulation whose edges consist of one member from each 1-parameter family of generic trajectories. It follows from the geometry of the horizontal foliation that each simple pole of $\varphi$ gives rise to a 1-gon, a zero of order $m$ gives rise to an $m+2$-gon, and an exponential singularity of the form $e^{z^{-n}}f(z)dz^2$ gives rise to $n$ distinct $\infty$-gons. The dual S-graph has as vertices the zeros and simple poles of $\varphi$ and edges the saddle connections which do not cross separating trajectories. Conversely, from a mixed-angulation or S-graph together with a choice of complex number in the upper half-plane for each edge one can construct a complex structure and quadratic differential. Moreover, the crossing of strata with a single saddle trajectory corresponds to a flip of the mixed-angulation.

We then use the correspondence between simple tilting of hearts and flips of S-graphs from \Cref{introthm:tilt}.iii), which allows to embed the skeleton (exchange graph) of the moduli space of quadratic differentials into the skeleton of the space of stability conditions. Using the $\CC$-actions on both spaces, one then extends such an embedding to all quadratic differentials, such that the image consists of connected components. This step involves non-trivial results about the surface topology of quadratic differentials, for which we will mostly refer to \cite[\S~4]{BMQS}. However, we consider a slightly more general setup than loc.~cit.~by allowing simple and double poles, as well as boundary singular points, corresponding to quadratic differentials with exponential singularities. Our main result is then the following (see \Cref{thm:app} in the main text).

\begin{theorem}\label{introthm:stab}
Suppose that $\sow$ is not closed with a single marked point and has a non-empty set of marked points. Let $\CS$ be the stable $\infty$-category generated by the simple-minded collection $\A_{\Sgh}$ associated with the edges of the S-graph $\Sgh$. There is a map
\be\label{eq:app*}
    \FQuad{S}{\sow} \to \Stab(\CS)
\ee
from the space of framed quadratic differentials on $\sow$ to the space of stability conditions, which is a biholomorphism onto a union of connected components.
\end{theorem}

\subsection{Examples}\label{introsec:examples}

In this paper, we consider two classes of examples of stable $\infty$-categories which arise from perverse schobers equipped with arc system kits:
\begin{enumerate}
    \item perverse schobers arising from $\infty$-categories of $\mathcal{D}(k)$-valued local systems on spherical fibrations between products of spheres, where $\mathcal{D}(k)$ is the derived category of chain complexes over the field $k$, in \Cref{subsec:sphericalfibrationstori},
    \item topological Fukaya categories of flat surfaces with infinite area in \Cref{subsec:topFukviaschobers}.
\end{enumerate}
\Cref{introthm:stab} applies to both of these examples. A class of more elaborate examples of perverse schobers with arc system kits, whose global sections describe the derived categories of relative graded Brauer graph algebras, is explored in the sequel paper \cite{CHQ24}.

The stable $\infty$-categories of global sections of the first of the two classes of perverse schobers generalize the derived categories of Ginzburg algebras associated with triangulated marked surfaces and more generally derived categories of relative Ginzburg algebras associated with $n$-angulated surfaces, see \cite{Chr22,Chr21b}, whose spaces of stability conditions have not previously been described. The stability conditions of the second class of examples were already identified in \cite{HKK17}, but we give a proof based on very different methods.

Before describing these two examples in more detail, we highlight the relation of our results to other works on quadratic differentials and stability conditions.

\begin{itemize}
    \item \cite{BS,KQ2}: We recover the case of quadratic differentials without double poles (i.e. with simple zeros and poles of order $>2$ only) as a special case of our first example. In the presence of double poles, our 3CY category differs from the one considered by Bridgeland--Smith (theirs is a deformation). The difference in the spaces of stability conditions is that there appear in our approach no quadratic differentials where a zero has collided with a double pole.
    \item \cite{Ike17}
    We also recover the case with order $N-2$ zeros and higher order poles as a special case of our first example, which generalizes the disk case in \cite{Ike17}.
    \item \cite{HKK17}: We recover the case of infinite area flat surfaces (equivalently: quadratic differentials with at least one pole of order $\geq 2$ or an exponential singularity). The case of finite area flat surfaces is beyond the scope of the tilting theory approach. For the same reason, we do not recover the results of \cite{Hai21}.
    \item \cite{BMQS}: This is possibly another special case for a suitable choice of perverse schober, but we do not consider this example here.
\end{itemize}

\subsubsection{Local systems on spherical fibrations between products of spheres}

Given a Kan complex $X$, we consider the $\infty$-category of $\mathcal{D}(k)$-valued local systems on $X$, which is defined as the $\infty$-categorical functor category $\on{Fun}(X,\mathcal{D}(k))$. A spherical fibration $f\colon X\to Y$ is a Kan fibration between Kan complexes whose fiber is homotopic to the $n$-sphere for some $n\geq 0$. If $f$ is a spherical fibration, the pullback functor $f^*\colon \on{Loc}(Y)\to \on{Loc}(X)$ is a spherical functor, see \cite{KS14,Chr20}, and thus describes the data encoding a perverse schober on the disc with a single singularity. The simplest spherical fibration $f\colon S^n\to \ast$ was shown in \cite{Chr22,Chr21b} to give rise to the perverse schobers describing relative Ginzburg algebras of $n$-angulated surfaces. To extend this construction to a mixed-angulated weighted marked surface $\sow$, we consider a perverse schober $\mathcal{F}_{\Sgh,\NUM}$, which, near each degree $m$ vertex, is described by the spherical functor arising from the spherical fibration
\[
\prod_{i\in\NUM} S^{i-1}\longrightarrow \prod_{i\in\NUM, i
\neq m} S^{i-1}\,.
\]
Here $\NUM\subset \mathbb{N}_{\geq 2}$ is a finite set containing all degrees of the interior singular points  of the given weighted marked surface $\sow$ (which is assumed to have no degree $1$ singular points). We show that $\mathcal{F}_{\Sgh,\NUM}$ admits an arc system kit, so that we obtain a description of the space of stability conditions of a proper subcategory of its stable $\infty$-category of global sections.

As shown by Smith \cite{Smi15}, the finite derived categories of Ginzburg algebras of triangulated surfaces embed fully faithfully into the Fukaya categories of $3$-dimensional exact symplectic manifolds with a Lefschetz fibration to the surface. Relative Ginzburg algebras of $n$-angulated surfaces behave similarly to partially wrapped Fukaya categories of $n$-dimensional exact symplectic manifolds with a Lefschetz fibration, see also \cite{Chr21b}. This is generalized by the behaviour of the $\infty$-category of global sections of $\mathcal{F}_{\Sgh,\NUM}$: it behaves is some ways like a partially wrapped Fukaya category of a $1+\prod_{i\in \NUM}(i-1)$ complex dimensional exact symplectic manifold with a Morse--Bott fibration to the surface $\sow$ with fiber $T^*(\prod_{i\in\NUM} S^{i-1})$.

The perverse schober $\mathcal{F}_{\Sgh,\NUM}$ can furthermore (in most cases) be constructed in such a way, that its $\infty$-category of global sections admits a relative left Calabi--Yau structure in the sense of Brav--Dyckerhoff \cite{BD19} of dimension $1+\prod_{i\in \NUM}(i-1)$, see \Cref{rem:CYstructuresforschobers}.

\subsubsection{Topological Fukaya categories of infinite area flat surfaces}

We can consider any weighted marked surface $\sow$ as a graded marked surface, which has an associated $\ZZ$-graded topological Fukaya category.
This topological Fukaya category was shown in \cite{DK15,HKK17} to arise as the global sections of a constructible cosheaf of dg-categories, defined on a spanning ribbon graph of the real blowup surface ${\bf S}^{\on{blw}}$ of $\sow$ at the singular points. The constructible cosheaf models a parametrized perverse schober $\hF$ (by first passing to derived $\infty$-categories and then to the right adjoint diagram, i.e.\ the dual sheaf of the cosheaf). This perverse schober has no singularities, meaning that the spherical adjunctions describing it at the vertices of $\rgraph$ are all trivial (given by $\mathcal{D}(k)\leftrightarrow 0$).

In \Cref{subsec:topFukviaschobers}, we show that the topological Fukaya category of $\sow$ arises as the global sections of a different, novel perverse schober $\hF_{\on{cut}}$ with non-trivial singularities, which is parametrized by an S-graph of $\sow$. The S-graph of $\sow$ does not itself define a spanning graph of ${\bf S}^{\on{blw}}$. However, we can choose a spanning graph that encloses each interior singular point by a loop, and upon removing these loops one obtains the S-graph. The value of the perverse schober $\hF_{\on{cut}}$ at a vertex formerly incident to a loop is given by the $\infty$-category of sections supported in a neighborhood of the loop.

The perverse schober $\hF_{\on{cut}}$ is shown to admit an arc system kit, which allows us to apply \Cref{introthm:stab} to the topological Fukaya category. This gives an alternative proof of the main result of \cite{HKK17} for flat surfaces of infinite area.

\subsection{Notation}
\begin{itemize}
\item Weighted marked surface $\sow=({\bf S},\M,\W,\wt,\nu)$
    with $\M$ the marked points, $\W$ the singular points, $\wt\colon\W\rightarrow \mathbb{N}_{\geq -1}\cup\{\infty\}$ the weight function, and $\nu$ the line field (\Cref{def:weightedmarkedsurf}).
\item Mixed-angulation $\AS$ of $\sow$ (\Cref{def:mixedangulation}), with dual graph an S-graph $\Sgh=\dAS$ (\Cref{def:sgraph}).
\item Forward flip $\AS^\sharp_\gamma$ of a mixed-angulation $\AS$ at an arc $\gamma\in\AS$ (\Cref{def:flip}).
\item Exchange graph $\EG(\sow)$ and exchange graph of S-graphs $\EG_S(\sow)$ of the weighted marked surface~$\sow$ (\Cref{subsec:WMS}).
\item Moduli space of framed quadratic differentials $\FQuad{}{\sow}$ (\Cref{def:fquad}).
\item The category $\Gamma(\rgraph,\hF)$ of global sections of the perverse schober $\hF$ parametrized by a ribbon graph $\rgraph$ (\Cref{def:sections}).
\item Extended ribbon graph $\eS$ of an S-graph $\Sgh$ (\Cref{def:extdgraph}).
\item Collection of objects $\ASG\subset \Gamma(\eS,\hF)$ corresponding to the edges of the S-graph (\Cref{sec:objfromarcs}), generating the stable subcategory $\CS$ with a canonical heart $\CSh$ (\Cref{subsec:flipsrevisited}).
\end{itemize}

Note that our convention of forward flip (moving endpoints clockwise) and morphism direction (counter-clockwise) are the inverse of the ones in \cite{Qiu16,KQ2,BMQS}.

Throughout this paper, $k$ denotes a fixed field.

\subsection{Acknowledgements}
We thank Nathan Broomhead, Raquel Coelho Simoes, David Pauksztello and Jon Woolf for pointing out a small mistake concerning a missing finiteness condition of the tilted heart in \Cref{pp:tilting} in an earlier version of the article and sharing with us their upcoming work~\cite{BCSPW}. We thank the anonymous referee for their careful reading and the suggestions to add more details.
M.C.~thanks Tobias Dyckerhoff for inspiring discussions about stability conditions on topological Fukaya categories of surfaces.
M.C.~further thanks the Hausdorff Research Institute for the hospitality during his stay, during which part of this work was written.

\section{Mixed-angulations on weighted marked surfaces}

In this section we define weighted marked surfaces, their mixed-angulations, the dual notion of an S-graph, as well as flips of both. We also discuss basic properties of quadratic differentials and their moduli spaces, and how horizontal foliations give rise to mixed-angulations.

\subsection{Grading of surfaces and curves}\label{subsec_grading}

We recall the definitions of grading for surfaces and curves on them, following~\cite{HKK17}.
This is the lowest-dimensional case of the more general notions of grading for symplectic manifolds and their Lagrangian submanifolds~\cite{seidel_grading}.

\begin{definition}
    Let $\bf S$ be a smooth, oriented surface. A \emph{grading structure} on $\bf S$ is a foliation $\nu$, i.e.\ a section of the projectivized tangent bundle $\mathbb P(T\bf S)$.
    A \emph{graded surface} is a surface together with a choice of grading structure.
    A \emph{morphism of graded surfaces} $({\bf S}_1,\nu_1)\to ({\bf S}_2,\nu_2)$ is a pair $(f,h)$ where $f\colon S_1\to S_2$ is an orientation preserving local diffeomorphism and $h$ is a homotopy class of paths from $f^*\nu_2$ to $\nu_1$ in the space of sections $\Gamma({\bf S}_1;\mathbb P(T{\bf S}_1))$.
\end{definition}

For us, grading structures arise as horizontal foliations of quadratic differentials.
This foliation, $\mathrm{hor}(\varphi)$, is defined on the complement of the zeros and poles of a quadratic differential $\varphi$ on a Riemann surface $\bf S$ by
\[
\mathrm{hor}(\varphi)(p)\coloneqq\{v\in T_p{\bf S}\mid \varphi(v,v)\in\mathbb R_{\geq 0}\}\qquad \in\mathbb P( T_p {\bf S}).
\]
If $({\bf S},\nu)$ is a graded surface, then any immersed loop $c\colon S^1\to \bf S$ has a \textit{Maslov index}, $\mathrm{ind}_\nu(c)$, which records the number of times (more precisely an intersection number) that $\nu$ becomes tangent to $c$.
The sign is normalized so that if $c$ is a counter-clockwise loop around the origin in $\CC$ and $\nu$ is the horizontal foliation of the quadratic differential $z^kdz^2$, then $\mathrm{ind}_\nu(c)=k+2$.
If $\nu$ is defined in a punctured neighborhood of a point $p\in\bf S$, then define $\mathrm{ind}_\nu(p)\coloneqq \mathrm{ind}_\nu(c)$, where $c$ is a small counter-clockwise loop around $p$.

\begin{definition}\label{def:gradedcurve}
    A \emph{graded curve} in a graded surface $(\bf S,\nu)$ is an immersed curve $c\colon I\to \bf S$ together with a homotopy class of paths, $\tilde{c}$, in $\Gamma(I;c^*\mathbb P(T\bf S))$ from $c^*\nu$ to $\dot{c}$ (i.e.\ from the line field given by the foliation to the line field given by the tangents to $c$).
\end{definition}

We note that if $I$ is an interval, then there is a $\mathbb Z$-torsor of choices of $\tilde{c}$ for given $c$.
If $I=S^1$, then $\tilde{c}$ exists if and only if $\mathrm{ind}_\nu(c)=0$.
To a pair of graded curves $(I_1,c_1,\tilde{c}_1)$, $(I_2,c_2,\tilde{c}_2)$ intersecting transversely in a point $p=c_1(t_1)=c_2(t_2)$ one can assign an integer $i_p(c_1,c_2)$ as the class in $\pi_1(\mathbb P(T_p{\bf S}))=\mathbb Z$ given by $\tilde{c}_1(t_1)\cdot\kappa\cdot\tilde{c}_2(t_2)^{-1}$, where $\kappa$ going from $\dot{c}_1(t_1)$ to $\dot{c}_2(t_2)$ is given by counterclockwise rotation in $T_p\bf S$ by an angle less than $\pi$.

We will also require a definition of the intersection index $i_p(c_1,c_2)$ in the case where two graded curves $(I_1,c_1,\tilde{c}_1)$, $(I_2,c_2,\tilde{c}_2)$ intersect at a common endpoint $p\in\bf S$ at which $\nu$ is undefined, but rather defined on a small punctured disk centered at $p$.
To do this, choose a small arc $\alpha\colon[0,1]\to \bf S$ starting on $c_1$ and ending on $c_2$ and going counterclockwise around $p$ and choose an arbitrary grading of $\alpha$. Then $i_p(c_1,c_2)\coloneqq i_{\alpha(0)}(c_1,\alpha)-i_{\alpha(1)}(c_2,\alpha)$ is independent of the choice of the graded curve $\alpha$, c.f.\ \cite[Equation (3.17)]{HKK17}.

\subsection{Weighted marked surfaces and mixed-angulations}\label{subsec:WMS}

A mixed-angulation is a collection of polygons with possibly varying numbers of edges and an identification of certain pairs of edges.
We will consider sets of mixed-angulations, related by a sequence of flips, on the same underlying surface. To set this up consistently, we need a notion of marked surface.

\begin{definition}\label{def:weightedmarkedsurf}
A \emph{weighted marked surface} is a quintuple $\sow=({\bf S},\M,\W,\wt,\nu)$  where
\begin{itemize}
\item $\bf S$ is a compact oriented surface, possibly with boundary,
\item $\M\subset {\bf S}$ is a non-empty subset of \textit{marked points} (vertices),
\item $\W\subset {\bf S}$ is a finite subset of \textit{singular points} (centers of polygons),
\item $\wt\colon \W\to \ZZ_{\geq -1}\cup \{\infty\}$ is the \textit{weight function},
\item $\nu$ is a grading structure on ${\bf S}\setminus (\M\cup \W)$.
\end{itemize}
such that
\begin{enumerate}
\item $\wt(x)=\infty$ if and only if $x\in\partial \bf S$,
\item $\M\cap \W=\emptyset$,
\item the interior of $\bf S$ contains only finitely many points of $\M$,
\item each component of $\partial \bf S$ contains at least one point of $\M$,
\item $\M$ is discrete in ${\bf S}\setminus \W$, and any non-compact component of $\partial {\bf S}\setminus \W$ contains countably infinitely many points of $\M$,
\item the index of any $x\in\W\cap \mathrm{int}(\bf S)$ with respect to $\nu$, $\mathrm{ind}_\nu(x)$, is equal to $\wt(x)+2$, and in particular finite.
\end{enumerate}
For simplicity, we also assume that $\bf S$ is connected.
An \emph{isomorphism} of weighted marked surfaces $\sow\to\sow'$ is an isomorphism of graded surfaces $(f,h)\colon ({\bf S},\nu)\to ({\bf S}',\nu')$ with $f(\M)=\M'$, $f(\Delta)=\Delta'$, and $\wt'\circ f=\wt$.
\end{definition}
We define $d(x)\coloneqq \wt(x)+2$ and refer to this as the \textit{degree} of $x\in\W$.
\begin{definition}\label{def:arc}
A \emph{compact arc} in a weighted marked surface $\sow$ is a curve $\gamma\colon [0,1]\to\surf$ connecting points in $\W$, i.e.\ $\gamma(\partial I)\subset\W$, that is disjoint from $\W\cup \M$ away from its endpoints.
We require that $\gamma$ be embedded except possibly at the endpoints, which may coincide.
If the endpoints of $\gamma$ coincide, we require $\gamma$ to not be homotopic in ${\bf S}\setminus\M$, relative to its endpoints, to the constant loop.
A \emph{non-compact arc} is a curve that connects points in $\M$ and is away from its endpoints disjoint from $\W\cup \M$.
A \emph{boundary arc} is a non-compact arc which cuts out a connected component of ${\bf S}$ containing one boundary singular point and no other singular points.
A \emph{closed curve} is a curve $\gamma\colon S^1\to\surf^\circ\setminus(\M\cup \W)$.
\end{definition}

\begin{definition}\label{def:mixedangulation}
    A \emph{mixed-angulation} of a weighted marked surface $\sow$ is a (finite) collection, $\mathbb A$, of non-compact graded arcs in ${\bf S}\setminus\W$ such that
    \begin{enumerate}
        \item endpoints of arcs lie in $\M$,
        \item arcs intersect only in endpoints,
        \item the set of arcs cuts $\bf S$ into polygons with vertices in $\M$, called $\AS$-polygons, and edges which are arcs in $\mathbb A$ and/or boundary arcs, so that each $n$-gon contains exactly one singular point $x\in\W$ with $d(x)=n$,
        \item if $X$,$Y$ are two consecutive edges of a polygon in clockwise order, then $i(X,Y)=0$ (for the edges of S-graphs, see below, the order reverses to counter-clockwise).  Here we use the grading on arcs and have chosen a grading on boundary arcs.
    \end{enumerate}
    We refer to the arcs in $\mathbb A$ as \textit{edges} of the mixed-angulation. The boundary arcs are the \textit{boundary edges} and the others are the \textit{internal edges}.
\end{definition}

The notion of a mixed-angulated surface is dual to that of an \textit{S-graph} in the sense of~\cite{HKK17}.

\begin{definition}
A \emph{graph} $\rgraph$ consists of two finite sets $\rgraph_0$ of \emph{vertices} and $\on{H}$ of \emph{halfedges} together with an involution $\tau\colon \on{H}\rightarrow \on{H}$ and a map $\sigma\colon \on{H}\rightarrow \rgraph_0$. We call the $\tau$-orbits the \emph{edges} of $\rgraph$ and denote their set by $\rgraph_1$. An edge is called \emph{internal} if the orbit contains two elements and is called \emph{external} if the orbit contains a single element. An internal edge is called a \emph{loop} at $v\in \rgraph_0$ if it consists of two halfedges both being mapped under $\sigma$ to $v$. The set of internal edges of $\rgraph$ is denoted by $\rgraph_1^\circ$ and the set of external edges is denoted by $\rgraph_1^\partial$.
\end{definition}

\begin{definition}\label{def:sgraph}
    An \emph{S-graph} is a graph $\Sgh$ with, for each vertex $v$, either a cyclic or total order on the set of halfedges meeting $v$ and a positive integer $d(h_1,h_2)$ for each pair $h_1,h_2$ of successive halfedges.
    We call those vertices of $\Sgh$ with cyclic order (resp.\ total order) on the set of halfedges meeting it \textit{internal vertices} (resp.\ \textit{boundary vertices}).
\end{definition}

Given the data of $d(h_1,h_2)$ for consecutive halfedges, we define $d(h_1,h_2)$ for other pairs $h_1$, $h_2$ of halfedges attached to the same vertex $v$ so that $d(h_1,h_2)=0$ if $h_1=h_2$ and $d(h_1,h_3)=d(h_1,h_2)+d(h_2,h_3)$ if the sequence $(h_1,h_2,h_3)$ is compatible with the cyclic/total order at $v$.
If $v$ is an internal vertex, then this determines $d(h_1,h_2)$ for all pairs $(h_1,h_2)$, and if $v$ is a boundary vertex then this determines $d(h_1,h_2)$ for all pairs with $h_1\leq h_2$.

\begin{definition}
    The \emph{dual graph} $\Sgh=\dAS$ of a mixed-angulated surface $(\sow,\mathbb A)$ is the S-graph obtained by choosing embedded intervals (edges of the graph) in $\bf S$ with endpoints in $\W$, intersecting only in endpoints, so that each internal edge of the mixed-angulation corresponds to a unique edge of the S-graph in the sense that the two intersect in exactly one point in the interior.
    The cyclic/total order on the set of halfedges meeting a given vertex is induced from the embedding of the graph in $\bf S$. (By convention, we go around a vertex counter-clockwise.)
    Finally, for each pair of consecutive halfedges $(h_1,h_2)$, we let $d(h_1,h_2)-1$ be the number of boundary edges between the interior edges corresponding to $h_1$ and $h_2$.
\end{definition}

The edges of the dual S-graph have a canonical grading such that $i(X,Y)=0$ if $X\in\mathbb A$ and $Y$ is the dual edge of the S-graph. Then $d(h_1,h_2)$ is just the intersection index of the two corresponding edges of the S-graph at the given vertex.

\begin{figure}[ht]\centering
\makebox[\textwidth][c]{
\begin{tikzpicture}[scale=.6]
\draw[very thick](0,0) circle (5);
\coordinate (z0) at (0,2);
\coordinate (z1) at (0,-2);
\coordinate (z2) at (2,-2);
\coordinate (z3) at (2,2);
\coordinate (z4) at (-2.5,0);
\coordinate (z5) at (-5,0);
\coordinate (z6) at (-45:5);
\coordinate (z7) at (45:5);
\coordinate (z8) at (100:5);
\coordinate (z9) at (140:5);
\coordinate (z10) at (-100:5);
\coordinate (z11) at (-22:5);
\coordinate (z12) at (-10:5);
\coordinate (z13) at (-5:5);
\coordinate (z14) at (-2.5:5);
\coordinate (z15) at (2.5:5);
\coordinate (z16) at (5:5);
\coordinate (z17) at (10:5);
\coordinate (z18) at (22:5);
\draw[blue,thick] (z0) -- (z1) -- (z2) -- (z3) -- (z0) -- (z4) -- (z1);
\draw[blue,thick] (z4) -- (z5);
\draw[blue,thick] (z2) -- (z6);
\draw[blue,thick] (z3) -- (z7);
\draw[blue,thick] (z0) -- (z8);
\foreach \j in {0,...,18}{\draw(z\j)\nn;}
\coordinate (v0) at (1,0);
\coordinate (v1) at (-1,0);
\coordinate (v2) at (-100:4);
\coordinate (v3) at (5,0);
\coordinate (v4) at (1,3);
\coordinate (v5) at (140:3.5);
\draw[red] (v0) -- (v1);
\draw[red] (v1) edge[bend right=25] (v2);
\draw[red] (v2) edge[bend right=30] (v3);
\draw[red] (v3) -- (v0) -- (v4);
\draw[red] (v4) edge[bend right=25] (v5);
\draw[red] (v5) -- (v1);
\draw[red] (v0) edge[bend left=15] (v2);
\draw[red] (v3) edge[bend right=20] (v4);
\draw[red] (v2) edge[bend left=45] (v5);
\draw[red](v0)\ww node[above left]{$\scriptstyle{1}$} node[above right]{$\scriptstyle{1}$} node[below left]{$\scriptstyle{1}$} node[below right]{$\scriptstyle{1}$};
\draw[red](v1)\ww node[left]{$\scriptstyle{1}$} node[above right]{$\scriptstyle{1}$} node[below right]{$\scriptstyle{1}$};
\draw[red](v2)\ww node[above]{$\scriptstyle{1}$} node[above right]{$\scriptstyle{\;1}$} node[above left]{$\scriptstyle{1}$} node[below]{$\scriptstyle{2}$};
\draw[red](v3)\ww node[below left]{$\scriptstyle{1\;}$} node[above left]{$\scriptstyle{1\;}$};
\draw[red](v4)\ww node[above right]{$\scriptstyle{3}$} node[below left]{$\scriptstyle{1}$} node[below right]{$\scriptstyle{1}$};
\draw[red](v5)\ww node[left]{$\scriptstyle{3}$} node[right]{$\scriptstyle{1}$} node[below]{$\scriptstyle{1}$};
\end{tikzpicture}}
\caption{Example of mixed-angulated surface and dual S-graph. Numbers $d(h_1,h_2)$ for consecutive halfedges $h_1$, $h_2$ are shown.}
\label{fig:mixangulation}
\end{figure}

Next, we introduce the flips of mixed-angulations.
We say an arc in a mixed-angulation is a monogon arc if it encloses a monogon.
\begin{definition}\label{def:flip}
The \emph{forward flip} $\AS^\sharp_\gamma$ of a mixed-angulation $\AS$ at an arc $\gamma\in\AS$ is an operation that moves the endpoints of $\gamma$ clockwise along the adjacent sides of the $\AS$-gons containing $\gamma$, as illustrated in Figure~\ref{fig:flip1}. A special case of the flip is the monogon case: if $\gamma$ is a monogon arc, we will move both its endpoints along the adjacent side of the other $\AS$-gon containing $\gamma$, as illustrated in Figure~\ref{fig:flip2}.
Then we obtain a new mixed-angulation $\AS^\sharp$ by
replacing $\gamma$ with $\gamma^\sharp$.
The inverse of a forward flip is a \emph{backward flip}, $\AS^\flat_\gamma$.
Note that $\gamma^\sharp$ inherits the grading from $\gamma$
when grading is considered.

The \emph{exchange graph $\EG(\sow)$ of the weighted marked surface~$\sow$} is the directed graph whose vertices are mixed-angulations and whose oriented edges are forward flips between them.
\end{definition}

We will refer to the two cases of flips, as in Figure~\ref{fig:flip1} and Figure~\ref{fig:flip2}, as the usual and the monogon case, respectively.
\begin{remark}\label{rem:f.f.}
The forward flip operation on an S-graph
\[\Sgh=\dAS\to\Sgh^\sharp_{e}=(\AS^\sharp_\gamma)^*\] 
can concretely be described as follows, where ${e}$ in $\Sgh$ is the dual edge of $\gamma$ in $\AS$.
The arc ${e}$ remains in $\Sgh^\sharp_{e}$ but its grading will be shifted downward by one.
Moreover,
in the usual case, the weights of both endpoints of ${e}$ are bigger than 1;
in the monogon case, the degree of exactly one endpoint of ${e}$ is 1.
For $h\neq{e}\in\Sgh$, we flip it to $h'$ as follows:
\begin{itemize}
  \item In the usual case,
  if there is an oriented intersection $x$ of index 1 from ${e}$ to $h$ at a singular point $Z$,
  then move the endpoint $Z$ of $h$ along ${e}$ with respect to $Z$.
  See \Cref{fig:flip1} for $h=a,b$.
  \item In the monogon case,
  if there is an oriented intersection $x$ of index 1 from ${e}$ to $\alpha$ at a singular point $Z$,
  then apply the square of the braid twist $B_{e}$ along ${e}$ (cf. \cite[Def.~3.10]{KQ2}) as shown in \Cref{fig:flip2} to $h$ and then $h'=B_{e}^2(h)$.
\end{itemize}

The \emph{exchange graph $\EG_S(\sow)$ of S-graphs of $\sow$} is the directed graph whose vertices are S-graphs and whose oriented edges are forward flips between them.
\end{remark}

\begin{figure}[ht]\centering
\makebox[\textwidth][c]{
\begin{tikzpicture}[xscale=-.5,yscale=.5]
\draw[dashed](0,0) circle (5.2)
    (30:1)node{$e$} (142:2)node{$_a$} (0:4)node{$_b$};
\foreach \j in {0,...,7}{\draw (90+45*\j:5) coordinate  (w\j) edge[blue,thick] (45+45*\j:5) coordinate  (u\j);}
\path ($(w0)!.5!(w2)$) coordinate (v1)
      ($(w0)!.5!(w5)$) coordinate (v2);
\draw[blue,thick](w0)to(w2);
\draw[cyan, very thick](w5)to(w0);

\foreach \j in {0,...,7}{\draw[red](22.5+45*\j:5.2) coordinate (z\j);}

\draw[red](90-22.5*3:3)edge(z0)edge(z1)edge(z7);
\draw[red](90-45*5:1)edge(z6)edge(z5)edge(z4);
\draw[red](90+45:4.2)edge(z2)edge(z3);
\draw[red](90-22.5*3:3)\ww node[below]{\small{$4$}} to
    (90-45*5:1)\ww node[above]{\small{$5$}}node[right,black]{$^Z$} to (90+45:4.2)\ww  node[left]{\tiny{$3$}};

\foreach \j in {0,...,7}{\draw[red,fill=white](z\j) circle(.08);}

\draw(-6.5,0)node{\Huge{$\rightsquigarrow$}};
\begin{scope}[shift={(-13,0)}]
\draw[dashed](0,0) circle (5.2)
    (30:1)node{$e$} (-40:2.5)node{$_{b'}$} (90:3.7)node{$_{a'}$};
\foreach \j in {0,...,7}{\draw (90+45*\j:5) coordinate  (w\j) edge[blue,thick] (45+45*\j:5) coordinate  (u\j);}
\path ($(w0)!.5!(w2)$) coordinate (v1)
      ($(w0)!.5!(w5)$) coordinate (v2);
\draw[cyan,very thick]
    (w2).. controls +(38:5) and +(120:4) ..($(w0)!.5!(w5)$)
        .. controls +(-60:2) and +(-125:2) ..(w6);
\draw[blue,thick](w0)to(w2);

\foreach \j in {0,...,7}{\draw[red](22.5+45*\j:5.2) coordinate (q\j);}
\draw[red](90-22.5*3:3)edge(q0)edge(q1);
\draw[red](90-45*5:1)edge(q6)edge(q5)edge(q4)edge(q7);
\draw[red](90+45:4.2)edge(q2)edge(q3);

\draw[red]  (90+45:4.2)\ww  node[below]{\tiny{$3$}} to [bend left=45]
    (90-22.5*3:3)\ww node[below]{\small{$4$}} to [bend left=0]
    (90-45*5:1)\ww node[above]{\small{$5$}}node[right,black]{$^Z$} ;

\foreach \j in {1,...,8}{\draw[red,fill=white](22.5+45*\j:5.2)circle(.08);}
\end{scope}
\end{tikzpicture}}
\caption{The forward flip at a usual arc.}
\label{fig:flip1}
\end{figure}

\begin{figure}[hb]\centering
\makebox[\textwidth][c]{
\begin{tikzpicture}[xscale=-.8,yscale=.8,arrow/.style={->,>=stealth}]
\draw[dashed](0,-.2) circle (4);
  \draw[cyan,very thick](-150:3).. controls +(35:3) and +(0:5) ..(-150:3);
  \draw[red,thick](0,1)to[bend left=45]
    node[left,black,font=\normalsize]{${h}$} (0,-4)\ww  (-3,2)\ww to (0,1)\ww to(3,2)\ww;
  \draw[red,thick,font=\scriptsize](0,-1)\ww node[right]{$1$}
    to node[right,black,font=\normalsize]{${e}$} (0,1)\ww node[above]{$4$};
  \draw[blue,thick](90:3) to[bend left=45] (-30:3)\nn
    to[bend left=45]  (-150:3)\nn to[bend left=45] (90:3)\nn;
\draw(-5,0)node{\Huge{$\rightsquigarrow$}};
\begin{scope}[xscale=-1,shift={(10,0)}]
\draw[dashed](0,-.2) circle (4);
  \draw[cyan,very thick](-150:3).. controls +(35:3) and +(0:5) ..(-150:3);
  \draw[red,thick](0,1)to[bend left=45]
    node[right,black]{$^{B_e^2(h)}$} (0,-4)\ww  (-3,2)\ww to (0,1)\ww to(3,2)\ww;
  \draw[red,thick,font=\scriptsize](0,-1)\ww node[left]{$1$}
    to node[left,black,font=\normalsize]{${e}$}(0,1)\ww node[above]{$4$};
  \draw[blue,thick](90:3) to[bend left=45] (-30:3)\nn to[bend left=45] (-150:3)\nn to[bend left=45] (90:3)\nn;
\end{scope}
\end{tikzpicture}
}
\caption{The forward flip at a monogon arc.}
\label{fig:flip2}
\end{figure}

\subsection{Quadratic differentials}\label{subsec_quad}

Let $C$ be a compact Riemann surface and $\varphi$ a quadratic differential which is holomorphic and non-vanishing away from a finite subset $D\subset C$.
Each point in $D$ should be either a zero, a pole, a regular point, or a transcendental singularity of $\varphi$ of the form $\exp(f(z))g(z)dz^2$ where $f$ and $g$ are meromorphic and $f$ has a pole of order $n$ at $z=0$, called an \textit{exponential singularity of order $n$}.
Thus, $D$ consists of the zeros, poles, and exponential singularities of $\varphi$, as well as possibly some additional marked points where $\varphi$ is holomorphic and non-vanishing.

Besides the complex-analytic point of view on quadratic differentials, there is also a metric point of view.
Indeed, $|\varphi|$ is a flat Riemannian metric on $C\setminus D$, i.e.\ $C\setminus D$ has the structure of a metric space which is locally isometric to the Euclidean plane.
However, $C\setminus D$ is in general not complete as a metric space and the metric completion, $\overline{C\setminus D}$, adds a finite number of \textit{conical points}.
More precisely, $\overline{C\setminus D}$ has a conical point with cone angle $(n+2)\pi$ for each zero of order $n$ of $\varphi$, a conical point of cone angle $\pi$ for each simple pole of $\varphi$, and $n$ infinite-angle conical points for each exponential singularity of order $n$, and a (non-singular) point with cone angle $2\pi$ for each regular point of $\varphi$ in $D$.
The metric is already complete near higher order poles of $\varphi$, so these do not lead to any additional conical points.
Proofs of these claims may be found in~\cite{HKK17}.

A flat surface has a total area which may be finite or infinite.
In fact, the area of $\varphi$ is finite if and only if $\varphi$ has no higher order poles and no exponential singularities.
The finite area case is interesting from the point of view of ergodic theory and has been much studied.
Quadratic differentials $\varphi$ with infinite area, on the other hand, are closely related to mixed-angulations. This was already noted in~\cite{HKK17} in terms of the dual language of S-graphs.
Let us recall this relation.
We need to assume also that $|\varphi|$ has at least one conical point, i.e.\ exclude the flat plane and flat cylinder.
For generic choice of $\theta\in\mathbb R$, the horizontal foliation of $e^{i\theta}\varphi$ then has the structure of a \textit{horizontal strip decomposition}.
This means the following: After possibly replacing $\varphi$ by $e^{i\theta}\varphi$, we can assume that the horizontal foliation of $\varphi$, $\mathrm{hor}(\varphi)$, has no finite-length leaves.
Then there are two types of leaves, or trajectories:
\begin{enumerate}
    \item \emph{Separating trajectories}: Approach a conical point in one direction and escape ``to infinity'' (towards a higher order pole or exponential singularity) in the other direction.
    \item \emph{Generic trajectories}: Avoid the conical points and go to infinity in both directions.
\end{enumerate}
Each separating trajectory belongs to some conical point, by definition, and a conical point with cone angle $n\pi$ gives rise to $n$ separating trajectories.
These separating trajectories cut the surface into \textit{horizontal strips}, foliated by one-parameter families of generic trajectories.
Horizontal strips are isometric to $(0,a)\times\mathbb R$ for some $a\in (0,\infty]$, their \textit{height}.
Finite (resp.\ infinite) height horizontal strips have four (resp.\ two) separating trajectories on their boundary, some of which are possibly identified.
Choosing a generic trajectory from each horizontal strip gives a mixed-angulation whose edges are these horizontal trajectories. Each conical point of cone angle $n\pi$ corresponds to an $n$-gon.
Edges corresponding to infinite-height horizontal strips live on the boundary of the mixed-angulation, i.e.\ belong to just one polygon.
\begin{figure}
    \centering
    \begin{tikzpicture}
        \foreach \n in {0,...,4}
        {
            \coordinate (v) at (360/5*\n:3);
            \coordinate (w) at (360/5*\n+360/5:3);
            \foreach \m in {0,.5,1,1.5,2,2.8}
            {
                \coordinate (u) at (360/5*\n+36:\m);
                \draw[gray] (v) .. controls (u) and (u) .. (w);
            }
            \coordinate (u) at (360/5*\n+36:3.3);
            \draw[gray,dotted] (v) .. controls (u) and (u) .. (w);
            \draw[gray] (0,0) -- (v);
            \draw (v) \nn;
            \draw[blue,thick] (v) -- (w);
            \draw[thick,red] (0,0) -- (360/5*\n+36:3);
            \draw[thick,red,dotted] (360/5*\n+36:3) -- (360/5*\n+36:3.6);
        }
        \draw[red] (0,0) \ww;
    \end{tikzpicture}
    \caption{A 5-gon around a triple zero of a quadratic differential. Parts of horizontal foliation (gray lines), mixed-angulation (blue lines), dual S-graph (red lines) are shown.}
    \label{fig:polygon}
\end{figure}

The dual S-graph of the mixed-angulation is obtained as the union of the finite-length geodesics which connect the conical points on opposite boundaries of the strip (which can be two different or the same conical point). The quadratic differential is then entirely encoded in terms of the S-graph and a positive number $Z(e)$ in the upper half-plane for each edge, which is the vector between the conical points on the boundary of the horizontal strip $e$.

One issue with the above construction of a mixed-angulation from a quadratic differential is that the mixed-angulation lives on the compact surface $C$ which we cannot take to be the weighted marked surface $\sow$ in the case where $\varphi$ has poles of order $>2$ or exponential singularities.
To correct this, we define $\sow(C,D,\varphi)=({\bf S},\M,\W,\wt,\nu)$ associated with $(C,D,\varphi)$ as follows:
\begin{itemize}
    \item $\bf S$ is the oriented real blow-up of $C$ in the poles of order $\geq 3$ and exponential singularities,
    \item $\M\cap\mathrm{int}(\bf S)$ is the set of second order poles,
    \item $\M\cap \partial \bf S$ is the set of directions from which horizontal trajectories approach (there are $n$ such directions for a pole of order $n+2$ and $n$ different $\mathbb Z$-torsors worth for an exponential singularity of order $n$),
    \item $\W$ is the set of conical points (simple poles, marked points, zeros),
    \item $(\wt(x)+2)\pi=d(x)\pi$ is the cone angle of the conical point,
    \item $\nu$ is the horizontal foliation of $\varphi$.
\end{itemize}

We consider certain moduli spaces of quadratic differentials which give rise to a given weighted marked surface.
These are special cases of the moduli spaces $\mathcal M(S)$ considered in~\cite{HKK17}, as here we restrict to flat surfaces of infinite area.

\begin{definition}\label{def:fquad}
Let $\sow$ be a weighted marked surface.
Define the \emph{moduli space} $\FQuad{}{\sow}$ as follows.
A point in $\FQuad{}{\sow}$ is represented by a triple $(C,D,\varphi)$ as above together with an isomorphism of weighted marked surfaces $f\colon \sow\to\sow(C,D,\varphi)$.
Two points $(C,D,\varphi,f)$ and $(C',D',\varphi,f')$ are equivalent if there is an isomorphism of weighted marked surfaces $g\colon \sow(C,D,\varphi)\to\sow(C',D',\varphi')$ so that $(f')^{-1}gf$ is isotopic to the identity (as an isomorphism of weighted marked surfaces).
We give $\FQuad{}{\sow}$ the topology of pointwise convergence.
\end{definition}

There is a well-known affine structure on $\FQuad{}{\sow}$ given by the \emph{period map}. This is defined as follows.
First, let $\Sigma\to {\bf S}\setminus (\M\cup \Delta)$ be the double cover whose fiber over $p$ is the two-element set of orientations of the line $\nu(p)\subset T_p\bf S$.
(If $\nu$ is the horizontal foliation of a quadratic differential $\varphi$, then this is just the covering where $\sqrt{\varphi}$ becomes well-defined.)
Let
\[
\Hone(\sow)\coloneqq H_1\left({\bf S}\setminus \M,\W;\mathbb Z\times_{\mathbb Z/2}\Sigma\right)
\]
where $\mathbb Z\times_{\mathbb Z/2}\Sigma$ is the local system of abelian groups whose fiber over $p$ is the abelian group generated by $\Sigma_p=\{o_1,o_2\}$ with relation $o_1=-o_2$, thus non-canonically isomorphic to $\mathbb Z$.
Note that an isomorphism of weighted marked surfaces $\sow\to \sow'$ induces an isomorphism of abelian groups $\Hone(\sow)\to \Hone(\sow')$.
Given a triple $(C,D,\phi)$ there is a map
\be\label{eq:periodmap}
    \int\colon\Hone(\sow(C,D,\phi))\to \mathbb C,\qquad \gamma\mapsto \int_{\gamma}\sqrt{\phi}
\ee
so if $(C,D,\phi,f)\in\FQuad{}{\sow}$ then we get a map $\Hone(\sow)\to\mathbb C$.
These combine to give a continuous map
\begin{gather}\label{eq:period}
    \Pi\colon\FQuad{}{\sow}\longrightarrow\mathrm{Hom}\left(\Hone(\sow),\mathbb C\right)
\end{gather}
which turns out to be a local homeomorphism (see~\cite[Thm.~2.1]{HKK17}, which also includes the case of exponential singularities).
The above map may be used to define an affine structure on $\FQuad{}{\sow}$ by pulling back the one on $\mathrm{Hom}\left(\Hone(\sow),\mathbb C\right)\cong \mathbb C^k$.

\section{Hearts and tilting}\label{sec:heartstilting}

This section contains definitions and preliminaries on hearts of t-structures and their tilting theory.
As a slight generalization of existing results, we compute the tilts of finite hearts with respect to simple objects which are not necessarily rigid.
They reader may also skip this section for now and return it while reading \Cref{subsec:flipsrevisited}.

\subsection{Preliminaries}\label{app:preliminaries}
We first recall some basic definitions around subcategories of additive/abelian/triangulated categories.

In an additive category $\C$ with a subcategory $\hua{B}$,
we denote by
\[
    \hua{B}^{\perp_{\C}}\coloneqq \left\{C\in\C:\Hom_{\C}(B,C)=0\ \forall B\in \hua{B}\right\}.
\]
We may write $\hua{B}^{\perp}$ when $\C$ is implied. The subcategory ${}^{\perp_{\C}}\hua{B}$ is defined similarly.

For subcategories $\hT,\hF$ of an abelian category $\h$,
denote by
$$
    \hT*\hF\coloneqq \{M\in\h\mid\text{$\exists$ s.e.s.
        $0\to T\to M\to F\to0$ s.t. $T\in\hT, F\in\hF$} \}.
$$
If additionally $\Hom(\hT,\hF)=0$, then we write $\hT\perp\hF$ for $\hT*\hF$.
For a subcategory $\C$ of $\h$,
denote by $\<\C\>$ the full subcategory of $\h$ consisting of objects that admit a filtration with factors in $\C$.

A map $f\colon E\to F$ in an arbitrary additive category is called \emph{right minimal} if
it does not have a direct summand of the form $E'\to0$.
Similarly for \emph{left minimal}.

Let $\C$ be an additive category with a full subcategory $\hX$.
A \emph{right $\hX$-approximation} of an object $C$ in $\C$ is a morphism $f\colon X\to C$ with $X\in\hX$,
such that $\Hom(X',f)$ is an epimorphism for any $X'\in \hX$.
Dually, a \emph{left $\hX$-approximation} of an object $C$ in $\C$ is a morphism $g\colon C\to X$ with $X\in\hX$,
such that $\Hom(g,X')$ is an epimorphism for any $X'\in\hX$.
A \emph{minimal right/left $\hX$-approximation} is a right $\hX$-approximation that is also right/left minimal.
We say $\hX$ is \emph{contravariantly finite} (resp.\ \emph{covariantly finite}) if every object in $\C$ admits a right (resp.\ left) $\hX$-approximation.
It is \emph{functorially finite} if it is both contravariantly finite and covariantly finite.

A \emph{torsion pair} in an abelian category $\h$ consists of two additive subcategories $\hT$ and $\hF$ such that $\h=\hT\perp\hF$.
Denote by $\Sim\h$ the set of simples of $\h$. An abelian category $\h$ is called a \emph{length category} if any object in $\h$ admits a finite filtration with factors in the set of simples $\Sim\h$. The abelian category $\h$ is called $\emph{finite}$ if it is a length category with a finite number of simples.

Similarly, we have the corresponding notions and notation in a triangulated category $\D$, replacing s.e.s.\ by triangles.
A \emph{t-structure} $\hP$ of $\D$ is the torsion part of some torsion pair of $\D$ which is closed under shift (i.e.\ $\hP[1]\subset\hP$).
We will write it as $\D=\hP\perp\hP^{\perp}$.
Such a t-structure is \emph{bounded} if for any object $M$ in $\D$ we have $M[\gg0]\in\hP$ and $M[\ll0]\in\hP^\perp$.

The \emph{heart} of a t-structure $\hP$ is $\h=\hP^\perp[1]\cap\hP$. An additive subcategory $\h$ of a triangulated category $\D$ is the heart of a bounded $t$-structure of $\D$ if and only if the following two conditions hold, see \cite[Lem.~3.2]{B1}.
\begin{itemize}
    \item $\Hom(\h[a],\h[b])=0$ for any $a>b$ in $\ZZ$.
    \item Any object $E$ in $\D$ admits a finite filtration whose associated graded has in each degree an object of the form $H[k]$ with $H\in\h$ and $k\in \mathbb{Z}$.
\end{itemize}
The additive subcategories of a triangulated category $\D$ satisfying the above two conditions will also be referred to as the hearts of $\D$. 

There is a natural partial order on hearts corresponding to the reverse inclusion relation on the associated t-structures, i.e. $\h_1\le\h_2 \Leftrightarrow \hP_1\supset\hP_2$.

\subsection{Tilting theory}\label{subsec:tilting}
In this section we discuss the tilting of hearts, and in particular the tilting with respect to simple objects which are not necessarily rigid.

Recall simple {Happel--Reiten--Smal\o} tilting:
\begin{definition}\cite[Def.~3.7]{KQ}\label{def:forward_tilt}
Let $\h=\hT\perp\hF$ be a torsion pair in a heart $\h$ of a triangulated category $\D$.
Then the \emph{forward} (resp.\ \emph{backward}) \emph{tilt} $\h^\sharp$ (resp.\ $\h^\flat$) w.r.t.\ $\hT\perp\hF$ is the heart that admits
a torsion pair $\h^\sharp=\hF[1]\perp\hT$ (resp.\ $\h^\flat=\hF\perp\hT[-1]$).
The forward (resp.\ backward) tilting is \emph{simple} if $\hF$ (resp.\ $\hT$) equals $\langle S\rangle$
for a simple object $S\in\Sim\h$.
Denote by $\h^\sharp_S$ (resp.\ $\h^\flat_S$) the simple forward (resp.\ backward) tilt of $\h$ w.r.t.\ a simple $S$.
\end{definition}

\def\top{\operatorname{top}}
\def\Sub{\operatorname{sub}}
\def\soc{\operatorname{soc}}
A simple $S$ is \emph{rigid} if $\Ext^1(S,S)=0$.
When $S\in\Sim\h$ is rigid,
the simple tilting with respect to $S$ is particularly easy to determine, see \Cref{rem:tilting}.
Next we consider the simple tilting with respect to a possibly non-rigid simple object. Such simple tiltings were also described in \cite[Def.~7.5, Prop.~7.10]{KY14}.

\begin{proposition}\label{pp:tilting}
Let $S$ be a simple in a heart $\h\in\EG(\D)$, which is a length category.
Then $\<S\>$ is functorially finite in $\h$.
If furthermore both $\h$ and $\h^\sharp_S$ are finite, then
\begin{gather}
  \Sim\h^\sharp_S =
  \{S[1]\}\;\cup\;\{ \psi^\sharp_{S}(X)\mid X\in \Sim\h,X\neq S\},
\label{eq:formula3}
\end{gather}
where $\psi^\sharp_{S}(X)=\Cone(f)[-1]$ for $f$ the left minimal $\<S[1]\>$-approximation of $X$.
Similarly, if both $\h$ and $\h^\flat_S$ are finite, then
\begin{gather}
  \Sim\h^\flat_S =
  \{S[-1]\}\;\cup\;\{ \psi^\flat_{S}(X)\mid X\in \Sim\h,X\neq S\},
\label{eq:formula4}
\end{gather}
where $\psi^\flat_{S}(X)=\Cone(g)$ for $g$ the right minimal $\<S[-1]\>$-approximation of $X$.
\end{proposition}

The statement that $\<S\>$ is functorially finite can also be found in \cite[Thm.~3.3]{D15}.

\begin{proof}[Proof of \Cref{pp:tilting}]
Given $M\in \h$, we define
\begin{gather}\label{eq:free}
  \top_S(M)\coloneqq S\otimes\Hom(M,S)^*\quad\text{and}\quad
    \Sub_S(M)\coloneqq \ker( M\xrightarrow{\text{ev}} \top_S(M) )
\end{gather}
and inductively define
\[\begin{cases}
    \Sub^m_S(M)=\Sub_S\big(\Sub^{m-1}_S(M)\big) \\
    \top^m_S(M)=\top_S\big(\Sub^{m-1}_S(M)\big)
\end{cases}\]
using \eqref{eq:free} and set $\Sub^0_S(M)=M$. For any $M\neq0$ in $\h$,
there exists $j\in\ZZ_{>0}$ with $\top^1_S(M)\ne0$ and $\top^{j+1}_S(M)=0$,
since $\h$ is a length category.
Thus, we obtain a filtration of short exact sequences
\begin{equation} \label{eq:filt}
\begin{tikzcd}[column sep=.5pc]
  \Sub_S^{j}(M) \arrow[rr] && \Sub_S^{j-1}(M) \arrow[dl] \arrow[rr] &&  \cdots \arrow[rr] && \Sub_S^1(M)
        \arrow[rr] && \Sub^0_S(M)=M. \arrow[dl] \\
  & \top^{j}_S(M) \arrow[ul, dashed]  && && && \top^1_S(M) \arrow[ul, dashed]
\end{tikzcd}
\end{equation}
In particular, we have a short exact sequence
\begin{equation} \label{eq:SES}
\begin{tikzcd}[column sep=1.2pc]
    0\arrow[r]& \Sub^j_S(M) \arrow[r]& M \arrow[r,"f_0"]& \top^\bullet_S(M)\arrow[r]& 0,
\end{tikzcd}
\end{equation}
where $\top^\bullet_S(M)$ admits a filtration with factors $\top^i_S(M)$ in $\<S\>$ for $1\le i\le j$.
Thus, $\top^\bullet_S(M)$ is also in $\<S\>$.
As $\top^{j+1}_S(M)=\top_S \big(\Sub^j_S(M)\big)=0$, we have $\Sub^j_S(M)\in {}^{\perp_{\h}} S$.
This implies there is a torsion pair $\h={}^{\perp_{\h}} S\perp \<S\>$.
By applying $\Hom(-,S)$ to \eqref{eq:SES}, we see that $f_0$ is the left $\<S\>$-approximation of $M$.
Hence, $\<S\>$ is covariantly finite.
Dually, there is a torsion pair $\h=\<S\>\perp S^{\perp_{\h}}$ and $\<S\>$ is contravariantly finite.
So $\<S\>$ is functorially finite.

Next, suppose that $\h_S^\sharp$ is also finite.
Then, $\<S[1]\>$ is functorially finite in $\h_S^\sharp$.
Let $X\not\simeq S$ be any other simple in $\h$, so that $X\in {}^{\perp_{\h}} S\subset\h_S^\sharp$.
Let $T$ be a simple subobject of $X$ in $\h_S^\sharp$.
As $X$ is simple in $\h$ and $T$ is simple in $\h_S^\sharp$,
there are short exact sequences
\begin{equation}\label{eq:ses2}
\begin{tikzcd}[column sep=1pc]
    0 \arrow[r]& M \arrow[r]& T \arrow[r,"h"]& X \arrow[r]& 0
    &\text{and}&
    0 \arrow[r]& T \arrow[r,"h"]& X \arrow[r,"f"]& M' \arrow[r]& 0
\end{tikzcd}
\end{equation}
in $\h$ and $\h_S^\sharp$ respectively.
Thus $M'=M[1]$.
As $T$ is a simple in $\h^\sharp_S$,
we have $\Hom(T,S[1])=0$.
Applying $\Hom(-,S[1])$ to the second short exact sequence in \eqref{eq:ses2}, we see that $f$ is the left $\<S[1]\>$-approximation of $X$. It must be minimal, since $f$ is an epimorphism in $\h$. Thus, $T=\psi^\sharp_{S}(X)$ as expected. To show that all simples in $\h_S^\sharp$ are as in \eqref{eq:formula3}, take any simple $T$ in $\h_S^\sharp$. Then $T\in{}^{\perp_{\h}} S\subset\h$ and we find a simple top $X$ of $T$ in $\h$ and as above. We obtain exact sequences as in \eqref{eq:ses2}, which again leads to $T=\psi^\sharp_{S}(X)$, concluding the argument.
\end{proof}

\begin{remark}\label{rem:tilting}
The proposition above is a slight generalization of \cite[Prop.~5.2]{KQ}.
More precisely,
when the simple $S$ is rigid, new simples in the forward simple tilting formulae \eqref{eq:formula3} and \eqref{eq:formula4} are given by the simplified formulae
\begin{eqnarray}
\label{eq:psi+}
  \tilt{\psi}{\sharp}{S}(X)
&=&\Cone\left(X\to S[1]\otimes\Ext^1(X, S)^* \right) [-1],
\\ \label{eq:psi-}
  \tilt{\psi}{\flat}{S}(X)
&=&\Cone \left(S[-1]  \otimes  \Ext^1(S, X)\to X \right).
\end{eqnarray}
\end{remark}

\begin{example}\label{ex:s.tilting}Consider the case where
\begin{itemize}
  \item $\Ext^1(S,S)=k$, with the corresponding non-split extension denoted by $\hat{S}$.
  \item $\Ext^1(\hat{S},S)=0$.
  \item $X\in\Sim\h$ with $\Ext^1(X,S)=k$.
\end{itemize}
Inspecting the proof of \Cref{pp:tilting}, one finds that $\psi^\sharp_{S}(X)=\Cone(X\xrightarrow{f}\hat{S}[1])$.
More precisely, let $\Ext^1(X,S)=k$ correspond to the object $\hat{X}$.
Then we have $\Ext^1(\hat{X},S)=k$, which corresponds to
$\psi^\sharp_{S}(X)$. Equivalently, we have the following tower of short exact sequences in $\h$ (which arises from \eqref{eq:filt} with $M=X$ in $\h^\sharp_{S}$):
\[
\begin{tikzcd}[row sep=.7pc,column sep=1.2pc]
    &&{\psi^\sharp_{S}(X)}\arrow[rd]\\
    & \hat{S}\arrow[ru]\arrow[rd]&& \hat{X}\arrow[rd]\\
    S\arrow[ru]&&S\arrow[ru]&&X.\\
\end{tikzcd}
\]
\end{example}

\section{Perverse schobers parametrized by ribbon graphs}\label{sec:backgroundschobers}

In this section, we recall a notion of categorified perverse sheaves, called parametrized perverse schobers, on surfaces with boundary. Such parametrized perverse schobers amount to certain constructible sheaves with values in stable $\infty$-categories, on ribbon graphs embedded in the surfaces. The exposition follows \cite{Chr22}, except that we refine the discussion of  perverse schobers on the $n$-spider by introducing novel local conditions in \Cref{def:schberngon}.

\subsection{Preliminaries on higher category theory}\label{sec:highercatthy}

We formulate \Cref{sec:backgroundschobers,sec:kits,sec:examples} in the language of Lurie's stable $\infty$-categories. In these sections, we will generally follow the notation of \cite{HTT,HA}. Given a stable $\infty$-category, its homotopy $1$-category inherits a triangulated structure. Because of their good formal properties, stable $\infty$-categories thus provide us with a convenient choice of enhancement for triangulated categories. Even better formal properties are possessed by presentable, stable $\infty$-categories. We will thus often work with the latter. Note that this is not much of a loss of generality: any stable $\infty$-category $\hC$ embeds fully faithfully into its $\on{Ind}$-completion $\on{Ind}(\hC)$, which is a presentable, stable $\infty$-category. Further, if $\hC$ is idempotent complete, then the $\infty$-category of compact objects $\on{Ind}(\hC)^c$ in $\on{Ind}(\hC)$ coincides with $\C$.

We will make use of the following $\infty$-categories of $\infty$-categories:

\begin{itemize}
\item The $\infty$-category of (not necessarily small) $\infty$-categories, denoted by $\on{Cat}_\infty$.
\item The subcategory of $\on{Cat}_\infty$ spanned by stable $\infty$-categories and exact functors,  denoted by $\on{St}$.
\item The subcategory of $\on{Cat}_\infty$ spanned by presentable and stable $\infty$-categories and left adjoint functors, denoted by $\mathcal{P}r^L_{\on{St}}$.
\item The subcategory of $\on{Cat}_\infty$ spanned by presentable and stable $\infty$-categories and right adjoint functors, denoted by $\mathcal{P}r^R_{\on{St}}$.
\item We let $\on{LinCat}_k=\on{LMod}_{\mathcal{D}(k)}(\mathcal{P}r^L_{\on{St}})$ be the $\infty$-category of $k$-linear $\infty$-categories, consisting of left modules over the symmetric monoidal derived $\infty$-category $\mathcal{D}(k)$ in $\mathcal{P}r^L_{\on{St}}$, see \cite[Appendix D.1]{SAG} for background.
\end{itemize}
For simplicity, we will only consider $k$-linear $\infty$-categories in this paper, even though many concepts and constructions would also work $R$-linearly, with $R$ any $\mathbb{E}_\infty$-ring spectrum.

We note that passing to adjoints defines an equivalence of $\infty$-categories
\[ \on{radj}(\mhyphen)\colon \mathcal{P}r^L_{\on{St}}\simeq (\mathcal{P}r^R_{\on{St}})^{\on{op}}\cocolon \on{ladj}(\mhyphen)\,,\]
see \cite[5.5.3.4]{HTT}. Recall also that a functor between presentable $\infty$-categories admits a right adjoint (i.e.\ is a left adjoint) if and only if it preserves all (small) colimits, see \cite[5.5.2.9]{HTT}.

Given a $k$-linear $\infty$-category $\mathcal{C}$, and two objects $X,Y\in \mathcal{C}$, we denote by $\on{RHom}_{\mathcal{C}}(X,Y)\in \mathcal{D}(k)$ the \textit{$k$-linear morphism object}, see \cite[4.2.1.28]{HA}, characterized by the universal property
\[ \on{Map}_{\mathcal{D}(k)}(C,\on{RHom}_{\mathcal{C}}(X,Y))\simeq \on{Map}_{\mathcal{C}}(C\otimes X,Y)\,.\]

The remainder of this section concerns the computation of limits in $\infty$-categories of $\infty$-categories. The forgetful functors $\on{LinCat}_k,\mathcal{P}r^L_{\on{St}},\mathcal{P}r^R_{\on{St}},\on{St}\to \on{Cat}_\infty$ preserve (small) limits. We remark that the forgetful functor $\on{LinCat}_k\to \mathcal{P}r^L_{\on{St}}$ also preserves small colimits, see \cite[4.2.3.5]{HA}. As we now recall, limits in $\on{Cat}_\infty$ of strictly commuting diagrams of simplicial sets indexed by $1$-categories admit a very explicit description in terms of sections of an explicit Grothendieck construction.

Let $F\colon Z\to \on{Set}_{\Delta}$ be a functor from a $1$-category $Z$ to the $1$-category of simplicial sets, i.e.\ a strictly commuting diagram of simplicial sets. We suppose that $F$ takes values in $\infty$-categories, and thus gives rise to a diagram $\hF\colon N(Z)\to \on{Cat}_\infty$ (which one should think of as a homotopy commutative diagram), where $N(Z)$ is the nerve of $Z$. In the following, we will not distinguish in notation between $Z$ and $N(Z)$.

Later on, we will be interested in the case where $Z$ is the exit path category of a quiver, and thus all $2$-simplices in $Z$ will be degenerate. In this case, we thus have that any diagram $\hF\colon Z\to \on{Cat}_\infty$ commutes strictly, i.e.\ arises from a functor $F$ as above.

To $F$, we associate an $\infty$-category $\Gamma(F)$, called the \textit{Grothendieck construction}, together with a coCartesian fibration $p\colon \Gamma(F)\to Z$, which is classified by the functor $\hF$. The definition of $\Gamma(F)$ is given in \cite[3.2.5.2]{HTT}, where $\Gamma(F)$ is called the relative nerve. Informally, an object in $\Gamma(F)$ consists of choices of $z\in Z$ and $x\in F(z)$. A morphism in $\Gamma(F)$ from $(z,x)$ to $(z',x')$ consists of a morphism $\alpha\colon  z\to z'$ in $Z$ and a morphism $F(\alpha)(x)\to x'$ in $F(z')$.

A section of $p$ consists of a functor $s\colon Z\to \Gamma(F)$, whose composite with $p$ is the identity on $Z$. These assemble into the \textit{$\infty$-category of sections}
\[\on{Fun}_Z(Z,\Gamma(F))\coloneqq \on{Fun}(Z,\Gamma(F))\times_{\on{Fun}(Z,Z)} \{\on{id}_Z\}\,.\]
A section $s$ is called \textit{coCartesian}, if $s(e)$ is a \textit{$p$-coCartesian edge} (in the dual sense of \cite[2.4.1.1]{HTT}) for all edges $e$ of $Z$. Concretely, an edge $s(e)\colon (z,x)\to (z',x')$ is $p$-coCartesian if and only if the morphism $F(\alpha)(x)\to x'$ in $F(z')$ is an equivalence. The full subcategory of $\on{Fun}_Z(Z,\Gamma(F))$ spanned by coCartesian sections of $p$ describes the limit of $\hF$, see \cite[7.4.1.9]{Ker}. This description of the limit is both simple and very explicit, thus allowing for concrete computations. Computationally, we also benefit from the observation that the limit arises as the full subcategory of the bigger $\infty$-category of all sections of $p$ (the lax limit).

\subsection{Spherical adjunctions and perverse schobers locally}

In this section we discuss perverse schobers on ribbon graphs with a single vertex of valency $n$ (\textit{$n$-spiders}).
These are shown to correspond to spherical adjunction, up to equivalence.
They serve as the local model for perverse schobers on general ribbon graphs, which we will consider in the subsequent section.

\begin{definition}
Let $F\colon \mathcal{V}\leftrightarrow \mathcal{N}\cocolon G$ be an adjunction between stable $\infty$-categories.
\begin{enumerate}
\item The \emph{twist functor} $T_{\mathcal{V}}\colon \mathcal{V}\rightarrow \mathcal{V}$ is defined as the cofiber of the unit $\on{u}\colon \on{id}_{\mathcal{V}}\rightarrow GF$ in the stable $\infty$-category $\on{Fun}(\mathcal{V},\mathcal{V})$ of endofunctors\footnote{Note that this definition of twist and cotwist functors differs from the standard convention for the definition of the twist along a spherical object (which corresponds to a shifted cotwist).}.
\item Dually, the \emph{cotwist functor} $T_{\mathcal{N}}$ is defined as the fiber of the counit $\on{cu}\colon FG\rightarrow \on{id}_{\mathcal{N}}$.
\item The adjunction $F\dashv G$ is called \emph{spherical} if both $T_{\mathcal{V}}$ and $T_{\mathcal{N}}$ are equivalences.
\item A functor is \emph{spherical} if it admits a right adjoint, such that the corresponding adjunction is spherical.
\end{enumerate}
\end{definition}

The notion of a spherical functor was introduced by Anno, see also \cite{AL17}. Treatments in the setting of stable $\infty$-categories can be found in \cite{DKSS19,Chr20}.

As observed in \cite{KS14}, spherical adjunctions may be seen as a categorification of perverse sheaves on the complex unit disc $\mathbb{D}\subset \mathbb{C}$ with a single singularity at $0$. In this description, the $\infty$-category $\mathcal{N}$ categorifies the generic stalk of the perverse sheaf, i.e.\ any stalk away from $0$, which is a chain complex concentrated in a single degree. The $\infty$-category $\mathcal{V}$ categorifies the derived global sections of the perverse sheaf with support on the positive real axis $\mathbb{R}_{\geq 0}\cap \mathbb{D} \subset \mathbb{D}$. These derived sections turn out to form a chain complex whose homology is concentrated in a single degree; this homology is called the vector space of vanishing cycles. Categorified perverse sheaves are also referred to as perverse schobers. We can think of $\mathbb{R}_{\geq 0} \cap \mathbb{D}\subset \mathbb{D}$ as describing a graph with a single, $1$-valent vertex at $0$ inside the disc. This graph is also called the $1$-spider, as a special case of the $j$-spider (which we will use later):

\begin{definition}\label{def:spider}
For any number $j\in\NN_{\ge1}$, the \emph{$j$-spider} $\spider_{j}$ is the ribbon graph with a single vertex $v$ and $j$ incident external edges.
\end{definition}
We may thus summarize the above in the following definition.
\begin{definition}\label{def:schber1gon}
 A \emph{perverse schober parametrized by the $1$-spider}, or on the $1$-spider for short, consists of a spherical adjunction between stable $\infty$-categories
\[ F\colon \mathcal{V}\longleftrightarrow \mathcal{N}\cocolon G\,.\]
 \end{definition}

More generally, one can describe perverse sheaves on $\mathbb{D}$ in terms of their sections with support on the $n$-spider, where the $n$-valent vertex again lies at $0$ and $n\geq 1$, see \cite{KS16}. We propose to categorify this description as follows, leading to the notion of perverse schober on the $n$-spider.

\begin{definition}\label{def:schberngon}
 Let $n\geq 2$. A collection of adjunctions
\[ (F_i\colon \mathcal{V}^n\longleftrightarrow \mathcal{N}_i\cocolon G_i)_{i\in \ZZ/n}\]
between stable $\infty$-categories is called a \emph{perverse schober parametrized by the $n$-spider}, or a perverse schober on the $n$-spider for short, if
\begin{enumerate}
    \item $G_i$ is fully faithful, i.e.\ $F_iG_i\simeq \on{id}_{\mathcal{N}_i}$ via the counit,
    \item $F_{i}\circ G_{i+1}$ is an equivalence of $\infty$-categories,
    \item $F_i\circ G_j\simeq 0$ if $j\neq i,i+1$,
    \item $G_i$ admits a right adjoint $\on{radj}(G_i)$ and
    \item $\on{fib}(\on{radj}(G_{i+1}))=\on{fib}(F_{i})$ as full subcategories of $\mathcal{V}^n$.
\end{enumerate}
We will also consider a collection of functors $(F_i\colon \mathcal{V}^n\rightarrow \mathcal{N}_i)_{i\in \ZZ/n}$ as determining a perverse schober on the $n$-spider, or as a perverse schober on the $n$-spider for short, if there exist adjunctions $(F_i\dashv \on{radj}(F_i))_{i\in \ZZ/n}$ which define a perverse schober on the $n$-spider.
\end{definition}

In the subsequent sections of this paper, we will be interested in, and thus consider, $k$-linear perverse schobers, meaning that they take values in $k$-linear stable $\infty$-categories (which we assumed to be presentable) and $k$-linear functors. The comparison of the above \Cref{def:schberngon} with the definition of a perverse schober on the $n$-spider given in \cite{Chr22b} however does not rely on the $k$-linearity or presentability, so we prefer to stay in the more general setting of arbitrary stable $\infty$-categories in this subsection.

\begin{remark}\label{rem:adjgifi}
Consider a perverse schober on the $n$-spider $(F_i\colon \mathcal{V}^n\rightarrow \mathcal{N}_i)_{i\in \ZZ/n}$. Since $G_i$ and $G_{i+1}$ are fully faithful, we find that their adjoints $F_i$ and $\on{radj}(G_{i+1})$ are localizations in the sense that they factor through Verdier quotients as
\[
F_i\colon \mathcal{V}^n\twoheadrightarrow \mathcal{V}^n/\on{fib}(F_i)\simeq \mathcal{N}_i\]
and
\[\on{radj}(G_{i+1})\colon \mathcal{V}^n\twoheadrightarrow  \mathcal{V}^n/\on{fib}(\on{radj}(G_{i+1}))\simeq \mathcal{N}_{i+1}\,.
\]
By condition $5^\circ$, we have $\on{fib}(F_i)=\on{fib}(\on{radj}(G_{i+1}))$, and the resulting equivalence $\mathcal{N}_{i+1}\simeq \mathcal{N}_i$ identifies with $F_i G_{i+1}$. We thus see that condition $2^\circ$ follows from condition $5^\circ$, making it redundant. Furthermore, we find that $F_i\simeq F_i G_{i+1} \on{radj}(G_{i+1})$, meaning that $F_i$ and $\on{radj}(G_{i+1})$ are equivalent up to postcomposition with the equivalence $F_i G_{i+1}\colon \mathcal{N}_{i+1}\simeq \mathcal{N}_{i}$.

Note that this also shows that $F_i$ admits a left adjoint $\on{ladj}(F_i)$, which is then also equivalent to $G_{i+1}$ up to precomposition with an equivalence.
\end{remark}

At first glance, \Cref{def:schberngon} may seem quite different from \Cref{def:schber1gon}. We proceed with justifying \Cref{def:schberngon}, by showing that the datum of a perverse schober on the $n$-spider is equivalent to the datum of a perverse schober on the $1$-spider. Before that, we illustrate the special case $n=2$. In this case, the $\infty$-category $\mathcal{V}^2$ admits a $4$-periodic semiorthogonal decomposition $(\mathcal{V}^1,\mathcal{N})$, which implies that the gluing functor $\mathcal{N}\to \mathcal{V}^1$ (see below) is spherical, giving rise to the corresponding perverse schober on the $1$-spider. See also \cite{HLS16,DKSS19} for discussions of the relation between $4$-periodic semiorthogonal decompositions and spherical functors. The two fully faithful functors $\mathcal{N}\simeq \mathcal{N}_i\xrightarrow{G_i} \mathcal{V}^2$, $i=1,2$, describe the inclusion of the component $\mathcal{N}$ of the semiorthogonal decomposition and the inclusion of a component of a mutated semiorthogonal decomposition.

\begin{proposition}\label{prop:ngonschober}
Let $n\geq 2$. Given a perverse schober on the $n$-spider
\[ (F_i\colon \mathcal{V}^n\longleftrightarrow \mathcal{N}_i\cocolon G_i)_{i\in \ZZ/n}\]
and an integer $1\leq j\leq n$, the collection of functors
\begin{equation}\label{eq:gif}
(F_i|_{\on{fib}(F_j)}\colon  \on{fib}(F_j)\longrightarrow \mathcal{N}_i)_{j\neq i\in \ZZ/n}
\end{equation}
determines a perverse schober on the $(n-1)$-spider.
\end{proposition}

\begin{proof}
We begin with the case $n\geq 3$. We fix $j$ as above and first describe the right adjoints of the functors $F_i|_{\on{fib}(F_j)}$ for any $i\in 
(\mathbb{Z}/n)\backslash \{j\}$. For $i\neq j,j+1$, we have $F_j G_i\simeq 0$. For $i\neq j,j+1$, the functor $G_j$ thus factors through $\on{fib}(F_j)\subset \mathcal{V}^n$ and this factorization defines the right adjoint of $F_i|_{\on{fib}(F_j)}$.

We next describe the right adjoint of $F_{j+1}|_{\on{fib}(F_j)}$. The adjunction unit defines a natural transformation $\epsilon \colon G_{j+1}\to G_{j}F_{j}G_{j+1}$, which becomes a natural equivalence after composing with $F_j$. The image of the fiber $\on{fib}(\epsilon)$ of $\epsilon$ is hence contained in $\on{fib}(F_j)$. We show that $\on{fib}(\epsilon)$ is right adjoint to $F_{j+1}|_{\on{fib}(F_j)}$. We consider the adjunction unit $\on{u}\colon \on{id}_{\mathcal{V}^n}\to G_{j+1} F_{j+1}$. Any natural transformation $\on{id}_{\on{fib}(F_j)}\to G_jF_jG_{j+1}F_{j+1}$ is adjoint to a natural transformation $0\simeq F_j|_{{\on{fib}(F_j)}}\to F_jG_{j+1}F_{j+1}$ and thus vanishes. Thus, the restriction of $\on{u}$ to ${\on{fib}(F_j)}$ induces a natural transformation
\[
\on{u}'\colon \on{id}_{\on{fib}(F_j)}\to \on{fib}(\epsilon)\circ F_{j+1}|_{\on{fib}(F_j)}\simeq \on{fib}(G_{j+1}F_{j+1}\to G_{j}F_{j}G_{j+1}F_{j+1})\,.
\]
We define a second natural transformation as the natural equivalence
\[
\on{cu}'\colon F_{j+1}|_{{\on{fib}(F_j)}} \circ \on{fib}(\epsilon)\simeq F_{j+1} G_{j+1}\simeq \on{id}_{\mathcal{N}_{j+1}}\,.
\]
The natural transformations $\on{u}'$ and $\on{cu}'$ satisfy the triangle identities:

The natural transformation $F_{j+1}|_{\on{fib}(F_j)}\on{u}'$ identifies with the restriction of $F_{j+1}\on{u}$ since $F_{j+1} G_j\simeq 0$. The composite $(\on{cu}'F_{j+1}|_{\on{fib}(\epsilon)})\circ (F_{j+1}|_{\on{fib}(\epsilon)}\on{u}')$ thus identifies with the restriction of the natural transformation $\on{id}_{F_{j+1}}\simeq (\on{cu}F_{j+1})\circ (F_{j+1}\on{u})\colon F_{j+1}\simeq F_{j+1}$ to ${\on{fib}(F_j)}$.

As above, any natural transformation 
\[ \on{fib}(\epsilon)\simeq \on{id}_{\on{fib}(F_j)}\circ \on{fib}(\epsilon)\to G_jF_jG_{j+1}\]
is adjoint to a zero transformation and thus vanishes. Hence, to show that $(\on{fib}(\epsilon) \on{cu}')\circ (\on{u}' \on{fib}(\epsilon))$ is equivalent to the identity, it suffices to show that $(\on{fib}(\epsilon) \on{cu}')\circ (\on{u}' \on{fib}(\epsilon))$ composes with the fiber natural transformation $\alpha\colon \on{fib}(\epsilon)\to G_{j+1}$ of $\epsilon$ to the fiber natural transformation $\alpha$. This follows from the commutativity of the following diagram:
\[
\begin{tikzcd}[column sep=large]
\on{fib}(\epsilon) \arrow[d, "\alpha"] \arrow[r, "\on{u}' \on{fib}(\epsilon)"']                                              & \on{fib}(\epsilon)\circ F_{j+1} \circ \on{fib}(\epsilon) \arrow[r, "\simeq"']  \arrow[r, "(\on{fib}(\epsilon)\circ F_{j+1})\alpha"]  \arrow[rr, "\on{fib}(\epsilon)\on{cu}'", bend left=18] \arrow[d, "\simeq"] \arrow[d, "(\on{fib}(\epsilon)\circ F_{j+1})\alpha"'] & \on{fib}(\epsilon)\circ F_{j+1} \circ G_{j+1} \arrow[d, "\alpha F_{j+1}G_{j+1}"'] \arrow[r, "\on{fib}(\epsilon) \on{cu}"'] & \on{fib}(\epsilon) \arrow[d, "\alpha"] \\
G_{j+1} \arrow[rr, "\on{u}G_{j+1}"', bend right=12] \arrow[rrr, "\on{id}_{G_{j+1}}"', bend right=20] \arrow[r, "\on{u}'G_{j+1}"] & \on{fib}(\epsilon)\circ F_{j+1} G_{j+1} \arrow[r, "\alpha F_{j+1}G_{j+1}"]                                                                                                                                     & G_{j+1}F_{j+1}G_{j+1} \arrow[r, "G_{j+1}\on{cu}"]                                                                               & G_{j+1}                               
\end{tikzcd}
\]
The commutativity of the top and bottom triangles follows from the definition of $\on{cu}'$, $\on{u}'$, and the triangle identity for $F_{j+1}\dashv G_{j+1}$. The commutativity of the left and right squares is the interchange law for natural transformations and the middle square trivially commutes.

We have thus proved the two triangle identities, showing that $F_{j+1}|_{\on{fib}(\epsilon)}$ is left adjoint to $\on{fib}(\epsilon)$, see \cite[{Tags \href{https://kerodon.net/tag/02EK}{02EK},\href{https://kerodon.net/tag/02EL}{02EL}}]{Ker}.

With the above, we verify in the following that the collection of functors \eqref{eq:gif} and their right adjoints satisfy the conditions of \Cref{def:schberngon}: For $i\in (\mathbb{Z}/n)\backslash \{j\}$, let $G_i'$ denote the right adjoint of $F_i|_{\on{fib}(F_j)}$. For $i\not= j+1$, $G_i'$ is a restriction of $G_i$ and thus fully faithful. The functor $G_{j+1}'=\on{fib}(\epsilon)$ is fully faithful since the adjunction counit $\on{cu}'$ is an equivalence. This shows condition $1^\circ$.

For condition $2^\circ$, the only new condition to check is that $F_{j-1}|_{\on{fib}(F_j)}\circ G_{j+1}'$ is an equivalence. This follows from
\[ F_{j-1}|_{\on{fib}(F_j)}\circ \on{fib}(\epsilon)\simeq \on{fib}(\underbrace{F_{j-1}\circ G_{j+1}}_{\simeq 0}\to F_{j-1}G_jF_jG_{j+1})\simeq F_{j-1}G_jF_jG_{j+1}[-1]\,,\]
where $F_{j-1}G_j$ and $F_jG_{j+1}$ are each equivalences by assumption.

For condition $3^\circ$, the new condition to check is that $F_i|_{\on{fib}(F_j)}\circ \on{fib}(\epsilon)\simeq 0$ for $i\not= j+1,j-1$, which, by inspecting the fiber sequence defining $\on{fib}(\epsilon)$, reduces to $F_i\circ G_{j+1},F_i\circ G_{j}\simeq 0$.

For condition $4^\circ$, we note that the right adjoint of $G_{i}'$ is given by $\on{radj}(G_i)|_{\on{fib}(F_j)}$ for $i\not= j+1$. Since the passage to right adjoints defines an exact functor, $\on{fib}(\epsilon)$ admits a right adjoint, given by $\on{cof}(\on{radj}(G_{j+1})G_j\on{radj}(G_j)\to \on{radj}(G_{j+1}))$.

For condition $5^\circ$, the only new condition to check is that the fiber of $\on{radj}(G_{j+1}')=\on{radj}(\on{fib}(\epsilon))$ agrees with $\on{fib}(F_{j-1}|_{\on{fib}(F_j)})$. We observe that
\[ \on{radj}(G_{j+1}')\simeq \on{cof}(\on{radj}(G_{j+1})G_j\on{radj}(G_j)\to \on{radj}(G_{j+1}))\] is equivalent to $\on{radj}(G_{j+1})G_j\on{radj}(G_j')[1]$, since $\on{radj}(G_{j+1})$ vanishes on $\on{fib}(F_j)=\on{fib}(\on{radj}(G_{j+1}))$ (where the latter equality arises from condition $5^\circ$ for the perverse schober on the $n$-spider). The assumption that $F_jG_{j+1}$ is an equivalence, implies after passing to right adjoints that $\on{radj}(G_{j+1})G_j$ is also an equivalence. It follows that
\[
\on{fib}(\on{radj}(G_{j+1}'))\simeq \on{fib}(\on{radj}(\on{G_j})|_{\on{fib}(F_j)})\simeq \on{fib}(F_{j-1}|_{\on{fib}(F_j)})\,,
\]
again by condition $5^\circ$ of the perverse schober on the $n$-spider.\\

We proceed with the case $n=2$. We prove the assertion via a different argument than above, but note for completeness that the right adjoint of $F_2|_{\on{fib}(F_1)}$ is again given by $\on{fib}(\epsilon)=\on{fib}(G_2\to G_1F_1G_2)$. Using the adjunctions $F_1\dashv G_1$ and $F_2\dashv G_2$ and \cite[Prop.~2.3.2]{DKSS19}, we find that there are semiorthogonal decompositions $(\on{Im}(G_i),\on{fib}(F_i))$ and $(\on{fib}(\on{radj}(G_i)),\on{Im}(G_i))$ for $i=1,2$, where $\on{Im}(G_i)\subset \mathcal{V}^n$ denotes the essential image of $G_i$. The condition $\on{fib}(\on{radj}(G_{i+1}))=\on{fib}(F_i)$ is equivalent to the $4$-periodicity of these semiorthogonal decompositions, in the sense of \cite{HLS16,DKSS19}. It follows from \cite[Prop.~2.5.12]{DKSS19} that the gluing functor of $(\on{fib}(\on{radj}(G_i)),\on{Im}(G_i))$, given by the restriction of $F_i$ to $\on{fib}(\on{radj}(G_i))=\on{fib}(F_{i-1})$, is a spherical functor.
\end{proof}

Next, we describe an explicit model for perverse schobers on the $n$-spider, taking as input a spherical functor. This model is inspired by Dyckerhoff's categorified Dold-Kan correspondence \cite{Dyc17}, see also the discussion in \cite{Chr22}. Let $F\colon \mathcal{V}\to \mathcal{N}$ be a spherical functor and define $\mathcal{V}^n_F$ as the stable $\infty$-category of diagrams \[ a\rightarrow b_1\rightarrow \dots \rightarrow b_{n-1}\]
where $a\in \mathcal{V}$, $b_i\in \mathcal{N}$ and the morphism $a\rightarrow b_1$ lies in the Grothendieck construction of $F$, i.e.\ encodes a morphism $F(a)\rightarrow b_1$. For a more precise definition of $\mathcal{V}^n_F$, in terms of sections of a Grothendieck construction, see \cite[Notation~3.2, Lem.~2.29]{Chr22}. We define $\varrho_1=\pi_{n-1}\colon \mathcal{V}^n_F\rightarrow \mathcal{N}$ to be the projection functor to $b_{n-1}$ and, for $2\leq i\leq n$, $\varrho_i$ to be the $(2i-2)$-fold left adjoint of $\pi_{n-1}$. On objects, the functor $\varrho_i$ can be explicitly described via the following formula, see \cite[Lem.~3.3]{Chr22}: Let $X=(a\to b_1\to \dots \to b_{n-1})\in \mathcal{V}^n_F$. Then
\begin{equation}\label{eq:description_rho_functor}
    \varrho_i(X)\simeq \begin{cases} 
    b_{n-1} & i=1\,, \\ 
    \on{fib}(b_{n-i}\to b_{n-i+1})[i-1] & 1<i<n\,,\\
    \on{fib}(F(a)\to b_1)[n-1] & i=n\,. \end{cases}
\end{equation}
The functor $\varrho_i$ admits a right adjoint $\varsigma_i\colon \mathcal{N}\to \mathcal{V}^n_F$. This functor can be described just as explicitly on objects, see again \cite[Lem.~3.3]{Chr22}: Let $b\in \mathcal{N}$. Then:
\begin{equation}\label{eq:description_sigma_functor}
    \varsigma_i(b)\simeq \begin{cases} 
    (G(b)\to b \to \dots \to b) & i=1\,, \\ 
    (0\to \dots \to 0\to \underbrace{b}_{\makebox[0pt][l]{$\scriptstyle{(n-i-2)\text{-th position}}$}}\to 0 \to \dots )[-i+2] & 1<i\leq n\,,
    \end{cases}
\end{equation}
where the above morphism $G(b)\to b$ is Cartesian in the Grothendieck construction, meaning it encodes the counit morphism $FG(b)\to b$.

The collection of adjunctions $(\varrho_i\dashv \varsigma_i)_{i\in \mathbb{Z}/n}$ defines a perverse schober on the $n$-spider, as follows readily from inspecting the above descriptions and using \cite[Lem.~3.8]{Chr22}:

\begin{proposition}[The model for perverse schobers on the $n$-spider]\label{prop:ngonschober2}
Let $F\colon \mathcal{V}\to \mathcal{N}$ be a spherical functor. The collection of adjunctions
\[(\varrho_i\colon \mathcal{V}^n_F\longleftrightarrow\mathcal{N}\noloc \varsigma_i)_{i\in \ZZ/n} \]
defines a perverse schober on the $n$-spider, denoted $\hF_n(F)$.
\end{proposition}

Consider a stable $\infty$-category $\mathcal{C}$ with a semiorthogonal decomposition $(\mathcal{A},\mathcal{B})$. The inclusion $\iota_{\mathcal{A}}\colon \mathcal{A}\subset \C$ always admits a left adjoint and the inclusion $\iota_\mathcal{B}\colon \mathcal{B}\subset \D$ always admits a right adjoint \cite[Prop.~2.3.2]{DKSS19}. If furthermore $\iota_\mathcal{A}$ admits a right adjoint $\iota_\mathcal{A}^R$, then $(\mathcal{A},\mathcal{B})$ is called a \textit{Cartesian semiorthogonal decomposition} and the corresponding \textit{gluing functor} is defined as $\iota_A^R\circ \iota_\mathcal{B}\colon \mathcal{B}\to \mathcal{A}$. Similarly, if $\iota_\mathcal{B}$ admits a left adjoint $\iota_\mathcal{B}^L$, then $(\mathcal{A},\mathcal{B})$ is called a \textit{coCartesian semiorthogonal decomposition}, and the \textit{gluing functor} is $\iota_\mathcal{B}^L\circ \iota_\mathcal{A}\colon \mathcal{A}\to \mathcal{B}$. In either of these two cases, $\C$ can be reconstructed given $\mathcal{A}$ and $\mathcal{B}$ as well as the gluing functor, see \cite[Rem.~2.3.7]{DKSS19}.

Given a perverse schober on the $n$-spider $\mathcal{F}$, we obtain the \textit{spherical adjunction underlying $\mathcal{F}$}, by repeatedly applying \Cref{prop:ngonschober} to any $(n-1)$ of the $n$ edges of the $n$-spider. Up to equivalence, the arising spherical adjunction is indepent of the choice of edges. 

\begin{proposition}\label{prop:localmodelforschobers}
Consider a perverse schober on the $n$-spider $\mathcal{F}$ 
with underlying spherical functor $F\colon \mathcal{V}\to \mathcal{N}$. Then there exists an equivalence of perverse schobers\footnote{By an equivalence of perverse schobers on the $n$-spider, we mean a compatible collection of equivalences of the $n$ adjunctions.} $\mathcal{F}\simeq \mathcal{F}_n(F)$.\\
In particular, \Cref{def:schberngon} is equivalent to the definition of a perverse schober on the $n$-spider given in \cite{Chr22}.
\end{proposition}

\begin{proof}
We prove the assertion by an induction on $n$. In the case $n=1$, the assertion is trivial. 

Suppose that $n>1$. We denote the perverse schober $\mathcal{F}$ by $(F_i\colon \mathcal{V}^n\longleftrightarrow \mathcal{N}_i\cocolon G_i)_{i\in \ZZ/n}$.  Fix $j\in \mathbb{Z}/n$ and denote the perverse schober on the $(n-1)$-spider arising from $\mathcal{F}$ via \Cref{prop:ngonschober} by $(F_i'\colon \mathcal{V}^{n-1}\longleftrightarrow \mathcal{N}_i\cocolon G_i')_{i\in \ZZ/n, i\not= j}$. For the induction step, assume that $\mathcal{F}'\simeq \mathcal{F}_{n-1}(F)$. Then $\mathcal{V}^{n-1}=\on{fib}(F_j)\simeq \mathcal{V}^{n-1}_F$. The stable $\infty$-category $\mathcal{V}^n$ has a semiorthogonal decomposition into $\mathcal{V}^{n-1}$ and $\mathcal{N}_j\simeq \mathcal{N}$ with fully faithful gluing functor, see \Cref{lem:SOD} below. Further $\mathcal{V}^n_F$ has a semiorthogonal decomposition into $\mathcal{V}^{n-1}_F$ and $\mathcal{N}$, and the corresponding gluing functor identifies with the gluing functor of $\mathcal{V}^n$. We thus find $\mathcal{V}^n\simeq \mathcal{V}^n_F$. Further, under this identification with $\mathcal{V}^n_F$, the functor $F_j$ corresponds to the left adjoint of the inclusion of the component $\mathcal{N}\simeq \mathcal{N}_j$, meaning to $\varrho_j$. By \Cref{rem:adjgifi}, the functors $F_i$ and $\varrho_i$, with $i\in (\mathbb{Z}/n)\backslash \{j\}$, can be recovered up to postcomposition with equivalences by passing to repeated adjoints of $F_j$ and $\varrho_j$, respectively. Thus also $F_i$ is equivalent to $\varrho_i$ (up to postcomposition with equivalences), showing the desired equivalence of perverse schobers.
\end{proof}

\begin{lemma}\label{lem:SOD}
Let $(F_i\colon \mathcal{V}^n\longleftrightarrow \mathcal{N}_i\cocolon G_i)_{i\in \ZZ/n}$ be a perverse schober on the $n$-spider. Then there is a Cartesian semiorthogonal decomposition $(\on{fib}(F_{i}),\on{Im}(G_{i+1}))$ for any $i\in \ZZ/n$. If $n=2$, the gluing functor is spherical and if $n\geq 3$, the gluing functor from $\on{Im}(G_{i+1})$ to $\on{fib}(F_{i})$ is fully faithful.
\end{lemma}

\begin{proof}
The semiorthogonal decomposition follows from $\on{fib}(F_{i})=\on{fib}(\on{radj}(G_{i+1}))$, see \cite[Prop.~2.3.2]{DKSS19}. There is similarly a semiorthogonal decomposition $(\on{Im}(G_i),\on{fib}(F_i))$ of $\mathcal{V}^n$. Hence, $(\on{fib}(F_{i}),\on{Im}(G_{i+1}))$ is Cartesian by \cite[Prop.~2.3.2]{DKSS19}.
In the case $n=2$, the semiorthogonal decomposition is $4$-periodic, as shown in the proof of \Cref{prop:ngonschober}, so that the gluing functor is spherical by \cite[Prop.~2.5.5]{DKSS19}. We next consider the case $n\geq 3$. Denote by $\iota$ the fully faithful functor $\on{fib}(F_{i})\to \mathcal{V}^n$.
The gluing functor $\mathcal{N}_{i+1}\xrightarrow{G_{i+1}} \mathcal{V}^n\xrightarrow{\on{radj}(\iota)}\on{fib}(F_{i})$ is the right adjoint of $F_{i+1}|_{\on{fib}(F_{i})}$, and thus fully faithful by \Cref{prop:ngonschober}.
\end{proof}

\subsection{Parametrized perverse schobers}\label{sec:pps}

In this section we recall the definition of a perverse schober parametrized by a ribbon graph in terms of certain functors from the exit path category of the graph, as well as the $\infty$-categories of global sections and lax sections.  

\begin{definition}
A \emph{ribbon graph} consists of a graph $\rgraph$ together with a choice of a cyclic order on the set of halfedges incident to $v$ for each $v\in\rgraph_0$.
\end{definition}

\begin{definition}
Let $\rgraph$ be a ribbon graph. The \textit{exit path category} $\on{Exit}(\rgraph)$ is the $1$-category with
\begin{itemize}
\item objects the vertices and edges of $\rgraph$ and
\item a morphism $v\xrightarrow{a} e$ for each halfedge $a$ of $\rgraph$ which is part of an edge $e$ incident to a vertex $v$ of $\rgraph$. All other morphisms are identities.
\end{itemize}
We do not distinguish in notation between $\on{Exit}(\rgraph)$ and its simplicial nerve $N(\on{Exit}(\rgraph))\in \on{Set}_\Delta$.
\end{definition}

Let $v$ be a vertex of valency $n$ of a ribbon graph $\rgraph$. Let $\on{Exit}(\rgraph)_{v/}$ be the undercategory, which has $n+1$ objects, which can be identified with $v$ and its $n$ incident halfedges and non-identity morphisms going from $v$ to these halfedges. There is a functor $\on{Exit}(\rgraph)_{v/}\to \on{Exit}(\rgraph)$, which is fully faithful if $\Gamma$ has no loops incident to $v$.

\begin{definition}\label{def:schober}
Let $\rgraph$ be a ribbon graph. A functor $\hF\colon \on{Exit}(\rgraph)\rightarrow \on{LinCat}_k$ is called a \textit{$\rgraph$-parametrized ($k$-linear) perverse schober} if for each vertex $v$ of $\rgraph$, the restriction of $\hF$ to $\on{Exit}(\rgraph)_{v/}$ determines a perverse schober parametrized by the $n$-spider in the sense of \Cref{def:schberngon}. In this case, we call the spherical functor obtained from repeated application of \Cref{prop:ngonschober} the spherical functor\footnote{The resulting spherical functor does not depend on the choices of integers $j$ in the process, up to composition with equivalences.} describing the perverse schober $\hF$ at $v$. We call $v$ a \textit{singularity} of $\hF$ if this spherical functor is non-trivial, i.e.\ its domain is not equivalent to the zero $\infty$-category.
\end{definition}

Given a connected ribbon graph $\rgraph$, a $\rgraph$-parametrized perverse schober $\mathcal{F}$ assigns to each edge of $\rgraph$ an equivalent stable $\infty$-category, referred to as the generic stalk of $\mathcal{F}$.

\begin{remark}
\Cref{prop:localmodelforschobers} implies that \Cref{def:schober} is equivalent to the definition of parametrized perverse schober appearing in Definition 4.14 of \cite{Chr22}.
\end{remark}

\begin{definition}\label{def:sections}
Let $\hF$ be a $\rgraph$-parametrized perverse schober.
\begin{itemize}
    \item We denote by $\glsec(\rgraph,\hF)=\on{lim}(\hF)$ the limit of $\hF$ in the $\infty$-category $\on{LinCat}_k$ of $k$-linear $\infty$-categories. We call $\glsec(\rgraph,\hF)$ the $\infty$-\textit{category of global sections} of $\hF$. Recall that $\glsec(\rgraph,\hF)$ can be identified with the $\infty$-category of coCartesian sections of the Grothendieck construction of $\hF$.
    \item We denote by $\losec(\rgraph,\hF)$ the $\infty$-category of all sections of the Grothendieck construction of $\hF$. We call $\losec(\rgraph,\hF)$ the $\infty$-\textit{category of lax sections} of $\hF$. Note that $\losec(\rgraph,\hF)$ describes the $(\infty,2)$-categorical lax limit of $\hF$, and thus has no immediate analog in classical sheaf theory. We further remark that $\losec(\rgraph,\hF)$ is stable and admits a canonical $k$-linear structure.
\end{itemize}
Note that we have a fully faithful $k$-linear functor $\glsec(\rgraph,\hF)\subset \losec(\rgraph,\hF)$.
\end{definition}

The local rotational ``symmetry'' of a perverse schober at a vertex is captured by composition with the equivalence $T_{\hF(v)}$ defined in the following proposition. This is in general not a genuine symmetry, as a full rotation does not return the perverse schober to itself; the change depends on the ``monodromy'' around the singularity given by the suspension of the cotwist functor of the corresponding spherical adjunction. We will refer to this rotational action as the paracyclic symmetry, see also \cite{Chr22} for background.

\begin{proposition}[{\cite[Lem.~3.8]{Chr22}}]\label{prop:twistv}
Let $\hF$ be a $\rgraph$-parametrized perverse schober and $v\in \rgraph_0$ a vertex with incident halfedges $a_1,\dots,a_n$ lying in edges $e_1,\dots,e_n\in \rgraph_1$. The functor
\[F_v'\coloneqq \prod_{i=1}^n\hF(v\xrightarrow{a_i}e_i)\colon \hF(v)\longrightarrow \prod_{i=1}^n\hF(e_i)\]
is spherical. Let $G_v'$ be its right adjoint. We denote the twist functor of the adjunction $F_v'\dashv G_v'$ by $T_{\hF(v)}$.
\end{proposition}

\subsection{Equivalences from contractions of ribbon graphs}\label{sec:EfromC}

To complement the discussion in the previous section, we recall here the fact that one can transport perverse schobers along certain contractions of ribbon graphs in a way that preserves the $\infty$-category of global sections, up to equivalence.

Let $\rgraph$ be a ribbon graph and $e\in \rgraph_1^\circ$ an internal edge which is not a loop. By contracting the edge $e$, we obtain a new ribbon graph $\rgraph'$, which we call the \textit{contraction} of $\rgraph$ at the edge $e$. For a more formal description of this procedure, see \cite[Def.~4.24]{Chr22}. Given two ribbon graphs $\rgraph$ and $\rgraph'$, we write $c\colon \rgraph\rightarrow \rgraph'$ if $\rgraph'$ can be obtained from $\rgraph$ by contracting finitely many edges, and call $c$ a contraction of ribbon graphs.

\begin{lemma}\label{lem:contr}
Let $c\colon \rgraph\rightarrow \rgraph'$ be a contraction of ribbon graphs.
\begin{enumerate}[label=(\arabic*)]
\item Let $\hF$ be a $\rgraph$-parametrized perverse schober and assume that $c$ contracts no edges incident to two singularities of $\hF$. Then there exists a canonical $\rgraph'$-parametrized perverse schober $c_*(\hF)$ together with an equivalence of $\infty$-categories $c_*\colon \glsec(\rgraph,\hF)\simeq \glsec(\rgraph',c_*(\hF))$.
\item Let $\hF'$ be a $\rgraph'$-parametrized perverse schober. For each choice of subset $S\subset \rgraph_0$, such that $c|_{S}\colon S\rightarrow \rgraph_0'$ defines a bijection between $S$ and the singularities of $\hF'$, there exists a canonical $\rgraph$-parametrized perverse schober $c^*(\hF')$ together with an equivalence of $\infty$-categories $c^*\colon \glsec(\rgraph',\hF')\simeq \glsec(\rgraph,c^*(\hF'))$.
\end{enumerate}
\end{lemma}

\begin{proof}
Part (1) was shown in \cite[Prop.~4.28]{Chr22}.

For part (2), it suffices to treat the case that $c$ contracts a single edge $e$ of $\rgraph$ incident to vertices $v_1,v_2$, whose image in $\rgraph'$ is denoted $v$. Swapping the labels of $v_1$ and $v_2$ if necessary, we may assume that $S=\{v_2\}$. Let $m_i$ be the valency of $v_i$, $i=1,2$, and $n=m_1+m_2-2$ the valency of $v$. Replacing $\hF'$ by an equivalent $\rgraph'$-parametrized perverse schober, we may assume that near $v$, $\hF'$ is given by the following diagram, with $F_v'\colon \mathcal{V}\to \mathcal{N}$ the spherical functor underlying $\hF'$ at $v$, see also \Cref{prop:localmodelforschobers}:
\[
\begin{tikzcd}
\mathcal{N} &                                                                                                                             & \mathcal{N} \\
\dots       & \mathcal{V}^n_{F_v'} \arrow[ld, "\varrho_{m_1-1}"'] \arrow[rd, "\varrho_{m_1}"] \arrow[ru, "\varrho_{n}"'] \arrow[lu, "\varrho_{1}"'] & \dots       \\
\mathcal{N} &                                                                                                                             & \mathcal{N}
\end{tikzcd}
\]
We define $c^*(\hF')$ to be identical to $\hF'$ away from $v_1,v_2$ and near $v_1,v_2$ as the following diagram,
\[
\begin{tikzcd}
\mathcal{N} &                                                                                                    &             &                                                                                                             & \mathcal{N} \\
\dots       & \mathcal{V}^{m_1}_{0_{\mathcal{N}}} \arrow[r, "\varrho_{m_1}"] \arrow[lu, "\varrho_1"] \arrow[ld, "\varrho_{m_1-1}"] & \mathcal{N} & \mathcal{V}^{m_2}_{F_v'} \arrow[rd, "\varrho_2"] \arrow[ru, "\varrho_{m_2}"] \arrow[l, "{\varrho_1}"] & \dots       \\
\mathcal{N} &                                                                                                    &             &                                                                                                             & \mathcal{N}
\end{tikzcd}
\]
where we denote by $0_{\mathcal{N}}$ the spherical functor $0\colon 0\rightarrow \mathcal{V}_v$. By \cite[Lem.~4.26]{Chr22}, we then get $c_*c^*(\hF')\simeq \hF'$. The equivalence $\glsec(\rgraph',\hF')\simeq \glsec(\rgraph,c^*(\hF'))$ in part (2) thus follows from part (1).
\end{proof}

\begin{remark}
We remark that the construction of the perverse schober $c^*(\hF')$ in the proof of \Cref{lem:contr}, though involving choices, leads to a uniquely determined perverse schober up to equivalence.  Furthermore, we have $c^*c_*(\hF)\simeq \hF$ and $c_*c^*(\hF')\simeq \hF'$, and we expect $c_*$ and $c^*$ to define inverse equivalences between the $\infty$-categories, or even $(\infty,2)$-categories, of $\rgraph$- and $\rgraph'$-parametrized perverse schobers.
\end{remark}

\section{Arc system kits for perverse schobers}\label{sec:kits}
For the entire section, we fix a weighted marked surface $\sow$, with a mixed-angulation $\AS$. We denote by $\Sgh=\dAS$ the dual S-graph of the mixed-angulation.

In \Cref{sec:EfromF}, we introduce the setup of this section, based on perverse schobers parametrized by \emph{extended S-graphs} and the effect of flips of S-graph on perverse schobers. In \Cref{sec:arcsystemkit}, we introduce the notion of an arc system kit for a perverse schober parametrized by such an extended S-graph. In \Cref{sec:objfromarcs,sec:Hom-Int}, we describe how to associate to each graded arc a global section of a perverse schober equipped with an arc system kit, and study their properties. In the final \Cref{subsec:flipsrevisited} we describe how arc system kits can be transported along flips of the S-graph.
Combining the different results of this section, we obtain the main result \Cref{thm:tilt=flip}, stating that for positive arc system kits, flips of an $S$-graph induce a simple tilt on the corresponding simple-minded collections.

\def\edge{\eta}
\subsection{Equivalences from flips of S-graphs}\label{sec:EfromF}
In this section we explain how to realize flips of S-graphs in terms of zig-zags of edge contractions, thereby allowing perverse schobers to be transported across such flips in a way that preserves their $\infty$-categories of global sections, up to equivalence.

\begin{definition}\label{def:extdgraph}
The \textit{extended graph} $\eS$ of an S-graph $\Sgh$ is obtained by adding an external edge to $\Sgh$ at each boundary vertex. This edge is placed at the final position in the total order of the halfedges at the boundary vertex, which induces a compatible cyclic order of the halfedges. Hence, we consider $\eS$ as a ribbon graph.
\end{definition}

We refer to the halfedges and edges of $\eS$ which are not part of $\Sgh$ as \textit{virtual}. These edges are not embedded in $\sow$. The primary reason for their introduction is rather technical: the diagrams we wish to study define $\eS$-parametrized perverse schobers, but not $\Sgh$-parametrized perverse schobers. The other halfedges and edges of $\eS$ are called \textit{non-virtual} and can be identified with the halfedges of edges of $\Sgh$. We also identify the vertices of $\eS$ and $\Sgh$, and use this when referring to vertices of $\eS$ as internal or boundary vertices.

\begin{remark}\label{rem:flip}
We can describe the effect of a forward flip of an S-graph $\Sgh$ at an edge $\edge$ on the
associated ribbon graph $\eS$ in terms of a zig-zag of contractions of ribbon graphs. Consider first the case that $\edge$ is not incident to a singular point of weight $-1$. In this case, for each halfedge whose incident
vertex changes under the flip (there are either 0, 1 or 2 of these), we once include the span of contractions which is near the flipped edge e of the form depicted in Figure~\ref{fig:contraction1}.
\begin{figure}[ht]
\makebox[\textwidth][c]{
\begin{tikzpicture}[scale=.7,xscale=-1,arrow/.style={->,>=stealth}]
  \draw[red,thick](-1,0) edge (-1,-2) edge (-3,0) to
                  (1,0) edge (1,2) edge (1,-2) edge (3,0)
                  (1,0) to (-1,2)
                  (1,0) node[above left]{$\cdots$}\ww(-1,0) node[below right]{$\cdots$}\ww;
\begin{scope}[xscale=1,shift={(8,0)}]
  \draw[red,thick](-1,0)edge (-1,-2) edge (-3,0) to
                  (1,0) edge (1,2) edge (1,-2) edge (3,0)
                  (0,0)\ww to (-1,2)
                  (1,0) node[above left]{$\cdots$}\ww(-1,0) node[below right]{$\cdots$}\ww;
\end{scope}
\begin{scope}[xscale=1,shift={(16,0)}]
  \draw[red,thick](-1,0) edge (-1,2) edge (-1,-2) edge (-3,0) to
                  (1,0)  edge (1,2) edge (1,-2) edge (3,0)
                  (1,0) node[above left]{$\cdots$}\ww(-1,0) node[below right]{$\cdots$}\ww;
\end{scope}
\draw[Emerald,ultra thick, opacity=.9,arrow](11.5,0)to(12.5,0);
\draw[Emerald,ultra thick, opacity=.9,arrow](4.5,0)to(3.5,0);
\end{tikzpicture}}
\caption{A span of contractions moving a halfedge, describing part of the forward flip at a normal arc.}
\label{fig:contraction1}
\end{figure}

In the case that $\edge$ is incident to a single singular point of weight $-1$ (i.e.\ degree $1$), we instead use the zig-zag of contractions depicted in Figure~\ref{fig:contraction2}.

\begin{figure}[ht]
\makebox[\textwidth][c]{
\begin{tikzpicture}[scale=.6,arrow/.style={<-,>=stealth}]
  \draw[red,thick](0,1)to[bend left=0](0,-3)\ww  (-3,2)\ww to (0,1)\ww to(3,2)\ww;
  \draw[red,thick,font=\scriptsize](0,-1)\ww node[left]{} to (0,1)\ww node[above]{$\cdots$};
\begin{scope}[xscale=1,shift={(10,0)}]
  \draw[red,thick](0,1)to[bend left=45](0,-3)\ww  (-3,2)\ww to (0,1)\ww to(3,2)\ww;
  \draw[red,thick,font=\scriptsize](0,-1)\ww node[left]{$1$}  to (0,1)\ww node[above]{$\cdots$};
\end{scope}
\begin{scope}[xscale=1,shift={(-10,0)}]
  \draw[red,thick](0,1)to[bend left=-45](0,-3)\ww  (-3,2)\ww to (0,1)\ww to(3,2)\ww;
  \draw[red,thick,font=\scriptsize](0,-1)\ww node[right]{$1$} to (0,1)\ww node[above]{$\cdots$};
\end{scope}
\begin{scope}[xscale=1,shift={(-5,5)}]
  \draw[red,thick,font=\scriptsize](0,-1)\ww node[right]{} to (0,1)\ww node[above]{$\cdots$};
  \draw[red,thick](0,0)\ww to[bend left=-60](0,-3)\ww (-3,2)\ww to (0,1)\ww to(3,2)\ww;
\end{scope}
\begin{scope}[xscale=1,shift={(5,5)}]
  \draw[red,thick,font=\scriptsize](0,-1)\ww node[right]{} to (0,1)\ww node[above]{$\cdots$};
  \draw[red,thick](0,0)\ww to[bend left=60](0,-3)\ww  (-3,2)\ww to (0,1)\ww to(3,2)\ww;
\end{scope}
\draw[Emerald,ultra thick, opacity=.9,arrow](2,3)to(4,4);
\draw[Emerald,ultra thick, opacity=.9,arrow](-2,3)to(-4,4);
\draw[Emerald,ultra thick, opacity=.9,arrow](9,2.5)to(7,4);
\draw[Emerald,ultra thick, opacity=.9,arrow](-9,2.5)to(-7,4);
\end{tikzpicture}}
\caption{The zig-zag of contractions describes a forward flip at a monogon arc.}
\label{fig:contraction2}
\end{figure}
\end{remark}

\begin{proposition}\label{prop:flipeq}
Let $\hF$ be an $\eS$-parametrized perverse schober and $\edge\in \eS_1$ a non-virtual edge.
Denote by $\Sgh^\sharp$ the forward flip of $\Sgh$ at $\edge$.
There exists a canonical $\eS^\sharp$-parametrized perverse schober $\hF^\sharp$ together with an equivalence of $\infty$-categories $\Gamma(\eS,\hF)\simeq \Gamma(\eS^\sharp,\hF^\sharp)$.

A similar statement holds for backward flips.
\end{proposition}

\begin{proof}
Combine \Cref{lem:contr} and \Cref{rem:flip}.
\end{proof}

\subsection{Arc system kits}\label{sec:arcsystemkit}
A well-known construction in the symplectic geometry of Lefschetz fibrations, originally due to Donaldson, and explained in detail in \cite{Sei08}, yields Lagrangian matching spheres (also known as matching cycles) using vanishing cycles and matching paths as input. A vanishing cycle refers to a Lagrangian sphere in the symplectic fiber of the Lefschetz fibration, which collapses to a point when parallel-transported along an arc in the base ending in a singular value of the fibration. If we can find two such singular values where the parallel transport of the Lagrangian sphere collapses, the union of the Lagrangian sphere and all its parallel transports in all fibers over the arc connecting the two singular values forms a Lagrangian sphere (of one dimension higher) in the total space of the fibration, referred to as the matching sphere (see Figure~\ref{fig:matchingsphere}). This construction has an analogue in the setting of perverse schobers: consider an edge $e$ of a ribbon graph \rgraph{} connecting the two vertices $v_1,v_2$. Let $\mathcal{F}$ be a \rgraph-parametrized perverse schober. The vanishing cycle corresponds to a choice of object $L_e\in \mathcal{F}(e)$. In a typical situation (see for instance the perverse schobers considered in \cite{Chr22,Chr21b}), $L_e$ would be a spherical object, meaning that the functor $\mhyphen\otimes L_e\colon \D(k)\to \mathcal{F}(e)$ is a spherical functor. The categorical meaning of $L_e$ being a vanishing cycle that collapses at either end of the arc $e$ is that this spherical functor $\mhyphen\otimes L_e$ describes the perverse schober $\mathcal{F}$ at $v_1$ and $v_2$. There are then distinguished objects $L_{v_1,e}\in \mathcal{F}(v_1),L_{v_2,e}\in \mathcal{F}(v_2)$, satisfying that $\mathcal{F}(v_1\to e)(L_{v_1,e})\simeq L_e$ and $\mathcal{F}(v_2\to e)(L_{v_2,e})\simeq L_e$. Note that if $v_i$ is $1$-valent, then simply $L_{v_i,e}=k\in \D(k)=\mathcal{F}(v_i)$. We can consider $L_e,L_{v_1,e},L_{v_2,e}$ as compatible lax sections of the perverse schober, which thus glue together to a global section of $\mathcal{F}$, which is the analogue of the matching sphere.

\begin{figure}
\begin{tikzpicture}[scale=.4]
\begin{scope}[shift={(37,-7)},scale=2.7,arrow/.style={-stealth}]
\draw[thick,fill=gray!7](0,0) ellipse (5 and 1);
\draw[blue!60,ultra thick,fill=blue!11](0,4)circle(2);\draw[red,ultra thick] (-2,0)to(2,0);
\foreach \j in {0}{
    \begin{scope}[shift={(\j,0)}]
    \draw[gray,dashed,very thick](-2,4)\ww to(-2,0)\ww (2,4)\ww to(2,0)\ww;
    \begin{scope}[shift={(0,4)}]
    \draw[very thick,cyan,fill=teal!10,opacity=0.3](60:2).. controls +(-30:.5) and +(30:.5) ..(-60:2);
    \draw[thick,cyan,dashed,fill=teal!10,opacity=0.3](60:2).. controls +(150:.5) and +(-150:.5) ..(-60:2);
    \draw[thick,cyan,fill=cyan!10,opacity=0.3](105:2).. controls +(-15:.5) and +(-15:.5) ..(-105:2);
    \draw[thick,cyan,dashed,fill=cyan!10,opacity=0.3](105:2).. controls +(-165:.5) and +(165:.5) ..(-105:2);
    \draw[very thick,teal](60:2).. controls +(-30:.5) and +(30:.5) ..(-60:2);
    \draw[thick,teal,dashed](60:2).. controls +(150:.5) and +(-150:.5) ..(-60:2);
    \draw[thick,cyan](105:2).. controls +(-15:.5) and +(-15:.5) ..(-105:2);
    \draw[thick,cyan,dashed](105:2).. controls +(-165:.5) and +(165:.5) ..(-105:2);
    \draw[gray,dashed,thick]
        (-60:2)to(1,-4)node[teal]{\tiny{$\bullet$}}
        (-105:2)to(-0.258819*2,-4) node[cyan]{\tiny{$\bullet$}};
    \end{scope}
    \end{scope}}
\end{scope}
\end{tikzpicture}
\caption{Arc in the base surface and corresponding matching sphere in the total space of the fibration.}
\label{fig:matchingsphere}
\end{figure}

We next consider the generalized, categorical analogue of the (quite favorable) situation where the fibers of the Lefschetz fibration are each equipped with a choice of a vanishing cycle, these being related to each other via parallel transport, and such that the vanishing cycles collapse over every singular value of the fibration. In this situation, we can associate matching spheres with arbitrary graded arcs in the base surface of the fibration. The corresponding categorical data for a perverse schober, see \Cref{def:ask}, is referred to as an arc system kit. The additional condition v) in \Cref{def:ask} essentially ensures that there is only a single object (up to shift) taking the role of the vanishing cycle near each vertex, but furthermore states that the arc system kit has a local rotational symmetry similar to the rotational symmetry of the polygons.

When considering arc system kits for an $\eS$-parametrized perverse schober $\hF$, we will always assume that the set of singularities of $\hF$ consists exactly of the interior vertices of $\eS$. We fix such an $\eS$-parametrized perverse schober $\hF$.

\begin{definition}\label{def:ask}
An \emph{arc system kit} for $\hF$ consists of
\begin{enumerate}[label=\roman*)]
\item an object $L_e\in \hF(e)$ for each non-virtual edge $e$ of $\eS$,
\item an object $L_{v,a}\in \hF(v)$ for each vertex $v$ and non-virtual incident halfedge $a$ of $\eS$,
\item an equivalence in $\hF(e)$
\[ \hF(v\xrightarrow{b}e)(L_{v,a})\simeq \begin{cases} L_e & a=b\\ 0 & \text{else}\end{cases}\]
for each pair of non-virtual halfedges $a,b$ incident to a vertex $v$ and where $b$ is part of the edge $e$,
\item an equivalence in $\hF(c)$
\[ \hF(v\to c)(L_{v,b})\simeq \hF(v\to c)(L_{v,a}[1-d(a,b)])\,,\]
for each virtual edge $c$ of $\eS$ incident to a vertex $v$ of weight $\infty$ and consecutive non-virtual halfedges $a,b$  (i.e.\ $b$ follows $a$) incident to $v$, and
\item an equivalence in $\hF(v)$
\begin{equation}\label{eq:twistofLva} T_{\hF(v)}(L_{v,b})\simeq L_{v,a}[1-d(a,b)]\,,\end{equation}
for each internal vertex $v$ of $\eS$ and consecutive internal halfedges $a,b$ (i.e.\ $b$ follows a) incident to $v$. Here $T_{\hF(v)}$ denotes the twist functor from \Cref{prop:twistv}. If $v$ has valency $1$ with the single incident halfedge $a$, we instead require $T_{\hF(v)}(L_{v,a})\simeq L_{v,a}[1-d(v)]$ with $d(v)$ the degree of $v$.
\end{enumerate}
Note that in particular, if $v$ is $q$-valent of degree $m$, one has $T_{\hF(v)}^m(L_{v,a})\simeq L_{v,a}[q-m]$.
\end{definition}

The datum in an arc system kit satisfies a `local-to-global principle', meaning that it amounts to local arc system kits, one for each vertex of $\eS$ and its incident edges, plus the requirement that the data at the edges agrees. The construction below describes the default local arc system kits that we will use.

\begin{construction}\label{constr:arcsyskit}
\textbf{Case I:} interior vertices

Consider the disc with a single interior singular point $v$ of degree $m\geq q$ and $q$ singular points on the boundary. There is a mixed-angulation with dual graph $\Sgh$ having a single $q$-valent interior vertex $v$ and all interior edges emanating from $v$. We choose any total order on these $q$ interior edges which is compatible with their given cyclic order, labeling them by $e_1,\dots,e_q$. Given a spherical functor $F\colon \mathcal{V}\to \mathcal{N}$, we let $\hF_v(F)$ be the $\spider_{q}$-parametrized perverse schober obtained from $\hF_q(F)$, as in \Cref{prop:ngonschober2}, by composing $\hF_q(F)(v\to e_i)$ with $[d(a_1,a_i)-(i-1)]$ for all $1\leq i\leq q$. Note that if $q=m$, then $\hF_v(F)=\hF_q(F)$. Let $L_v\in \mathcal{V}$ be an object satisfying that $T_{\mathcal{V}}(L_v)\simeq L_v[1-m]$, with $T_{\mathcal{V}}$ the twist functor of $F\dashv G$. For instance, if $\mathcal{V}=\D(k)$, $L_v=k\in \D(k)$ and $F(k)\in \mathcal{N}$ is an $m$-spherical object, then $T_{\mathcal{V}}(L_v)\simeq L_v[1-m]$ holds. The data for an arc system kit for $\hF_v(F)$ at the vertex $v$ is obtained as follows.

We set $L_e=F(L_v)\in \mathcal{N}=\hF_v(F)(e)$ for all edges $e\in \eS_1$. Let $a_i$ be the $i$-th halfedge incident to $v$. We set  $L_{v,a_i}\in \mathcal{V}_{F}^q=\hF_v(F)(v)$ to be the diagram
\begin{equation}\label{eq:locLv}
\left(L_v\rightarrow F(L_v)\xrightarrow{\on{id}}F(L_v)\xrightarrow{\on{id}} \dots\xrightarrow{\on{id}} \underbrace{F(L_v)}_{q-i+1\text{-th}}\rightarrow 0\rightarrow \dots \rightarrow 0\right)[-d(a_1,a_i)]
\end{equation}
where the first morphism $L_v\to F(L_v)$ is coCartesian in the Grothendieck construction of $F$. Then we have
\[
\mathcal{F}_q(F)(v\to e_1)(L_{v,a_i})\simeq \varrho_1(L_{v,a_i})=\pi_{q-1}(L_{v,a_i})\simeq \begin{cases} F(L_v) & a=1 \\ 0 & a>1 \,. \end{cases}
\]
Let $2\leq j\leq n-1$. Using the equivalence $\varrho_j\simeq \on{fib}_{n-j,n-j+1}[j-1]$ from equation \eqref{eq:description_rho_functor}, we find
\[ \mathcal{F}_v(F)(v\to e_j)(L_{v,a_i})\simeq \varrho_j(L_{v,a_i})[d(a_1,a_j)-(j-1)]\simeq \on{fib}_{n-j,n-j+1}(L_{v,a_i})[d(a_1,a_j)]\,,\]
where the latter objects indicates by definition the ($[d(a_1,a_j)]$-shifted) fiber of the morphism between the $(n-j)$-th and $(n-j+1)$-th entries in the diagram \eqref{eq:locLv}. We thus find
\[
\mathcal{F}_v(F)(v\to e_j)(L_{v,a_i})\simeq \begin{cases}
\on{fib}(F(L_v)\xrightarrow{\on{id}} F(L_v))[d(a_1,a_j)-d(a_1,a_i)]\simeq 0 & j>i \\
\on{fib}(F(L_v)\to 0)\simeq F(L_v) & j=i \\
\on{fib}(0\to 0)\simeq 0 &j<i\,.
\end{cases}
\]
The last functor $\varrho_q$ corresponds to taking the relative fiber along the morphism in the Grothendieck construction, as in equation \eqref{eq:description_rho_functor}, meaning that
\begin{align*}
\mathcal{F}_v(F)(v\to e_q)(L_{v,a_i})& \simeq \varrho_q(L_{v,a_i})[d(a_1,a_q)-(q-1)](L_{v,a_i})\\
& \simeq \begin{cases} \on{fib}(F(L_v)\to F(L_v))[d(a_1,a_q)-d(a_1,a_i)]\simeq 0 & i<q \\ \on{fib}(F(L_v)\to 0)\simeq F(L_v) & i=q\,.\end{cases}
\end{align*}
This shows the equivalences as in iii) in the definition of an arc system kit. We next show the equivalence \eqref{eq:twistofLva} from v). Let $\mathcal{F}_v(F)(v\to e_i)^R$ be the right adjoint of $\mathcal{F}_v(F)(v\to e_i)$. We then have $\prod_{1\leq j\leq q} \mathcal{F}_v(F)(v\to e_j)^R \mathcal{F}_v(F)(v\to e_j) (L_{v,a_i})\simeq \mathcal{F}_v(F)(v\to e_i)^R(F(L_v))\in \mathcal{F}_v(F)(v)$. This object can be depicted for $i>1$ as
\[
\left(0\to \dots \to 0\to \underbrace{F(L_v)}_{(q-i+2)-th} \to 0\to \dots \to 0\right)[-d(a_1,a_i)+1]
\]
and for $i=1$ as
\[
\left( GF(L_v)\to F(L_v)\xrightarrow{\on{id}}\dots \xrightarrow{\on{id}}F(L_v)\right)\,,
\]
as follows from equation \eqref{eq:description_sigma_functor} and the defintion of $\mathcal{F}_v(F)$. For $i=1$, we can depict the resulting (vertical) cofiber sequence computing $T_{\mathcal{F}_v(F)(v)}(L_{v,a_i})\in \mathcal{F}_v(F)(v)$ as follows. Note that the cofiber is computed ``entrywise in the diagram''.
\[
\begin{tikzcd}
L_v \arrow[r] \arrow[d, "\on{unit}"'] & F(L_v) \arrow[r, "\on{id}"] \arrow[d, "\on{id}"] & \dots \arrow[r, "\on{id}"] & F(L_v) \arrow[d, "\on{id}"] \\
GF(L_v) \arrow[r] \arrow[d]           & F(L_v) \arrow[r, "\on{id}"] \arrow[d]            & \dots \arrow[r, "\on{id}"] & F(L_v) \arrow[d]            \\
T_{\mathcal{V}}(L_v) \arrow[r]                      & 0 \arrow[r]                                      & \dots \arrow[r]            & 0
\end{tikzcd}
\]
Since $T_{\mathcal{V}}(L_v)\simeq L_v[1-m]$ and $m=d(a_1,a_q)+d(a_q,a_1)$, we find $T_{\mathcal{F}_v(F)(v)}(L_{v,a_1})\simeq L_{v,a_{q}}[1-d(a_{q},a_1)]$. For $i>1$, we similarly see that $T_{\mathcal{F}_v(F)(v)}(L_{v,a_i})\simeq L_{v,a_{i-1}}[1-d(a_{i-1},a_i)]$.

\textbf{Case II:} boundary vertices

Consider the disc with $q$ singular points on the boundary and with a mixed-angulation whose dual S-graph $\Sgh$ consists of $q$ edges emanating from a single boundary vertex $v$. At $v$, the extended S-graph $\eS$ is described by the $(q+1)$-spider. The halfedges of the $(q+1)$-spider are equipped with the total order where the virtual halfedge is in the final position. Let $\hF$ be an $\eS$-parametrized perverse schober which restricts on the $(q+1)$-spider to the perverse schober $\hF_v(0)$ obtained from $\hF_q(0)$, with $0\colon 0\rightarrow \mathcal{N}$ the zero functor, by composing $\hF_q(0)(v\to e_i)$ with $[d(a_1,a_i)-(2i-1)]$ for $1\leq i\leq q$. Note that $\hF_v(0)(v)= \mathcal{V}_0^{q+1}\simeq \on{Fun}(\Delta^{q-1},\mathcal{N})$. Choose any object $L\in \mathcal{N}$. We set $L_e=L\in \hF(e)=\mathcal{N}$ for all non-virtual edges $e$ of $\eS$. We further set $L_{v,a_i}\in \mathcal{V}_0^{q+1}$ to be the object
\begin{equation}\label{eq:locLv2}
\left(L\xrightarrow{\on{id}}L\xrightarrow{\on{id}} \dots\xrightarrow{\on{id}} \underbrace{L}_{q-i+1\text{-th}}\rightarrow 0\rightarrow \dots \rightarrow 0\right)[i-d(a_1,a_i)]\,.
\end{equation}
A computation similar to the above again shows that these objects describe the data of an arc system kit for $\hF$ at $v$.
\end{construction}

In the following, we assume that $\hF$ is equipped with an arc system kit.

\begin{lemma}\label{lem:locend}
There are the following equivalences in $\mathcal{D}(k)$:
\begin{enumerate}[label=(\arabic*)]
\item Let $e,\eprime$ be two non-virtual edges of $\eS$. Then
    $\on{End}_{\hF(e)}(L_e)\simeq \on{End}_{\hF(\eprime)}(L_{\eprime})$.
\item Let $v$ be a vertex of $\eS$ with incident non-virtual halfedges $a,b$.
    Then $\on{End}_{\hF(v)}(L_{v,a})\simeq \on{End}_{\hF(v)}(L_{v,b})$.
\item Let $v$ be a boundary vertex of $\eS$ and let $a$ be a non-virtual halfedge incident to $v$ which is part of an edge $e$.
    Then $\on{End}_{\hF(v)}(L_{v,a})\simeq \on{End}_{\hF(e)}(L_e)$.
\end{enumerate}
\end{lemma}

\begin{proof}
We begin by showing part (1). Since $\sow$ is assumed to be connected, it suffices to consider the case where $e,\eprime$ are incident to the same vertex $v$ with $a\in e$ and $b\in \eprime$ with $b$ the successor halfedge of $a$ at $v$. Let $T_{\hF(v)}$ be the twist functor of the adjunction $F_v'\dashv G_v'$ from \Cref{prop:twistv}. There exists an equivalence of functors
\[ \hF(v\xrightarrow{a}e)\circ T_{\hF(v)}\simeq T'\circ \hF(v\xrightarrow{b}\eprime)\]
with $T'\colon \hF(\eprime)\simeq \hF(e)$ some equivalence. This follows from the fact that, at $v$, $\hF$ is locally equivalent to a perverse schober on the $n$-spider $\hF_n(F)$ of the form described in \Cref{prop:ngonschober2}, see \Cref{prop:localmodelforschobers}. For $\hF_n$, one has $\hF_n(F)(v\xrightarrow{a}e)\circ T_{\hF_n(F)(v)}\simeq T''\circ \hF_n(F)(v\xrightarrow{b}\eprime)$ for some equivalence $T''$ by \cite[Prop.~3.11]{Chr22}. We thus find
\begin{align*}
T'(L_{\eprime})& \simeq T'\circ \hF(v\xrightarrow{b}\eprime)(L_{v,b}) \simeq \hF(v\xrightarrow{a}e)\circ T_{\hF(v)}(L_{v,b})\\
 & \simeq \hF(v\xrightarrow{a}e)(L_{v,a}[1-d(a,b)]) \simeq L_{e}[1-d(a,b)]\,.
\end{align*}
Since any equivalence, such as $T'[d(a,b)-1]$, preserves derived endomorphisms, part (1) follows.

Part (2) follows from part v) of \Cref{def:ask} by using that the autoequivalence $T_{\hF(v)}$ also preserves derived endomorphisms.

We conclude by showing part (3). Let $b$ be the virtual halfedge at $v$ with $v$ of valency $q+1$ in the extended S-graph. Let $a$ be the $i$-th halfedge at $v$ in the total order of halfedges where $b$ is final. Consider the functor
\[ \alpha\colon \hF(v)\xrightarrow{\prod_{c\neq a,b}\hF(v\xrightarrow{c}e_c)}\prod_{c\neq a,b} \hF(e_c)\,,\]
where the product runs over all halfedges $c\neq a,b$ at $v$ and $e_c$ denotes the corresponding edge of $\eS$. Since $v$ is a boundary vertex, the spherical functor describing $\hF$ at $v$ is given by $0\colon 0\rightarrow \mathcal{N}$.  Using the local model $\hF_n(0)$ from \Cref{prop:ngonschober2}, we find that under the equivalence $\hF(v)\simeq \on{Fun}(\Delta^{q-1},\mathcal{N})$, the fiber $\on{fib}(\alpha)$ describes the full subcategory of $\on{Fun}(\Delta^{q-1},\mathcal{N})$ generated by objects of the form
\[
L\xrightarrow{\on{id}}L\xrightarrow{\on{id}} \dots\xrightarrow{\on{id}} \underbrace{L}_{q-i+1\text{-th}}\rightarrow 0\rightarrow \dots \rightarrow 0
\]
with $L\in \hF(e)$. The functor $\hF(v\xrightarrow{a}e)$ thus induces an equivalence of $\infty$-categories $\on{fib}(\alpha)\simeq \hF(e)$. Using that by definition $L_{v,a}\in \on{fib}(\alpha)$, we find
\[ \on{End}_{\hF(v)}(L_{v,a})\simeq \on{End}_{\hF(e)}(\hF(v\xrightarrow{a}e)(L_{v,a}))\simeq \on{End}_{\hF(e)}(L_e)\,,\]
showing (3).
\end{proof}

\def\Ende{\on{End}_{L}}
\def\Endv{\on{End}_v}
\begin{definition}\label{rem:locend}~
\begin{itemize}
\item We denote $\Ende=\on{End}_{\hF(e)}(L_e)\in \mathcal{D}(k)$ for any choice of non-virtual edge $e$ of $\eS$ and $\Endv=\on{End}_{\hF(v)}(L_{v,a})\in \mathcal{D}(k)$ for each vertex $v$ of $\eS$ and any choice of incident non-virtual halfedge $a$.
\item We call the arc system kit of $\hF$ \textit{positive} if $\on{H}^0(\Ende)\simeq \on{H}^0(\Endv)\simeq k$, $\on{H}^i(\Ende)\simeq \on{H}^i(\Endv)\simeq 0$ for all $v\in \Sgh_0, i<0$ and finally if, for any weight $-1$ vertex $v$ with incident halfedge $a\in e\in \Sgh_1$, the $k$-vector space $\on{H}^1(\Endv)\simeq k$ is generated by the extension arising from combining the cofiber sequence
\[ L_{v,a}\xlongrightarrow{\on{unit}} \on{radj}(\hF(v\xrightarrow{a} e))\circ \hF(v\xrightarrow{a} e) (L_{v,a}) \longrightarrow T_{\hF(v)}(L_{v,a})\]
with the equivalence $L_{v,a}\simeq T_{\hF(v)}(L_{v,a})$.
\end{itemize}
\end{definition}

\begin{lemma}\label{lem:locmor}Let $v$ be a vertex of $\eS$ with incident non-virtual halfedges $a\neq b$.
\begin{enumerate}
\item[(1)]  Assume that $v$ is an interior vertex. Then there exists an equivalence
\[ \on{RHom}_{\hF(v)}(L_{v,a},L_{v,b})\simeq \Endv[-d(a,b)]\,.\]
\item[(2)] Assume that $v$ is a boundary vertex.  If $a<b$ in the total order of halfedges at $v$, then
\[ \on{RHom}_{\hF(v)}(L_{v,a},L_{v,b})\simeq \Ende[-d(a,b)]\,.\]
If $a>b$ in the total order of halfedges at $v$, then
\[ \on{RHom}_{\hF(v)}(L_{v,a},L_{v,b})\simeq 0\,.\]
\end{enumerate}
\end{lemma}

\begin{proof}
We begin bu proving part (1). We perform an induction on the number of halfedges appearing between $a$ and $b$ in the cyclic order. The base case of the induction is the case $a=b$, in which case $\on{RHom}_{\hF(v)}(L_{v,a},L_{v,b})\simeq \Endv$ by \Cref{lem:locend}. We proceed with the induction step. Let $c$ be the successor of $a$ at $v$. Denote the edges containing $a$ and $c$ by $e_a$ and $e_c$, respectively. Consider the restriction of the functor $G_v'$ from \Cref{prop:twistv} to a functor $\tilde{G}\colon \hF(e_c)\rightarrow \hF(v)$. We define the presentable stable $\infty$-category $\C$ via the cofiber sequence in $\mathcal{P}r^R_{\on{St}}$ (and also in $\mathcal{P}r^L_{\on{St}}$)
\[ \hF(e_c)\xrightarrow{\tilde{G}}\hF(v)\xrightarrow{\pi} \hC
\,.\] Since the cofiber coincides with the fiber of the adjoint diagram, which can be computed in $\on{LinCat}_k$, we see that $\C$ is $k$-linear.

The left adjoint of $\pi$ exhibits $\hC$ as the full subcategory of $\hF(v)$ spanned by those objects $X$ satisfying $\hF(v\xrightarrow{c}e_c)(X)\simeq 0$. It follows that
\[
\on{RHom}_{\hF(v)}(L_{v,a},L_{v,b})\simeq \on{RHom}_{\hC}(\pi(L_{v,a}),\pi(L_{v,b}))\,.
\]
However, in the following we want to consider $\hC$ as another full subcategory of $\hF(v)$, by instead using the right adjoint of $\pi$. Namely, $\hC$ is the full subcategory spanned by those objects $X$ satisfying $\hF(v\xrightarrow{a} e_a)(X)\simeq 0$. These two embeddings of $\hC$ into $\hF(v)$ differ by applying the twist functor $T_{\hF(v)}$ from \Cref{prop:twistv}, as follows from the fact that $\hF(v\xrightarrow{a} e_a)\circ T_{\hF(v)}$ and $\hF(v\xrightarrow{c}e_c)$ are equivalent functors up to postcomposition with an equivalence, see \cite[Prop.~3.11]{Chr22}. Under this identification of $\hC$ with a subcategory of $\hF(v)$, by \Cref{def:ask} part iii), we find $\pi(L_{v,b})\simeq L_{v,b}$ and $\pi(L_{v,c})\simeq L_{v,c}$.

Observe that $\pi G_v' F_v'(L_{v,c})\simeq \pi \tilde{G}(L_{e_c})\simeq 0$ by iii) of \Cref{def:ask}. Applying $\pi$ to the fiber and cofiber sequence
\[ L_{v,c}\rightarrow G_v' F_v'(L_{v,c})\rightarrow T_{\hF(v)}(L_{v,c})\]
we thus obtain an equivalence $L_{v,c}[1]\simeq \pi(T_{\hF(v)}(L_{v,c}))\simeq \pi(L_{v,a})[1-d(a,c)]\in \hC$. Using the induction assumption, we thus find
\begin{align*}
\on{RHom}_{\hF(v)}(L_{v,a},L_{v,b})& \simeq \on{RHom}_{\hC}(L_{v,c}[d(a,c)],L_{v,b})\\
& \simeq \Endv[-d(c,b)-d(a,c)]\\
& \simeq \Endv[-d(a,b)]\,,
\end{align*}
concluding the induction step and the proof of part (1).

Part (2) follows from a direct computation. For that, one uses that $\mathcal{F}(v)\simeq \on{Fun}(\Delta^{r-2},\mathcal{N})$, with $r$ the valency of $v$ and $\mathcal{N}$ the generic stalk. Under this equivalence, the objects correspond to the objects described in part II of \Cref{constr:arcsyskit}, since an object in $\on{Fun}(\Delta^{r-2},\mathcal{N})$ is uniquely determined by its values under $\varrho_1,\dots,\varrho_r$ if these values are non-trivial for only two of these functors\footnote{This follows from the observation that for any subset $I\subset \{1,\dots,r\}$ of cardinality $r-2$ the fiber of $\prod_{i\in I}\varrho_i\colon \on{Fun}(\Delta^{r-1},\mathcal{N})\to \mathcal{N}$ is equivalent to $\mathcal{N}$ and the two non-trivial functors induced by the $\varrho_i$ are equivalences. See also \eqref{eq:description_rho_functor}.}. 
\end{proof}

\subsection{Objects from arcs}\label{sec:objfromarcs}
Our next goal is to give a construction which assigns to a graded curve an object in the $\infty$-category of global sections of a perverse schober with arc system kit. This is reduced to a local problem by cutting the curve along arcs of a mixed-angulation and then gluing the associated lax sections.

Let $\eS$ be the extended ribbon graph of the dual graph $\Sgh$ and $\hF$ an $\eS$-parametrized perverse schober equipped with an arc system kit.
\begin{definition}
Let $e$ be a non-virtual edge of $\eS$ (or equivalently an edge of $\Sgh$) with incident vertices $v,\vprime$ (with $v=\vprime$ if $e$ is a loop). Let $a,\aprime$ be the halfedges of $e$ at $v,\vprime$.
We define a global section $\A_e\in \Gamma(\eS,\hF)$ associated to $e$ as the coCartesian section of the Grothendieck construction of $\hF$ with
\[ \A_e(x)=\begin{cases} L_e & x=e\\ L_{v,a}& x=v\\ L_{\vprime,\aprime}& x=\vprime \\ 0 &\text{else}\end{cases}\]
and $\A_e(v\xrightarrow{a} e)$ and $\A_e(\vprime \xrightarrow{\aprime} e)$ describing the chosen equivalences $\hF(v\xrightarrow{a} e)(L_{v,a}) \simeq L_e$ and $\hF(\vprime\xrightarrow{\aprime} e)(L_{\vprime,\aprime}) \simeq L_e$. If $v=\vprime$ and $e$ is thus a loop, we instead set $\A_e(v)=L_{v,a}\oplus L_{v,\aprime}$.
\end{definition}

Explicitly, the global section $\A_e$ can be depicted near $e$ in the case $v\neq \vprime$ and $v,\vprime$ lying at interior singular points as follows:
\[
\begin{tikzcd}
      & 0                                                                      &               & 0                                                                      &       \\
\dots & L_{v,a} \arrow[r] \arrow[u] \arrow[l] \arrow[d] & L_e & L_{\vprime,\aprime} \arrow[l] \arrow[u] \arrow[r] \arrow[d] & \dots \\
      & 0                                                                      &               & 0                                                                      &
\end{tikzcd}
\]
The section $\A_e$ vanishes on the remainder of $\on{Exit}(\eS)$.

If $v$ and $\vprime$ lie at boundary singular points, incident to the external edges $\eprime,\eprime'$, we can depict $\A_e$ near $e$ as follows.
\[
\begin{tikzcd}
\ddots   & 0                                             &     & 0                                                   & \reflectbox{$\ddots$} \\
{\mathcal{F}(v\to \eprime)(L_{v,a})} & L_{v,a} \arrow[r] \arrow[u] \arrow[l] \arrow[d] & L_e & L_{\vprime,\aprime} \arrow[l] \arrow[u] \arrow[r] \arrow[d] & {\mathcal{F}(\vprime\to \eprime')(L_{\vprime,\aprime})}\\
\reflectbox{$\ddots$}   & 0                                             &     & 0                                                   & \ddots
\end{tikzcd}
\]
The section $\A_e$ vanishes at the positions indicated by dots and on the remainder of $\on{Exit}(\eS)$.

We define $\ASG=\{\A_e\}_{e\in \Sgh_1}\subset \Gamma(\eS,\hF)$ to be the collection of objects corresponding to the edges of the S-graph.

\begin{remark}\label{rem:disassembletoask}
Given the collection of objects $\ASG$, we can recover the local objects contained in the arc system kit by evaluating these global sections at the edges and vertices of $\eS$. For instance, given an edge $e$, we recover $L_e\in \mathcal{F}(e)$ as the value of the section $\A_e$ at $e$.
\end{remark}

The definition of $\ASG$ is a special case of a more general construction, which associates a global section to graded arcs in $\sow$.
Here, and throughout this section, a \emph{graded arc} in $\sow$ is a graded curve in $\sow$ in the sense of \Cref{def:gradedcurve}, whose underlying curve is a compact arc (see \Cref{def:arc}).

We describe in the following how to associate global sections to arbitrary graded arcs. A more explicit version of this general construction for one specific choice of $\hF$ is also described in \cite{Chr21b}. The general strategy is as follows.
Given a graded arc $\edge$ in $\sow$, we choose a representative of its homotopy class and intersect it with the polygons comprising the given mixed-angulation $\AS$ of $\sow$. The representative of $\edge$ is chosen to intersect the mixed-angulation minimally. The intersections between $\AS$ and (its dual) $\Sgh$ define points on their edges, which we call midpoints. Replacing $\eta$ by a homotopic arc, we can assume that the intersections between $\eta$ and $\AS$ lie at these midpoints.

The result is a decomposition of $\edge$ into graded local curves, called segments, each lying in a polygon of the given mixed-angulation and considered up to suitable homotopies. The endpoints of the segments are the midpoints. We associate lax sections of $\hF$ to the segments of $\edge$, which we then glue in \Cref{constr:glsec} to produce the global section of $\hF$ associated with $\edge$.

\begin{figure}[ht]\centering
\begin{tikzpicture}[scale=.7]
\clip(-6,-5.5)rectangle(6,5.5);
\begin{scope}[shift={(0,0)}]
\draw[thick](0,0) circle (5) ;
\draw[thick,fill=gray!23](0,-2) circle (1);

\draw[JungleGreen,ultra thick] plot [smooth,tension=.7] coordinates{
    (15:2.3) (36:3) (0,3.6) (180-36:3) (180+5:3) (238:3.6) (0,-5)};
\draw[JungleGreen,font=\tiny]
    (30:3)node[below right]{$\delta_0$}
    (150:3.2)node[below left]{$\delta_4$}
    (65:3.2)node[below]{$\delta_1$}
    (115:3.2)node[below]{$\delta_2$}
    (210:3.4)node{$\delta_5$}
    (-105:4.2)node[below]{$\delta_6$}
;
\foreach \j in {0,...,4}{
    \draw[dashed,white,very thick]
    (54+72*\j:5)arc (54+72*\j:54+72*\j+18:5)
    (54+72*\j:5)arc (54+72*\j:54+72*\j-18:5);}
\foreach \j in {0,...,4}{
    \draw[font=\small] (90+72*\j:5) coordinate  (w\j)
        (72*\j+5:5)\nn(72*\j-5:5)\nn(72*\j-36-5:5)\nn(72*\j-36+5:5)\nn
        (90-36+72*\j:5) coordinate  (v\j);}
\draw[blue, thick](w0)to(0,2)edge(w1)edge(w4)
    (w1)to(-1,-2)edge(w2)edge[bend left=25](0,2)
    (w4)to(1,-2)edge(w3)edge[bend left=-25](0,2);

\draw[red](0,0)edge(15:2.3)to(165:2.3)
    (v0)to[bend left=15](v1)to(165:2.3)to(v2)to[bend left=15](v3)to[bend left=15](v4)to(15:2.3)to(v0);

\foreach \j in {0,...,4}{\draw(w\j)\nn;\draw(v\j)\ww;}
\draw(-1,-2)\nn(1,-2)\nn(0,2)\nn (0,0)\ww (165:2.3)\ww (15:2.3)\ww;
\end{scope}
\end{tikzpicture}
\caption{The decomposition of an arc into segments}
\label{fig:segments}
\end{figure}
\begin{example}
Consider the weighted marked annulus ${\bf S}$, with $5$ singular points on one of the boundary components and three interior singular points, which can be depicted as in Figure~\ref{fig:segments}.
The mixed-angulation $\AS$ of $\surf$ is depicted in blue and the dual S-graph $\Sgh$ is depicted in red.
A graded arc $\eta$ in ${\bf S}$ is depicted in green. Upon intersecting $\eta$ with the polygons of the mixed-angulation, $\eta$ decomposes into $7$ segments, $\delta_0,\dots,\delta_6$. The segments $\delta_0,\delta_6$ are of type I (in the sense defined below), and the other segments are of type II.
\end{example}

Next, we define the two kinds of segments $\delta_{v,i}$ and $\delta_{v,i,j}$ appearing in $\sow$ and the associated lax sections of $\hF$.

\textbf{Type I.}
Let $v$ be a vertex of $\eS$ and $e$ an edge of $\Sgh$ with halfedge $a$ at $v$. We consider $v$ as a point in $|\Sgh|\subset {\bf S}$ and $e$ as a subset of $|\Sgh|\subset {\bf S}$.
Consider the segment $\delta_{v,a}$ arising from the halfedge $a$, whose endpoints are $v$ and the midpoint of $e$. We let $\A_{\delta_{v,a}}\in \losec(\eS,\hF)$ be the lax section of $\hF$ defined via
\[ \A_{\delta_{v,a}}(x)=\begin{cases} L_e & x=e\\ L_{v,a}& x=v\\ 0 &\text{else}\end{cases}\]
and $\A_{\delta_{v,a}}(v\xrightarrow{a} e)$ being the coCartesian morphism describing the chosen equivalence $\hF(v\xrightarrow{a} e)(L_{v,a})\simeq L_e$.

\textbf{Type II.}
Let $v$ be a vertex of $\eS$ and $e_a,e_b$ two incident non-virtual edges of $\eS$ with halfedges $a, b$ at $v$. Let $\delta_{v,a,b}$ be an embedded curve lying in the $\AS$-polygon containing $v$, starting from the midpoint of $e_a$ and ending at the midpoint of $e_b$. We first consider the case that $\delta_{v,a,b}$ goes in the counterclockwise direction. We define
\[ \A_{\delta_{v,a,b}}=\on{fib}(\A_{\delta_{v,a}}\rightarrow \A_{\delta_{v,b}}[d(a,b)])\,,\]
as the fiber in $\losec(\eS,\hF)$ of a morphism in $\on{RHom}_{\losec(\eS,\hF)}(\A_{\delta_{v,a}},\A_{\delta_{v,b}})\simeq \on{RHom}_{\mathcal{F}(v)}(L_{v,a},L_{v,b})$ corresponding via \Cref{lem:locmor} to $\on{id}_{L_{v,a}}\in \Endv$. Spelling out the definition, if $a\neq b$, we observe
\[
\A_{\delta_{v,a,b}}(x)\simeq \begin{cases} L_{e_a} & x=e_a\,, \\ 
L_{e_b}[d(a,b)-1] & x=e_b\,, \\ 
\on{fib}(L_{v,a}\to L_{v,b}[d(a,b)]) & x=v\,,\\
0 & x\neq v,e_a,e_b\,.\end{cases}
\]
If $a=b$ and $v$ is not $1$-valent, we instead choose some halfedge $c\neq a$ with edge $e_c$, and define
\[
\A_{\delta_{v,a,a}}=\on{fib}(\A_{\delta_{v,a,c}}\to \A_{\delta_{v,c,a}}[d(a,c)-1])
\]
as the fiber of the (uniquely determined) morphism, which evaluates under $\mathcal{F}(v\xrightarrow{c}e_c)$ to the identity on $L_{e_c}[d(a,c)-1]$. The result is independent on the choice of $c$, up to equivalence, and we find
\begin{equation}\label{eq:A(x)split}
\A_{\delta_{v,a,a}}(x)\simeq \begin{cases} L_{e_a}\oplus L_{e_a}[d(v)-1] & x=e_a\,, \\
\on{fib}(L_{v,a}\to L_{v,a}[d(v)])& x=v\,, \\
0 & x\neq v,e_a\,.\end{cases}
\end{equation}
If $a=b$ and $v$ is $1$-valent, we have $T_{\mathcal{F}(v)}(\A_{\delta_{v,a}})[d(v)-1]\simeq \A_{\delta_{v,a}}$, see \Cref{def:ask}.v), and we define $\A_{\delta_{v,a,a}}=\on{fib}(\A_{\delta_{v,a}}\to \A_{\delta_{v,a}}[d(v)])\simeq G_v'F_v'(L_e)$ as the corresponding extension. There is a canonical splitting as in \eqref{eq:A(x)split} arising from the fact that $L_e\simeq G_v'(L_v)$.

Let $\delta_{v,a,b}'$ be the same curve as $\delta_{v,a,b}$ but with the reversed clockwise orientation. We define
\[ \A_{\delta_{v,a,b}'}=\A_{\delta_{v,a,b}}[1-d(a,b)]\,.\]

\begin{construction}\label{constr:glsec}
We fix a graded arc $\edge$ in ${\bf S}$. We write $\edge$ as the composite of segments $\delta_1,\dots,\delta_m$ as above, with $\delta_i(1)=\delta_{i+1}(0)$ lying at the midpoint of an edge $e_i$ of $\Sgh$.

The edge $e_1$ canonically defines a graded arc $e_1(t)\colon [0,1]\rightarrow \sow$. Replacing $\edge$ by a homotopic arc if necessary, we can assume that $\edge(t)=e_1(t)$ for some $\epsilon>0$ and all $0\leq t\leq \epsilon$. Using the grading of $e_1$ at $e_1(\frac{\epsilon}{2})$ as a basepoint of the $\ZZ$-torsor $\pi_1(e_1^*\mathbb{P}(TS)_{e_1(\frac{\epsilon}{2})})$, we denote by $g_1$ the grading of $\edge$ at $\edge(\frac{\epsilon}{2})$. For $2\leq i\leq m-1$, we recursively define $g_{i}=g_{i-1}+d(a,b)-1$ if the $i$-th segment of $\edge$ wraps counterclockwise around the vertex of $\Sgh$ from the halfedge $a$ to the halfedge $b$ and $g_{i}=g_{i-1}+1-d(a,b)$ if the $i$-th segment of $\edge$ wraps clockwise from the halfedge $a$ to the halfedge $b$.

Associated to each segment $\delta_i$ we have a lax section $\A_{\delta_i}\in \losec(\eS,\hF)$. By definition of $\A_{\delta_i}(e_i)$, $\A_{\delta_{i+1}}(e_i)$, we find
\[\A_{\delta_i}(e_i)[g_{i}]\simeq L_{e_i}[g_{i+1}]\simeq \A_{\delta_{i+1}}(e_i)[g_{i+1}]\,,\]
unless $e$ is a loop and $\delta_i$ or $\delta_{i+1}$ have two endpoints at $e$. In that case, the above equivalences are replaced by inclusions of direct summands and the construction in the following is adapted accordingly.

We may thus glue $\A_{\delta_i}$ and $\A_{\delta_{i+1}}$ to a lax section $\A_{\delta_i\ast \delta_{i+1}}$ by declaring
\[ \A_{\delta_i\ast \delta_{i+1}}(x)=\begin{cases} \A_{\delta_i}(x)[g_i]\oplus \A_{\delta_{i+1}}(x)[g_{i+1}] & x\neq {e_i}\\ L_{e_i}[g_{i+1}] & x=e_i\end{cases}\]
on objects of $\on{Exit}(\eS)$ and defining $\A_{\delta_i\ast \delta_{i+1}}$ on $1$-simplices in the obvious way. Alternatively, one can describe this gluing as a pushout as follows. We let $Z_{e_i}\in \losec(\eS,\hF)$ be the section with $Z_{e_i}(e_i)=L_{e_i}$ and $Z_{e_i}(x)=0$ for $x\neq e_i$. Then we have morphisms of lax sections
\[ Z_{e_i}[g_{i+1}]\hookrightarrow \A_{\delta_i}[g_i],\A_{\delta_{i+1}}[g_{i+1}]\,,\] which evaluate to the identity on $L_{e_i}[g_{i+1}]$ at $e_i\in \on{Exit}(\eS)$ and vanish everywhere else. The section $\A_{\delta_i\ast\delta_{i+1}}$ is equivalent to the pushout of the following diagram in $\losec(\eS,\hF)$:
\begin{equation}\label{eq:mpu}
\begin{tikzcd}
{Z_{e_i}[g_{i+1}]} \arrow[r, hook] \arrow[d, hook] & {\A_{\delta_i}[g_i]} \\
{\A_{\delta_{i+1}}[g_{i+1}]}                    &
\end{tikzcd}
\end{equation}

If $\edge$ has only two segments, we find that $\A_\edge\coloneqq \A_{\delta_i\ast\delta_{i+1}}\in \Gamma(\eS,\hF)$ already defines a global section, and we are done. This happens for instance if $\edge=e$ is an edge of $\Sgh$. Equipping $e$ with the canonical grading, the construction of $\A_{\edge}$ recovers in this case the construction of $\A_e$ from the beginning of this section.

If $\eta$ has three or more segments, we proceed with gluing $\A_{\delta_{i-1}}$ or $\A_{\delta_{i+2}}$ with $\A_{\delta_i\ast \delta_{i+1}}$ in a similar way. Continuing this process, we obtain the section $\A_{\edge}\in \losec(\eS,\hF)$ after gluing the lax sections of all segments. It is straightforward to verify that $\A_{\edge}$ lies in the full subcategory $\Gamma(\eS,\hF)\subset \losec(\eS,\hF)$, i.e.\ defines a coCartesian, meaning global, section.
\end{construction}

\subsection{Homs from (oriented) intersections}\label{sec:Hom-Int}
\def\inint{\overrightarrow{\bigcap}_{\surf^\circ}}
\def\singint{\overrightarrow{\bigcap}_{\W}}
\def\bdint{\overrightarrow{\bigcap}_{\partial\surf}}

In this section we calculate, in certain cases, the morphisms between objects of the sort constructed in the previous section in terms of counts of intersections of the corresponding graded curves.
This will be used in \Cref{subsec:flipsrevisited} to relate flips of S-graphs to tilting of hearts.

Let $\eS$ be the extended ribbon graph obtained from the dual graph $\Sgh$ and $\hF$ an $\eS$-parametrized perverse schober equipped with an arc system kit.

\begin{definition}
Let $\edge,\zeta$ be two graded arcs in ${\bf S}$. We choose representatives of $\edge$ and $\zeta$ with the minimal number of intersections.
\begin{itemize}
\item An \emph{interior intersection} from $\edge$ to $\zeta$ is a transverse intersection of $\edge$ and $\zeta$ away from their endpoints. We write $\inint(\edge,\zeta)$ for the set of interior intersections from $\edge$ to $\zeta$.
\item A \emph{singular intersection} from $\edge$ to $\zeta$ is an intersection of their endpoints at an interior singular point. We write $\singint(\edge,\zeta)$ for the set of singular intersections from $\edge$ to $\zeta$. Given $x\in \singint(\edge,\zeta)$, we write $v_x$ for the vertex of $\Sgh$ at which the intersection lies.
\item A \emph{boundary intersection} from $\edge$ to $\zeta$ is an intersection at a boundary singular point of $\Sgh$, satisfying the following directedness property: we ask that $\eta$ comes before $\zeta$ when going counterclockwise around $v$ inside ${\bf S}$. We write $\bdint(\edge,\zeta)$ for the set of boundary intersections from $\edge$ to $\zeta$.
\end{itemize}
Each intersection $x$ as defined above is directed, i.e.\ involves an ordering of its involved graded arcs, and this is needed to make sense of the intersection index $d(x)\in \ZZ$ of the intersection (cf. the discussion after Definition~\ref{def:gradedcurve}).
\end{definition}

\begin{proposition}\label{prop:inthom}~
\begin{enumerate}[label=(\arabic*)]
\item Let $e\neq \eprime$ be two edges of $\Sgh$. Then
\begin{equation}\label{eq:homMe} \on{RHom}_{\Gamma(\eS,\hF)}(\A_e,\A_{\eprime})\simeq \bigoplus_{x\in \singint(e,\eprime)}\on{End}_{v_x}[-d(x)]\oplus \bigoplus_{x\in \bdint(e,\eprime)}\Ende[-d(x)]\,,
\end{equation}
see also \Cref{rem:locend} for the definition of $\on{End}_{v_x}$ and $\Ende$.
\item Let $e$ be an edge of $\Sgh$ connecting the vertices $v,\vprime$. Define $\on{End}_{e,v,u}$ as the following pullback in the derived $\infty$-category $\mathcal{D}(k)$.
\[
\begin{tikzcd}[column sep=large]
\on{End}_{e,v,u} \arrow[r] \arrow[d] \arrow[r] \arrow[rd, "\lrcorner", phantom] & \Endv \arrow[d, "\hF(v\to e)"] \\
\on{End}_{\vprime} \arrow[r, "\hF(\vprime\to e)"]                                                   & \Ende
\end{tikzcd}
\]
If $v\neq \vprime$, then
\[ \on{End}_{\Gamma(\eS,\hF)}(\A_e)\simeq \on{End}_{e,v,u} \,.\]
If $v=\vprime$, then $e$ is a loop with halfedges $a\neq b$ at $v$, and
\[ \on{End}_{\Gamma(\eS,\hF)}(\A_e)\simeq \on{End}_{e,v,u} \oplus \Endv[-d(a,b)]\oplus \Endv[-d(b,a)]\,.\]
\end{enumerate}
In particular, if the arc system kit is positive, then $\A_{\Sgh}$ forms a simple-minded collection in the stable $\infty$-category it generates, meaning that for all $e,\eprime\in \Sgh_1$
\[ \on{Ext}^i(\A_e,\A_{\eprime})=\on{H}^i\on{RHom}_{\Gamma(\eS,\hF)}(\A_e,\A_{\eprime})\simeq \begin{cases} k & \text{for }i=0\text{ and }e=\eprime\,,\\ 0 & \text{for }i=0\text{ and }e\neq \eprime\,,\\ 0 & \text{for }i<0\,.\end{cases}\]
\end{proposition}

\begin{proof}
We only prove part (1), part (2) follows from a similar argument. Using that
\[\on{RHom}(\mhyphen,\A_{\eprime})\colon \losec(\eS,\hF)^{\on{op}}\rightarrow \mathcal{D}(k)\]
preserves limits and that $\A_e$ is the pushout of a diagram as in \eqref{eq:mpu}, we find that $\on{RHom}(\A_e,\A_{\eprime})$ fits into a pullback diagram in $\mathcal{D}(k)$ as follows.
\[
\begin{tikzcd}
{\on{RHom}_{\losec(\eS,\hF)}(\A_e,\A_{\eprime})} \arrow[r] \arrow[d] \arrow[rd, "\lrcorner", phantom] & {\on{RHom}_{\losec(\eS,\hF)}(\A_{\delta_1},\A_{\eprime})} \arrow[d] \\
{\on{RHom}_{\losec(\eS,\hF)}(\A_{\delta_2},\A_{\eprime})} \arrow[r]                                 & {\on{RHom}_{\losec(\eS,\hF)}(Z_e,\A_{\eprime})}
\end{tikzcd}
\]
Here $\delta_1,\delta_2$ denote the two segments of the edge $e$ considered as a graded arc. Let $v_1,v_2$ be the two vertices of $\Sgh$ at which $\delta_1,\delta_2$ lie and let $p\colon \Gamma(\hF)\rightarrow \on{Exit}(\Sgh)$ be the Grothendieck construction of $\hF$. The lax sections $\A_{\delta_1},\A_{\delta_2}$ and $Z_e$ are $p$-relative left Kan extensions of their restrictions to $v_1,v_2$ and $e$, respectively. Hence, we find
\begin{align*}
{\on{RHom}_{\losec(\eS,\hF)}(\A_{\delta_1},\A_{\eprime})}& \simeq \on{RHom}_{\hF(v_1)}(\A_{\delta_1}(v_1),\A_{\eprime}(v_1))\\
{\on{RHom}_{\losec(\eS,\hF)}(\A_{\delta_2},\A_{\eprime})}& \simeq \on{RHom}_{\hF(v_2)}(\A_{\delta_2}(v_2),\A_{\eprime}(v_2))\\
{\on{RHom}_{\losec(\eS,\hF)}(Z_e,\A_{\eprime})} & \simeq \on{RHom}_{\hF(e)}(Z_e(e),\A_{\eprime}(e))\simeq \on{RHom}_{\hF(e)}(L_e,0)\simeq 0\,.
\end{align*}
It follows that
\[ \on{RHom}_{\losec(\eS,\hF)}(\A_e,\A_{\eprime})\simeq {\on{RHom}_{\losec(\eS,\hF)}(\A_{\delta_1},\A_{\eprime})}\oplus {\on{RHom}_{\losec(\eS,\hF)}(\A_{\delta_2},\A_{\eprime})}\,.
\]
We only describe the first direct summand above, the second can be described analogously. If $\eprime$ has no endpoints at $v_1$, then by construction $\A_{\eprime}(v_1)\simeq 0$ and hence \[{\on{RHom}_{\losec(\eS,\hF)}(\A_{\delta_1},\A_{\eprime})}\simeq 0\,.\]
 We have $\A_{\delta_1}(v_1)=L_{v_1,a}$ with $a$ the halfedge of $e$ at $v_1$. If $\eprime$ has endpoints at $v_1$, either one or two, then $\A_{\eprime}(v_1)$ is the direct sum of one or two objects $L_{v_1,b}$, where $b$ is a halfedge of $\eprime$ at $v_1$.  By \Cref{lem:locmor}, we have
\[ \on{RHom}_{\hF(v_1)}(L_{v_1,a},L_{v_1,b})\simeq \on{End}_{v_1}[-d(a,b)]\]
if $v_1$ has finite weight and
\[ \on{RHom}_{\hF(v_1)}(L_{v_1,a},L_{v_1,b})\simeq \Ende[-d(a,b)]\]
if $v_1$ has weight $\infty$ and $a<b$. We observe that $d(a,b)=d(x)$ is the degree of the intersection $x$ of $e,\eprime$ at $v_1$. Assembling the results of the above computations yields \eqref{eq:homMe}.
\end{proof}

\begin{remark}
\Cref{prop:inthom} is a special case of a more general description of the derived Homs between two objects $\A_{\edge},\A_{\zeta}$ associated to two graded arcs $\edge,\zeta$ in terms of intersections. We expect that in favorable situations, namely that $\edge$ and $\zeta$ together cut out no discs in ${\bf S}\backslash M$, there is an equivalence
\[
\on{RHom}_{\Gamma(\eS,\hF)}(\A_{\edge},\A_{\zeta})\simeq \bigoplus_{x\in \singint(\edge,\zeta)}\on{End}_{v_x}[-d(x)]\oplus \bigoplus_{x\in \inint(\edge,\zeta)\cup \bdint(\edge,\zeta)}\Ende[-d(x)]\,.
\]
If $\edge$ and $\zeta$ together cut out discs, there are counter-examples to this description, see \cite{Chr21b}.
\end{remark}

The next \Cref{lem:sm} expresses that in typical situations taking cones along morphisms arising from an intersection corresponds geometrically to smoothing the intersection.

\begin{lemma}\label{lem:sm}
\begin{enumerate}
\item Let $e,\eprime$ be two edges of $\Sgh$, with a singular intersection or boundary intersection $x$ from $e$ to $\eprime$ at a vertex $v$.  Let
\[
    \beta_x\colon \A_e\rightarrow \A_{\eprime}[d(x)]
\]
be the morphism corresponding via \Cref{prop:inthom} to the identity in $\Endv$ or $\Ende$ in the summand of the intersection $x$. There exists an equivalence in $\Gamma(\eS,\hF)$
\be\label{eq:cone1}
    \on{fib}(\beta_x)\simeq \A_{\edge} ,
\ee
where $\edge$ is the graded arc obtained from $e$ and $h$ by smoothing out the intersection $x$ in the counterclockwise direction. This is illustrated on the left in Figure~\ref{fig:arc segments}.
\item
Suppose that the arc system kit is positive and let $e$ be an edge of $\Sgh$ incident to two vertices $v,\vprime$ with $\vprime$ of weight $-1$. Let $\eta$ be the graded monogon arc, going from $v$ to $v$ enclosing $\vprime$, such that $\A_{\eta}\simeq \on{cof}(\A_{e}[-1]\rightarrow \A_e)$ is the non-split self-extension of $\A_e$. Let $h\neq e$ be an edge incident to $u$ and $\zeta$ the counterclockwise smoothing of the intersection $y$ of $\eta$ and $h$, as depicted in Figure~\ref{fig:arc segments}. Consider the morphism \[ \beta_y\colon \A_{\eta}\rightarrow \A_{h}[d(e,h)]\]
arising from the commutative square
\[
\begin{tikzcd}
{\A_e[-1]} \arrow[d] \arrow[r] & \A_e \arrow[d, "\beta_x"] \\
0 \arrow[r]                    & {\A_{h}[d(e,h)]}
\end{tikzcd}
\]
where $\beta_x$ is the morphism arising from the intersection, denoted $x$, of $e$ and $h$ at $v$. Then
\be\label{eq:cone2}
    \on{fib}(\beta_y)\simeq \A_{\zeta}\,.
\ee
\end{enumerate}
\end{lemma}

\begin{figure}[ht]
\begin{tikzpicture}[scale=.7,arrow/.style={->,>=stealth}]
\draw[dashed,thick] (0,0) circle (3.6) (.7,1)node[font=\tiny]{$y$};
\draw[->,>=stealth](69:.6)to[bend left=-60](72:1.5);
  \draw[orange!50,thick]
    (0,1).. controls +(-60:4) and +(-120:4) ..(0,1)
    (.8,-.5)node{$\eta$};
  \draw[blue!50,thick]
    (0,1).. controls +(-135:4.5) and +(-70:2) ..(1.2,-1).. controls +(110:2) and +(-120:1) ..(2,3)
    (1.3,1)node{$\zeta$};
  \draw[red,thick]
    (2,3)to node[above]{$\eprime$} (0,1)to(-2,3) (0,0) (0,-.3)node[left]{$e$}
    (2,3)\ww(-2,3)\ww(0,1)\ww to (0,-1)\ww  node[below]{$u$} (0,1) node[above]{$v$}
        (0,1.6) node[font=\tiny]{$\,\cdots$};
\begin{scope}[scale=-1,shift={(10,0)}]
\draw[->,>=stealth](.7,0)to[bend left=-60]node[font=\tiny]{$\;\;^x$}(1.3,-.3);
\draw[dashed,thick] (0,0) circle (3.6);
  \draw[red,thick](3,-2)edge[orange,bend left=-10] node[above]{$\eta$}(-1,0);
\draw[red,thick]
    (-1,0) edge (1,0) edge (-3,2) edge (-3,-2)\ww
    (-1.5,0) node[rotate=90,font=\tiny]{$\cdots$}
    (1.5,0) node[rotate=90,font=\tiny]{$\cdots$}
    (1,0) edge (3,2) edge node[left]{$\eprime$}(3,-2)\ww
    (3,2)\ww(-3,2)\ww(-3,-2)\ww(3,-2)\ww(0,0) node[below]{$e$};
\end{scope}
\end{tikzpicture}
\caption{The smoothings of intersections.}
\label{fig:arc segments}
\end{figure}

\begin{proof}
We only prove part (1), part (2) can be proven analogously. The morphism $\beta_x$ evaluates trivially at all edges of $\Sgh$ and all vertices except $v$. Let $a$ and $b$ be the halfedges of $e$, respectively $h$, at $v$ participating in the intersection $x$. Evaluating at $v$, we obtain the morphism $\beta_x(v)\colon L_{v,a}\to L_{v,b}[d(x)]\simeq L_{v,b}[d(a,b)]$, whose fiber describes the value at $v$ of the section associated with a type II segment in \Cref{sec:objfromarcs}. This shows $\on{fib}(\beta_x)\simeq \A_{\edge}$.
\end{proof}

\subsection{Flip equivalences revisited}\label{subsec:flipsrevisited}
In this section, which is central to the paper, we relate mixed-angulations and their flips to hearts and their tilts in $\infty$-categories of global sections of perverse schobers with positive arc system kit.
We first construct finite hearts whose simples correspond to the edges of an S-graph.
Then, we show that if a perverse schober admits an arc system kit, then so does the induced schober on the flipped S-graph.
Finally, we state and prove the main result of this section on flips and tilts.

As previously, let $\eS$ be the extended ribbon graph obtained from the dual graph $\Sgh$ and $\hF$ an $\eS$-parametrized perverse schober equipped with an arc system kit.

\begin{definition}
We define $\CS\subset \Gamma(\eS,\hF)$ as the idempotent complete stable subcategory generated by the collection of objects $\ASG$ associated with the edges of $\Sgh$, see \Cref{sec:objfromarcs}.
\end{definition}

\begin{proposition}\label{pp:hearts}
If the arc system kit is positive, then $\CS$ admits a finite heart $\CSh$ with simples $\ASG$.
\end{proposition}

\begin{proof}
The positivity of the arc system kit implies by \Cref{prop:inthom} that $\ASG$ is a simple-minded collection. The proposition is thus shown in Corollary 9.5 of \cite{KN12}.
\end{proof}

Let $e$ be an edge of $\Sgh$ and $\Sgh^\flat$ the backward flip of $\Sgh$ at $e$. By \Cref{prop:flipeq}, we get an induced $\eS^\flat$-parametrized perverse schober $\hF^\flat$ and an equivalence of $\infty$-categories $\mu_e^\flat\colon \Gamma(\eS,\hF)\simeq \Gamma(\eS^\flat,\hF^\flat)$. As observed in \Cref{rem:disassembletoask}, we can recover the local objects in the arc system kit by evaluating the objects in $\ASG$. The following data is thus a canonical choice to define (a part of) an arc system kit for $\hF^\flat$:
\begin{itemize}
    \item Let $f$ be an edge of $\Sgh$. Note that $f$ determines an edge $f^\flat$ of $\Sgh^\flat$ (since a backward flip induces a bijection between the edges of the S-graphs). This is understood to mean in the case $f=e$, that $e^\flat=e[-1]$. For $f^\flat\not =e^\flat$, we set $L_{f^\flat}^\flat=\mu_e^\flat(\A_f)(f^\flat)\in \mathcal{F}^\flat(f^\flat)$. We further set $L_{e^\flat}^\flat=\mu_e^\flat(\A_e)(e^\flat)[-1]$.
    \item Let $v$ be a vertex of $\Sgh^\flat$ with an incident halfedge $a$.
    Then $a$ is part of an edge $f$ of $\Sgh^\flat$.
    By considering $f$ as a graded arc, there is an associated global section $\A_{f}\in \glsec(\eS,\mathcal{F})$. If $f\not =e^\flat$, we set $L_{v,a}^\flat=\mu_e^\flat(\A_{f})(v)\in \mathcal{F}^\flat(v)$. If $f=e^\flat$, we set $L_{v,a}^\flat=\mu_e^\flat(\A_{e})(v)[-1]\simeq \mu_e^\flat(\A_{e^\flat})(v)$. Note that if $f$ is a loop, one finds that $\mu_e^\flat(\A_{f})(v)$ has two non-zero direct summands, and we instead choose $L_{v,a}^\flat$ as the direct summand whose image under $\mathcal{F}^\flat(v\xrightarrow{a}f)$ is non-zero.
\end{itemize}

\begin{proposition}\label{prop:arcsystemkitflip}
The data $\{L_f^\flat\}_{f\in \Sgh^\flat_1},\{L_{v,a}^\flat\}_{v\in \Sgh^\flat_0,a\in H(v)}$ defined above extends to an arc system kit for $\mathcal{F}^\flat$.
\end{proposition}

\begin{proof}
To complete the data of an arc system kit for $\hF^\flat$, we need to specify equivalences as in iii)-v) in \Cref{def:ask}.

In the following, to avoid overloaded notation, we identify the vertices of $\Sgh^\flat$ with the vertices of $\Sgh$, and treat $e$ as both an edge of $\Sgh$ and $\Sgh^\flat$. When we consider the corresponding graded arc arising from the edge of $\Sgh^\flat$, we indicate this by writing $e^\flat$. 

Let $v,\vprime$ be the two vertices incident to $e$ (possibly identical). We only treat the case that $e$ is not incident to a singular point of weight $-1$, in the other case a similar argument to below applies. Under the flip at $e$, the S-graph changes by moving between $0$ and $2$ halfedges incident to $v,\vprime$. In the case that the flip moves no halfedges, we have $\mathcal{F}=\mathcal{F}^\flat$ and the statement of the proposition follows immediately.  The case that the flip moves two halfedges is analogous to the case that a single halfedge moves, which we thus treat in the following.

We introduce the following notation: we let $d$ be the halfedge which is moved under the backward flip from $v$ to $u$, and denote by $h$ the edge containing $d$. We let $a_1$ be the halfedge predecessor of $d$ and $a_2$ the halfedge successor of $d$ (in the counterclockwise direction) in $\eS$. We let $b_2$ be halfedge of $e$ incident to $u$ and $b_1$ its  halfedge predecessor in $\eS$. These halfedges are also depicted in \Cref{fig:arc segments2}.

\begin{figure}[ht]
\begin{tikzpicture}[scale=.6,arrow/.style={->,>=stealth}]
\begin{scope}[xscale=1,shift={(-10,0)}]
\draw[dashed,thick] (0,0) circle (3.6);

\draw[red,thick](-1,0)node[left]{$v$}
    (1.8,0) node[rotate=90,font=\tiny]{$\cdots$}
    (-.8,-.4) node[rotate=30,font=\tiny]{$\cdots$}
    (.8,-.4) node[rotate=-30,font=\tiny]{$\cdots$}
    (1,0)node[right]{$\vprime$} (-2.1,1.15)node[left]{$h$}
    (-1,0) edge (1,0) edge (-3,-2)\ww
    (-1,0) edge (-3,2)\ww
    (1,0) edge (3,2) edge (3,-2)\ww
    (3,2)\ww(-3,2)\ww(-3,-2)\ww(3,-2)\ww(0,0) node[below]{$e$};
\draw[font=\tiny](1.2,0.2)node[above]{$b_1$} (-1.2,0.2)node[above]{$d$}
   (.55,.22)node{$b_2$}(-.5,0.2)node{$a_1$}(-1.7,-.3)node{$a_2$};
\end{scope}
\end{tikzpicture}
\caption{The S-graph $\Sgh$ near the edge $e$.}
\label{fig:arc segments2}
\end{figure}

It is straightforward to see that we have all equivalences in iii), and those equivalences as in iv) and v) which do not involve $d$ or halfedges of $e$. We next construct these remaining equivalences.

Let $\eS^\wedge$ be the ribbon graph from which $\eS$ and $\eS^\flat$ arise via contractions $c_1,c_2$, i.e.\ the central one in \Cref{fig:contraction1}. We have an $\eS^\wedge$-parametrized perverse schober $\hF^\wedge$ with equivalences $(c_1)_*(\hF^\wedge)\simeq \hF$ and $(c_2)_*(\hF^\wedge)\simeq \hF^\flat$. We have that $\hF^\flat(v)\simeq \hF^\wedge(v)$. By \cite[Lem.~4.26, Prop.~4.28]{Chr22}, there is a functor $\pi^v\colon \hF(v)\to \hF^\wedge(v)\simeq \hF^\flat(v)$, arising from the pullback diagram involving $\hF(v)$ and $\hF^\wedge(v)$. The functor admits a fully faithful left adjoint $\iota^{v}\colon \hF^{\flat}(v)\simeq \hF^\wedge(v)\to \hF(v)$ which exhibits $\hF^\flat(v)$ as the fiber of the functor $\hF(v\xrightarrow{d} \eprime)$. We note that $\iota^v(L_{v,a_i}^\flat)\simeq L_{v,a_i}$ for $a_i\not =d$ a halfedge at $v$, as follows from unraveling the definitions. Since $\iota^v$ is fully faithful, we thus also have $L_{v,a_i}^\flat\simeq \pi^v(L_{v,a_i})$.

If $v$ has weight $\infty$, all equivalences at $v$ as in iv) thus arise from the data in iv) of the arc system kit of $\hF$. Now assume that $v$ has finite weight.

By the data in v) of the arc system kit of $\hF$, we have $T_{\hF(v)}(L_{v,d})\simeq L_{v,a_2}[1-d(d,a_2)]$. The fiber of $\pi^{v}$ contains the image of the right adjoint $\hF(v\xrightarrow{d} \eprime)^R$ of $\hF(v\xrightarrow{d} \eprime)$. Hence $\pi^{v}\circ G'_vF_v'(L_{v,d})\simeq \pi^{v} \circ \hF(v\xrightarrow{d} \eprime)^R (L_e)\simeq 0$, by the data from iii), with $F_v',G_v'$ as in \Cref{prop:twistv}. Applying $\pi^{v}$ to the fiber and cofiber sequence
\[ L_{v,d}\to G'_vF_v'(L_{v,d})\to T_{\hF(v)}(L_{v,d})\]
we thus obtain by v) an equivalence $\pi^{v}(L_{v,d})\simeq \pi^vT_{\hF(v)}(L_{v,d})[-1]\simeq \pi^{v}(L_{v,a_2})[-d(d,a_2)]\simeq L_{v,a_2}^\flat[-d(d,a_2)]$. Note that $d(d,a_2)=d^\flat(a_1,a_2)$.

Composing with the equivalence obtained from applying $\pi^{v}$ to the equivalence $L_{v,d}\simeq T_{\hF(v)}(L_{v,a_1})$ (note that $1=d(a_1,d)$) yields the equivalence
\[ L_{v,a_2}^\flat[-d^\flat(a_1,a_2)]\simeq \pi^{v} T_{\hF(v)}(L_{v,a_1})\,.\]
We have $L_{e^\flat}\simeq L_e[-1]$, and thus $L_{v,a_1}^\flat\simeq \pi^{v}(L_{v,a_1})[-1]$. We further find that
\[ \pi^{v} T_{\hF(v)}(L_{v,a_1})\simeq T_{\hF^\flat(v)}\pi^{v}(L_{v,a_1})\,,\]
which uses that the unit of $F_v'\circ \iota^v \dashv \pi^{v}\circ G_v'$ at $L_{v,a_1}^\flat$ arises from applying $\pi^{v}$ to the unit of $F_v'\dashv G_v'$ at $L_{v,a_1}$. Combined, this yields $L_{v,a_2}^\flat[1-d^\flat(a_1,a_2)]\simeq T_{\hF^\flat(v)}(L_{v,a_1}^\flat)$, as desired for v). This concludes the construction of the data at the vertex $v$. We proceed with constructing the data at the vertex $u$.

Assume that $\vprime$ has finite weight. We first show that $T_{\hF^\flat(\vprime)}(L_{\vprime,b_2}^\flat)\simeq L_{\vprime,d}^\flat$ (note that $d^\flat(d,b_2)=1$). Unraveling the definition of $\hF^\flat$, one finds the following pullback diagram of $\infty$-categories. More specifically, one uses for this the local model for perverse schobers, see \Cref{prop:ngonschober2}, \Cref{prop:localmodelforschobers}, and \cite[Lem.~4.26]{Chr22}. Note that the existence of the pullback diagram also directly follows from \cite[Lem.~4.26]{Chr22} by using the observation that the spherical functors underlying $\hF$ and $\hF^\flat$ at $u$ agree.
\[
\begin{tikzcd}
\hF^\flat(\vprime) \arrow[r, "\pi^{u}"] \arrow[d] \arrow[rd, "\lrcorner", phantom] & \hF(\vprime) \arrow[d, "\hF(\vprime\xrightarrow{b_2}e)"] \\
{\on{Fun}(\Delta^1,\hF(e))} \arrow[r, "\on{ev}_1"]            & \hF(e)
\end{tikzcd}
\]
Here, $\on{Fun}(\Delta^1,\hF(e))\simeq \hF^\wedge(\vprime^\wedge)$ is the value of $\hF^\wedge$ at the central trivalent vertex $\vprime^\wedge$ in \Cref{fig:contraction1} which is collapsed with $\vprime$ by the contraction $c_2$. The images of the objects $L_{\vprime,b_2}^\flat,L_{\vprime,d}^\flat$ in $\hF(u)$ agree, given by $L_{u,b_2}$, but differ in $\on{Fun}(\Delta^1,\hF(e))$. Namely, the image of $L_{\vprime,b_2}^\flat$ is given by $\left( L_e\xrightarrow{\on{id}} L_e\right)[-1]\in \on{Fun}(\Delta^1,\hF(e))$, and the image of $L_{\vprime,d}^\flat$ is given by $\left(0\rightarrow L_e\right)\in \on{Fun}(\Delta^1,\hF(e))$. One finds that the equivalence $T_{\hF^\flat(\vprime)}(L_{\vprime,b_2}^\flat)\simeq L_{\vprime,d}^\flat$ reduces to the apparent equivalence $T_{\hF^\flat(\vprime^\wedge)}(L_e\xrightarrow{\on{id}} L_e)[-1]\simeq \left( 0\rightarrow L_e\right)$, which arises from the fiber and cofiber sequence
\[ (L_e\xrightarrow{\on{id}} L_e)[-1] \longrightarrow (L_e\to 0)[-1]\longrightarrow (0\to L_e)\,.\]

We next show that $T_{\hF^\flat(u)}(L_{u,d}^\flat)\simeq L_{u,b_1}^\flat[1-d^\flat(b_1,d)]$. Let $(F_u')^\flat\dashv (G_u')^\flat$ be the adjunction of \Cref{prop:twistv} arising from $\hF^\flat$ at $u$. The element $(G_u')^\flat (F_u')^\flat(L_{u,d}^\flat)\in \hF^\flat(u)$ is mapped by $\pi^{u}$ to $G_u' F_u'(L_{u,b_2})$, since $F_u'(L_{u,b_2})\simeq (F_u')^\flat(L_{u,d}^\flat)$ and $G_u'\simeq \pi^{u}\circ (G_u')^\flat$. Thus, $T_{\mathcal{F}^{\flat}(u)}(L_{u,d}^\flat)$ is mapped by $\pi^{u}$ to $T_{\mathcal{F}(u)}(L_{u,b_2})\simeq L_{u,b_1}[1-d(b_1,b_2)]\in \mathcal{F}(u)$, where $d(b_1,b_2)=d^\flat(b_1,d)$.

Further, $(G_u')^\flat (F_u')^\flat(L_{u,d}^\flat)$ is mapped to $(0\to L_e)\in \on{Fun}(\Delta^1,\hF(e))$: this follows from using properties $2^\circ,3^\circ$ from \Cref{def:schberngon} and the observations that the functor $\mathcal{F}^\flat(b_2)$ (where we consider $b_2$ as a morphism in the exit path category) factors as $\mathcal{F}^\flat(u)\to \on{Fun}(\Delta^1,\mathcal{F}(e)) \xrightarrow{\on{ev}_0[1]} \mathcal{F}(e)$ and similarly $\mathcal{F}^\flat(d)$ factors as $\mathcal{F}^\flat(u)\to \on{Fun}(\Delta^1,\mathcal{F}(e)) \xrightarrow{\on{cof}} \mathcal{F}(e)$. We thus obtain that $T_{\mathcal{F}^{\flat}(u)}(L_{u,d}^\flat)$ is mapped to $(0\to 0)\in\on{Fun}(\Delta^1,\mathcal{F}(e))$.


Using the above two observations, we also find that the image of $L_{u,b_1}^\flat$ in $\on{Fun}(\Delta^1,\mathcal{F}(e))$ vanishes, thus agreeing with the value of $T_{\mathcal{F}^{\flat}(u)}(L_{u,d}^\flat)$ in $\on{Fun}(\Delta^1,\mathcal{F}(e))$.

We have thus shown that $T_{\hF^\flat(u)}(L_{u,d}^\flat)$ and $L_{u,b_1}^\flat[1-d^\flat(b_1,d)]$ are equivalent when restricted to $\mathcal{F}(u)$ and vanish in $\on{Fun}(\Delta^1,\mathcal{F}(e))$, and thus also when restricted to $\mathcal{F}(e)$. In total, this implies that $T_{\hF^\flat(u)}(L_{u,d}^\flat)\simeq L_{u,b_1}^\flat[1-d^\flat(b_1,d)]$ in the pullback $\mathcal{F}^\flat(u)$.

Finally, consider the case that $\vprime$ has infinite weight. As in the finite weight case, we find that the objects $L_{u,d}^\flat,L_{u,b_2}^\flat$ both restrict in $\mathcal{F}(u)$ to $L_{u,b_2}$ and that $L_{u,d_1}^\flat$ restricts in $\mathcal{F}(u)$ to $L_{u,d_1}$. The remaining equivalences as in iv) are thus obtained from the corresponding equivalences as in iv) for the arc system kit of $\mathcal{F}$ using the observation that the functor corresponding to the virtual halfedge factors through $\mathcal{F}^\flat(u)\xrightarrow{\pi^{u}}\mathcal{F}(u)$ by \cite[Lem.~4.26.2]{Chr22}. This concludes the construction of the arc system kit for $\hF^\flat$.
\end{proof}

We denote by $\A_\gamma^\flat$ the global section of $\mathcal{F}^\flat$ associated with a graded arc $\gamma$, defined via the arc system kit from \Cref{prop:arcsystemkitflip}. There is an apparent dual version of \Cref{prop:arcsystemkitflip} for the forward flip, giving the global sections $\A_\gamma^\sharp$ of $\mathcal{F}^\sharp$.

\begin{lemma}\label{lem:tilt}
Let $e\in \Sgh_1$ be an edge of the S-graph and $e^\flat=e[-1]$ the corresponding edge of the backward flip $\Sgh^\flat$ at $e$. The flip equivalence
\[ \mu_e^\flat\colon \A(\eS,\hF)\simeq \A(\eS^\flat,\hF^\flat)\]
at $e$ satisfies for all graded arcs $\gamma$
\begin{equation}\label{eq:flipofAe}
\mu_e^\flat(\A_{\gamma})\simeq \A_{\gamma}^\flat\,.
\end{equation}
In particular, we find that $\mu_e^\flat(\A_{e})\simeq \A_{e}^\flat\simeq \A_{e[-1]}^\flat[1]=\A_{e^\flat}^\flat[1]$.

A similar statement holds for the forward flip.
\end{lemma}

\begin{proof}
Denote by $v_1,v_2$ the two vertices incident to $e$. The construction of global sections from arcs is local in the ribbon graph and so is the action of the backward flip equivalence on global sections. It thus suffices to prove the assertion for local arcs in the small weighted marked surface containing $v_1,v_2$.

Let $\Sgh|_e$ be the subgraph of $\Sgh$ containing the edge $e$, $v_1,v_2$ as well as all other halfedges incident to $v_1,v_2$ (these giving rise to external edges). If $e$ is a monogon arc, then $v_1=v_2$, but the argument below does not change in an essential way. Labeling the two corresponding vertices of $\Sgh^\flat$ also by $v_1,v_2$, we similarly define $\Sgh^\flat|_{e}$ as the subgraph of $\Sgh^\flat$ containing $v_1,v_2$ as well as their incident halfedges, giving rise to external edges except for the internal edge $e[-1]$. Denote by $\mathcal{F}|_e$ and $\mathcal{F}^\flat|_e$ the corresponding restricted perverse schober. The graphs $\Sgh|_e,\Sgh^\flat|_e$ describe S-graphs (with the $1$-valent boundary vertices removed, note that this does not affect the categories of global sections) for the weighted marked surface containing $v_1,v_2$ and a boundary singular point for every external edge of $\Sgh|_e$.

The backward flip induces an equivalence of global sections $\mu_e^\flat\colon \glsec(\Sgh|_e,\mathcal{F}|_e)\simeq \glsec(\Sgh^\flat|_e,\mathcal{F}^\flat|_e)$.

The edges of $\Sgh|_e$, considered as graded arcs, give rise to a generating collection of the $\infty$-category $\glsec(\Sgh|_e,\mathcal{F}|_e)$ and other arc objects in $\glsec(\Sgh|_e,\mathcal{F}|_e)$ arise as repeated cones between these objects (which geometrically correspond to smoothing out singular intersections). It thus suffices to show the assertion for arcs given by the edges of $\Sgh|_e$. Inspecting the definition of the arc system kit for $\mathcal{F}^\flat|_e$, it is immediate that $\mu_e^\flat(\A_f)\simeq \A^\flat_f$ for each edge $f\not =e$ of $\Sgh|_e$ that is not moved by the backward flip.

The flip equivalence $\mu_e^\flat$ acts trivially on the full subcategories of $\glsec(\Sgh|_e,\mathcal{F}|_e), \glsec(\Sgh^\flat|_e,\mathcal{F}^\flat|_e)$ consisting of sections which vanish at all external edges of $\Sgh|_e$ and $\Sgh|_e^\flat$. Thus $\mu_e^\flat(\Gamma_e)\simeq \Gamma^\flat_e$.

Finally, it remains to describe the action of $\mu_e^\flat$ on arc objects arising from edges of $\Sgh|_e$ that are moved by the backward flip. For simplicity of notation, we assume that a single edge $f$ of $\Sgh|_e$ is moved. Replacing $\mathcal{F}$ by an equivalent perverse schober, we may assume that $\mathcal{F}(f)=\mathcal{F}(e)$. Reading \Cref{fig:contraction1} from right to left, we obtain the span of contractions corresponding to the backward flip at $e$. We can depict the arising perverse schober $\mathcal{F}^\vee|_e$ (determined up to equivalence) parametrized by the middle ribbon graph $\Sgh^\vee|_e$ in \Cref{fig:contraction1} as follows:
\[
\begin{tikzcd}
      & \dots                                                          &                & \mathcal{F}(e)= \mathcal{F}(f)                                                                        &                &                                                &       \\
\dots & \mathcal{F}^\flat(v_1) \arrow[l] \arrow[d] \arrow[r] \arrow[u] & \mathcal{F}(e) & {\on{Fun}(\Delta^1,\mathcal{F}(e))} \arrow[l, "{\pi_0[1]}"'] \arrow[r, "\pi_1"] \arrow[u, "\on{cof}"] & \mathcal{F}(e) & \mathcal{F}(v_2) \arrow[l] \arrow[r] \arrow[d] & \dots \\
      & \dots                                                          &                &                                                                                                       &                & \dots                                          &
\end{tikzcd}
\]
The dots indicate the values at the remaining external edges. As a coCartesian section of the Grothendieck construction of the above diagram, the object $\A_f\in \glsec(\Sgh|_e,\mathcal{F}|_e)\simeq \glsec(\Sgh^\vee|_e,\mathcal{F}^\vee|_e)$ can be depicted as follows:
\[
\begin{tikzcd}
\dots & 0                                         &   & L_e                                                                          &     &                                                 & \dots \\
0     & {L_{v_1,a}^\flat} \arrow[l] \arrow[d] \arrow[r] \arrow[u] & L_e & (L_e[-1]\to 0) \arrow[l, "{\pi_0[1]}"'] \arrow[r, "\pi_1"] \arrow[u, "\on{cof}"] & 0 & 0 \arrow[l] \arrow[r] \arrow[d] & 0     \\
\dots & 0                                         &   &                                                                              &     & 0                                               & \dots
\end{tikzcd}
\]
Here $a$ denotes the halfedge of $e$ at $v_1$. With this, we can compute the corresponding global section of $\mathcal{F}^\flat|_e$ induced by the contraction equivalence, and find that it is the arc object associated with the graded arc $f$, which is the composite of two segments in the mixed-angulation dual to $\Sgh^\flat|_e$, namely one segment of the first kind at $v_2$ and one segment of the second kind at $v_1$.
\end{proof}

\begin{corollary}\label{cor:mutrestrictstoas}
The equivalence $\mu_e^\flat$ restricts to an equivalence between the stable subcategories
\begin{gather}\label{eq:musharp}
    \mu_e^\flat\colon \CS \simeq \hC({\Sgh^\flat,\hF^\flat})
\end{gather}
generated by $\ASG$ and $\A_{\Sgh^\flat}$ respectively. A similar statement holds for the forward flip.
\end{corollary}

\begin{proof}
Combine \Cref{lem:sm} and \Cref{lem:tilt}.
\end{proof}

Tilting the objects $\ASG$ corresponds to the passage to the objects associated with the edges of the flipped S-graph.

\begin{theorem}\label{thm:tilt=flip}
Suppose that $\hF$ is equipped with a positive arc system kit and let $e\in \Sgh_1$ be an edge. Then $(\mu_e^\flat)^{-1}( \hC({\Sgh^\flat,\hF^\flat})^\heartsuit )$ is the backward tilt of $\CSh$ at $\A_e$ in $\CS$, i.e.
\begin{gather}\label{eq:tilt}
    ( \CSh )_{\A_e}^\flat = (\mu_e^\flat)^{-1}(  \hC({\Sgh^\flat,\hF^\flat})^\heartsuit  ).
\end{gather}
Dually, $( \hC({\Sgh^\flat,\hF^\flat})^\heartsuit  )_{\A_{e[-1]}}^\sharp= \mu_{e}^\flat(\CSh)$.
\end{theorem}
\begin{proof}
We first notice that the hearts in the proposition are all finite by construction.
In the following, we write $\h$ for the RHS of \eqref{eq:tilt} and $\h_0$ for $\CSh$.
Applying \eqref{eq:cone1} (in the usual flip case) or \eqref{eq:cone2} (in the monogon flip case),
we can pair each simple $X^\flat\ne \A_e[-1]$ of $\h$ with a simple $X\ne\A_e$ of $\h$ that sits in a triangle
\[
    L[-1] \to X \to X^\flat \to L
\]
for some $L$, which is either zero, $\A_e$ or is a self-extension of $\A_e$. Note that $L$ is a self-extension of $\A_e$ only in the monogon case when the simple $X^\flat$ corresponds to the arc $\zeta$ in the right picture of \Cref{fig:arc segments}, cf. \Cref{ex:s.tilting}.
Thus we deduce that $\h_0[-1]\le \h \le \h_0$, see \Cref{app:preliminaries} for the notation.
Hence (e.g.\ by \cite[Rem.~3.3]{KQ}) $\h$ is a backward tilt of $\h_0$ with respect to the torsion pair
\[\hua{F}_e=\h\cap\h_0\quad\text{and}\quad\hua{T}_e=\h[1]\cap\h_0\,.\]
Let $\<\A_e\>\subset \mathcal{H}_0$ be the extension closure of $\A_e$. We have $\<\A_e\>\subset \hua{T}_e$ and $X\in\hua{F}_e$ for any simple $X\ne \A_e$ in $\h_0$.

Now take any $M\in\hua{T}_e=\h[1]\cap\h_0$, by Prop.~\ref{pp:tilting},
there is a short exact sequence in $\h_0$
$$0\to M'\to M \xrightarrow{f}  \widehat{\A_e} \to0$$
for $\widehat{\A_e}\in\<\A_e\>$, $f$ a left $\<\A_e\>$-approximation and $\Hom_{\mathcal{H}_0}(M',\A_e)=0$.
Then $f[-1]$ is surjective in $\h$, which implies that $M'[-1]=\ker(f[-1])\in\h$ and $M'\in\hua{T}_e$.
Using that $\Hom_{\mathcal{H}_0}(M', X)=0$ for any other simple $X\ne \A_e$ in $\h_0$ (since $X\in\hua{F}_e$),
we deduce that $\Hom_{\mathcal{H}_0}(M',\Sim\h_0)=0$, which implies $M'=0$.
We thus have $M=\widehat{\A_e}$, and therefore $\hua{T}_e=\<\A_e\>$.
In other words, $\h=(\h)^\flat_{\A_e}$ as claimed.
\end{proof}

\begin{remark}
Note that when flipping at a usual arc (cf. \Cref{fig:flip1}),
one only needs to apply the formulae in \Cref{rem:tilting} (for rigid simple tilting) to prove the proposition above.
In the case of flipping at a monogon arc (cf. \Cref{fig:flip2}),
one has the formulae in \Cref{pp:tilting}, or more precisely, in the case of \Cref{ex:s.tilting} for $S=\A_e[-1]$ and $\h$ as above.
\end{remark}

\section{Application: quadratic differentials as stability conditions}
In this section, we apply the construction of hearts from positive arc system kits to prove the correspondence between
quadratic differentials and stability conditions. For this, we use \Cref{thm:tilt=flip}, which essentially states that the flip equivalence from \Cref{prop:flipeq} acts on the corresponding arc system as a simple tilt. This leads to an isomorphism between the exchange graphs of S-graphs and hearts.
We then follow the strategy of \cite{BS,BMQS} to obtain an isomorphism between the complex manifolds of quadratic differentials and stability conditions. Note that our setting is more general than the one in \cite{BMQS}, in that we allow simple poles, double poles, and exponential singularities of quadratic differentials (which correspond to the boundary vertices of the S-graphs).

\subsection{Generic-finite components via exchange graph of finite hearts}\label{subsec:exgraphfinitehearts}
Before specializing to the case of surfaces and schobers with a positive arc system kit, we recall here the general relation between the exchange graph of hearts and the space of stability conditions of a triangulated category.

Let $\D$ be a triangulated category. For the definition of the space of stability conditions $\Stab\D$ we refer to \cite{B1}.

\begin{lemma}\cite{B1}
Giving a stability condition $\sigma$ on $\D$ is equivalent to giving a stability function on a heart $\h$ of $\D$ satisfying the HN-property.
We say that such $\sigma$ are \emph{supported on $\h$}.

In particular, if $\h$ is finite, then
the coordinates $\{Z(S) \mid S\in \Sim\h\}$ give an isomorphism $\cub(\h)\isom\UHP^{\Sim\h}$,
where $\cub(\h)$ consists of stability conditions supported on $\h$.
\end{lemma}
Here $\UHP$ denotes the union of the open upper half plane in $\mathbb{C}$ with the negative real axis.

Let $\EGp\D$ be a connected component of the exchange graph of hearts in $\D$
consisting of finite hearts.
For each simple tilting $\h\to\h'=\tilt{\h}{\sharp}{S}$,
there is codimension $1$ wall where $Z(S)\in \RR_{>0}$,
\[
 \partial^\sharp_S\cub(\h) \coloneqq \partial^\flat_{S[1]}\cub(\h')
 \coloneqq \overline{\cub(\h)}\cap\overline{\cub(\h')}.
\]
Denote by
\begin{gather}
    \Stap_0\D=\bigcup_{\h\in\EGp\D} \cub(\h),\\
    \Stap_2\D=\Stap_0\D(\A)\cup\bigcup_{\h\in\EGp\D,\, S\in\Sim\h} \partial^\sharp_S\cub(\h),
\end{gather}
which is connected, and $\Stap\D$ the connected component containing it. In particular, $\Stap\D$ also contains $\CC\cdot\Stap_2$.
Note that the subset $\Stap_2\D$ is analogous to the set $B_2$ of quadratic differentials as in Section~\ref{subsec:isomoduli}.

By the same argument as in \cite[\S3]{Q2},
the cell structure determines an embedding of the exchange graph as a skeleton
for the space of stability conditions, uniquely up to homotopy,
\begin{gather}\label{eq:embed}
    \skel_{\D}\colon \EGp\D \to \Stap\D.
\end{gather}
so that the image is dual to $\Stap_2\D\setminus \Stap_0\D$.

\begin{definition}
Suppose that any heart in $\EGp\D$ is finite as above.
$\Stap\D$ is called
\begin{itemize}
    \item a \emph{finite} type component, if $\Stap\D=\Stap_0\D$;
    \item a \emph{generic-finite} type component, if $\Stap\D=\CC\cdot\Stap_0\D$
    and it is not of finite type.
\end{itemize}
\end{definition}

\cite{QW} proved that, if any $\h$ in $\EG_0$ has only finitely many torsion pairs, then the finite type component $\Stab_0\D$ is contractible.
The components from surfaces are often of generic-finite type, cf. \cite{BS, HKK17, KQ2, BMQS}.

\subsection{Isomorphisms between exchange graphs}\label{subsec:exgraphs}
We fix a weighted marked surface $\sow$ with an initial mixed-angulation $\AS$
with dual S-graph $\Sgh=\AS^*$ and extended ribbon graph $\eS$. We also fix an $\eS$-parametrized perverse schober $\hF$ together with a positive arc system kit. By definition, $\CS$ is the smallest idempotent complete and stable subcategory of the $\infty$-category of global sections $\glsec(\eS,\hF)$ containing the objects $\ASG$ associated with the edges of $\eS$.

We consider the exchange graphs $\EG(\CS)$ of finite hearts, see the previous section, $\EG(\sow)$ of mixed-angulations, see \Cref{def:flip}, and $\EG_S(\sow)$ of S-graphs, see \Cref{rem:f.f.}.

There is a map between directed graphs $\EG(\sow)\to \EG_S(\sow)$, which sends a mixed-angulation to its dual. Note that this map is not necessarily injective (e.g.\ the rotation of the marked points on a boundary component, corresponding to a higher order pole, changes a mixed-angulation but usually not the dual S-graph). These two graphs are thus not canonically isomorphic, and we use $\EG_S(\sow)$ to compare with the exchange graph of finite hearts.
Nevertheless, each connected component of $\EG(\sow)$ maps to a connected component of $\EG_S(\sow)$, and both graphs are $(m,m)$-regular, meaning that each vertex is $2m$-valent with each $m$ arrows coming in and coming out. Here $m$ is the number of interior edges in any mixed-angulation.

A direct corollary of \Cref{thm:tilt=flip} is the following.
\begin{corollary}\label{cor:iso}
The arc-to-object correspondence induces an injection
\begin{gather}\label{eq:iso}
\begin{array}{rcl}
    \iota_\A\colon\EG_S(\sow) &\to& \EG(\CS)\\
    \T&\mapsto&\Ch_{\T},
\end{array}
\end{gather}
where $\Ch_{\T}$ is the heart generated by the simples $\Sim\Ch_\T=\A_\T$.
Denote by $\EGb(\CS)$ the image of $\iota_{\Gamma}$, which consists of connected components.
In particular, any heart in $\EGb(\CS)$ is finite.
\end{corollary}

\begin{proof}
For any S-graph $\T$ together with a sequence of forward and backward flips from $\Sgh$ to $\T$, we obtain a $\T$-parametrized perverse schober $\hF'$ together with an equivalence $\glsec(\Sgh,\hF)\simeq \glsec(\T,\hF')$ by \Cref{prop:flipeq}.
Then we pull back the canonical heart $\hC(\T,\hF')^\heartsuit$ of $\hC({\T,\hF'})$ to
a heart $\Ch_\T$ in $\CS$ generated by the set of its simples $\A_\T$.
That forward flips correspond to simple forward tiltings can be deduced from \Cref{thm:tilt=flip}.
What is left to show is injectivity of the map $\iota_\A$.
In fact, we will prove the stronger injectivity of the arc-to-object (arc-to-global section) correspondence $\eta\mapsto\A_\eta$.

Consider two graded arcs $e$ and $h$ with $\A_e\simeq \A_h$. We show that $e=h$.
Using the equivalences from \Cref{sec:EfromC}, we find that it suffices to consider the case when $e$ is an edge of $\Sgh$. The global section $\A_{h}\simeq \A_e$ evaluates trivially at all objects in $\on{Exit}(\eS)$ arising from edges of $\Sgh$, except $e$. If $h$ has segments away from $e$, we thus have $\A_e\not\simeq \A_h$. If $h$ has more than two segments, we also find $\A_{e}(e)\not \simeq \A_{h}(e)$. It follows that $e=h$.
\end{proof}

\subsection{Isomorphisms between moduli spaces}\label{subsec:isomoduli}
Next, we upgrade the injection in Corollary~\ref{cor:iso} into an injection between complex manifolds, whose image is a union of connected components. As the proof closely follows the ones in \cite{BS,BMQS},
we only highlight the differences caused by exponential singularities.
For standard terminologies such as (saddle/recurrent) trajectories, see e.g.\ \cite[\S~3]{BS}.

To start with, we have $\FQuad{}{\sow}$, the moduli space of $\sow$-framed quadratic differentials.
There is a stratification
\[
    B_0 = B_1 \subset B_2 \subset \cdots \subset \FQuad{}{\sow}\,
\]
where the subspace
\[
    B_p \,\coloneqq\, B_p(\sow) = \{q \in  \FQuad{}{\sow}\colon r_q + 2s_q \leq p\}
\]
is defined by the number $s_q$ of saddle trajectories and the number $r_q$ of recurrent trajectories.
Let $F_p=B_p\setminus B_{p-1}$.

\begin{lemma}
Let $m_2$ be the number of order 2 poles.
Then $B_p=\FQuad{}{\sow}$ if $p\ge 4\dim \Hone(\sow)+m_2$. Furthermore, we have
$\FQuad{}{\sow}=\CC\cdot B_0(\sow)$.
\end{lemma}
\begin{proof}
The set of saddle trajectories is linearly independent in $\Hone(\sow)$, cf. \cite[Lem.~3.2]{BS} and thus $s_q\le \dim \Hone(\sow)$.
As for each recurrent trajectory, it is bounded by a ring domain, either around a double pole or some (at least one) saddle trajectories. On the other hand, each saddle trajectory can appear in at most two boundaries of recurrent trajectories.
Thus, $r_q\le 2 s_q+m_2$. To sum up, the first assertion holds.

For the second statement, we only need to notice that there are finitely many points in $\W$ and there are at most countably many compact arcs which are saddle trajectories.
\end{proof}

Next, recall from Subsection~\ref{subsec_quad} that there is the period map \eqref{eq:period}, defining a local homeomorphism. Applying the surface perturbation arguments of \cite[\S~4]{BMQS}, we obtain the following result.
In the following, polar type $(-2)$ refers to the case that $\partial\sow=\emptyset$ and $\M\subset\sow^\circ$ consists of a single point/puncture.

\begin{proposition}\label{prop:pathinBp}
Suppose that the polar type of $\FQuad{}{\sow}$ is not $(-2)$.
If $p>2$, then each component of the stratum~$F_p$ contains a point~$q$
and a neighbourhood $U \subset \FQuad{}{\sow}$ of~$q$ such that
$U \cap B_{p}$ is contained in the locus $\int_\alpha \sqrt{q} \in \RR$ for some $\alpha \in \Hone(\sow)$, and that this containment is strict in the more precise sense that $U \cap B_{p-1}$ is connected.

As a consequence, we find that any path in $\FQuad{}{\sow}$ is homotopic relative to its endpoints to a path in $B_2$.
\end{proposition}

In particular, connected components of $\FQuad{}{\sow}$ are in one-to-one correspondence with connected components of $\EG(\sow)$.
Denote by $\FQuad{S}{\sow}$ the quotient moduli space of $\FQuad{}{\sow}$,
where we identify two connected components if the corresponding exchange graphs of mixed-angulations share the same dual exchange graphs of S-graphs.

As in \cite{BS,KQ2,BMQS}, we can complexify the isomorphism in \Cref{cor:iso} as follows:
\begin{itemize}
\item Using the correspondence $\eta\mapsto\Gamma_\eta$ for $\eta$ in the initial S-graph $\Sgh$,
one identifies the Grothendieck group and the homology group $\Hone$ as $\kappa\colon K_0(\CS)\cong\Hone(\sow)$.
\item We first define a map $\iota_\A$ from $B_0(\sow)$, the saddle-free part of $\FQuad{S}{\sow}$, to $\Stab(\CS)$.
More precisely, given a saddle-free quadratic differential $\phi$,
the union of those saddle connections that cross a single horizontal strip is an S-graph $\T_\phi$ of $\sow$ and hence corresponds to a heart $\Ch_\T$.
Then we map $\phi$ to the stability condition $\sigma$ determined by $\Ch_\T$ with central charge
given by the period map in \cref{eq:periodmap}, i.e.\ $Z=\int\circ\;\kappa$.
\item Extend the map $\iota_\A$ from $B_0$ to $B_2$ continuously by compatible $\CC$-actions on both sides.
This is the analogue of \cite[Prop.~10.7]{BS},
as the key is the identification between the Grothendieck group and the homology group above and how they change after tilting or flip (which follows from \eqref{eq:iso}).
\item Finally, extend the map $\iota_\A$ from $B_p$ to $B_{p+1}$ inductively.
Here, we use the same argument as \cite[Prop.~5.8]{BS}, cf.\ paragraph Extension-to-non-tame-differentials in \cite[\S~7.2]{BMQS},
and the walls-have-ends property on any connected component of $F_p$ \Cref{prop:pathinBp} is essential.
\item A remark is that the $\sow$-framing (also known as Teichm\"{u}ller framing) of the quadratic differentials identifies connected components of the exchange graphs (of S-graphs) and of the moduli spaces.
    See \cite[Sec.~4.1]{KQ2} for detailed discussion on (two versions of) framing.
\end{itemize}

Hence we obtain the result, as in \cite{BS,KQ2,BMQS}.
\begin{theorem}\label{thm:app} 
Consider a weighted marked surface $\sow$ whose set of marked points is non-empty and that is not a closed surface with a single marked point. Assume further that we are given an initial mixed-angulation with dual S-graph $\Sgh=\AS^*$ and extended ribbon graph $\eS$
and an $\eS$-parametrized perverse schober $\hF$ together with a positive arc system kit. 
The injection $\iota_\A$ in Corollary~\ref{cor:iso} extends to a map between complex manifolds
\be\label{eq:app}
    \FQuad{S}{\sow} \to \Stab(\CS)\,,
\ee
which is an isomorphism onto its image. We denote its image by $\Stab^\bullet(\CS)$, which consists of generic-finite components corresponding to $\EGb(\CS)$.
\end{theorem}

\section{First examples of perverse schobers with arc system kits}\label{sec:examples}
In this section, we construct two classes of perverse schobers admitting positive arc system kits. In each case, we can use Theorem~\ref{thm:app} to describe the corresponding spaces of stability conditions. A more elaborate class of examples, generalizing those of~\cite{BS}, is explored in the paper \cite{CHQ24}. We begin in \Cref{subsec:sphericalfibrationstori} with perverse schobers arising from spherical fibrations between tori, providing a simple generalization of the constructions of \cite{Chr22,Chr21b} to arbitrary weighted marked surfaces with weights $\geq 0$. In \Cref{subsec:topFukviaschobers}, we describe topological Fukaya categories as the global sections of perverse schobers parametrized by S-graphs and show that these admit positive arc system kits.

\subsection{Spherical fibrations between products of spheres}\label{subsec:sphericalfibrationstori}
\def\psP{\hF_{\Sgh,\NUM}}
By a space, we mean a Kan complex; these assemble into the $\infty$-category of spaces. A spherical fibration $f\colon A\rightarrow B$ is a Kan fibration between Kan complexes, such that the fiber $f^{-1}(b)$ of any point $b\in B$ is homotopy equivalent to the singular simplicial set of the topological $m$-sphere with $m\geq 0$, denoted $S^m$. A simple class of examples of spherical fibrations arises as follows. We fix a finite set $\NUM\subset \mathbb{N}_{\geq 2}$ of integers greater or equal to two and denote by
\[
    \prodS^\NUM\coloneqq \prod_{i\in\NUM} S^{i-1}
\]
the product of spheres. We allow $\NUM=\emptyset$, in which case we set $\prodS^\NUM=\ast$. For each $j\in \NUM$, we have the spherical fibration
\begin{equation}\label{eq:sphfib}
    f_j=f^{\NUM}_j\colon \prodS^\NUM \rightarrow \prodS^{\NUM\setminus \{j\}}\,,
\end{equation}
collapsing one $(j-1)$-sphere to a point.

Given a space $X$, we denote $\on{Loc}(X)\coloneqq \on{Fun}(X,\mathcal{D}(k))$. The $k$-linear $\infty$-category $\on{Loc}(X)$ is called the $\infty$-category of $\mathcal{D}(k)$-valued local systems on $X$. Given a morphism of spaces $f\colon X\rightarrow Y$, we have an associated pullback functor $f^*\colon \on{Loc}(Y)\rightarrow \on{Loc}(X)$, which admits left and right adjoints $f_!,f_*\colon \on{Loc}(X)\rightarrow \on{Loc}(Y)$, given by left and right Kan extension, see \cite[4.3.3.7]{HTT}. If $f$ is a spherical fibration, the corresponding adjunction $f^*\dashv f_*$ is spherical, see \cite[Prop.~3.14]{Chr20}.

The simplest case of a spherical fibration as in \eqref{eq:sphfib} is the map $S^{n-1}\rightarrow \ast$. Perverse schobers built from the corresponding spherical adjunction were considered in \cite{Chr22,Chr21b}. Their global sections were shown to describe the derived $\infty$-categories of relative Ginzburg algebras of $n$-angulated surfaces. In this section, we consider the more general class of perverse schobers which are locally described by the spherical adjunctions arising from the spherical fibrations from \eqref{eq:sphfib}.

Recall that $\spider_j$ denotes the $j$-spider (Definition~\ref{def:spider}).
\begin{construction}\label{constr:sphfibschober}
Let $j\in \NUM$. For all $1\leq k\leq j$, the following two types of data are equivalent:
\begin{enumerate}
\item[i)] an identification of the halfedges of $\spider_{k}$ with a subset of the halfedges of $\spider_j$, respecting the cyclic order on halfedges.
\item[ii)] integers $d(a,b)\geq 1$ for all consecutive halfedges of $\spider_{k}$ at $v$, such that the sum of all these integers is equal to $j$.
\end{enumerate}
One passes from the data as in i) to data as in ii) by setting $d(a,b)=l$ for two consecutive halfedges $a,b$ of $\spider_{k}$ if $b$ follows in $\spider_j$ after the halfedge $a$ after $l$ steps.

Consider the $\spider_{j}$-parametrized perverse schober $\hF_{j}(f_j^*)$ arising via \Cref{prop:ngonschober2} from the spherical adjunction
\[
    f_j^*\colon \on{Loc}(\prodS^{\NUM\setminus\{j\}})\longleftrightarrow \on{Loc}(\prodS^\NUM)\cocolon (f_j)_*
\]
induced by the spherical fibration $f_j$ from \eqref{eq:sphfib}. Let $1\leq k<j$ and make a choice of data as in i) above. We further choose a total order on the halfedges of $\spider_k$, compatible with their cyclic order. By \Cref{prop:ngonschober}, we obtain from $\hF_{j}(f_j^*)$ a $\spider_{k}$-parametrized perverse schober, denoted $\hF_{k}(f_j^*)$, by replacing $\hF_{j}(f_j^*)(v)$ with the iterated fibers of the functors $\hF_{j}(f_j^*)(v\xrightarrow{c}e_c)$, with $c\in e_c$ a halfedge of $\spider_j$ but not a halfedge of $\spider_{k}$. We remark that the notation $\hF_{k}(f_j^*)$ leaves the choice of total order and data as in i) or ii) implicit.
\end{construction}

Let $\sow$ be a weighted marked surface, without singular points of weight $-1$. Let further $\Sgh$ be a choice of S-graph of $\sow$. Let $\NUM\subset \mathbb{N}_{\geq 2}$ be a set of numbers containing the degrees (i.e.\ weights $+2$) of the vertices of $\Sgh$.
We choose for each internal vertex of the extended ribbon graph $\eS$ a total order of its incident halfedges. Given a vertex $v$ of $\eS$ of degree $j<\infty$ and valency $1\leq k\leq j$, the S-graph data of $\Sgh$ gives rise to data as in ii) in \Cref{constr:sphfibschober}. We associate with $v$ the $\spider_{k}$-parametrized perverse schober $\hF_{k}(f_j^*)$ from \Cref{constr:sphfibschober}. With $v\in \eS_0$ of degree $\infty$ and valency $k$, we instead associate the $\spider_{k}$-parametrized perverse schober $\hF_{k}(0)$ of \Cref{prop:ngonschober2} associated with the spherical functor $0\colon 0\rightarrow \on{Loc}(\prodS^{\NUMS})$.

We define $\psP\colon\on{Exit}(\eS)\rightarrow \on{LinCat}_k$ as the gluing of the $\hF_{k}(f_j^*)$'s and $\hF_m(0)$'s, meaning that $\psP$ is the unique diagram which restricts at the subcategories $\on{Exit}(\eS_{k})\subset \on{Exit}(\eS)$ arising from any given vertex and its incident halfedges to the corresponding such diagram.

\begin{proposition}\label{prop:arcsyskitsphfib}
    The perverse schober $\psP$ admits an arc system kit with
    \begin{itemize}
        \item $L_e\in\psP(e)=\on{Loc}(\prodS^{\NUMS})$ the constant local system with value $k$ for all edges $e$ of $\Sgh$,
        \item $L_{v,a_i}\in\psP(v)$ given by the object \eqref{eq:locLv}, with $L_v \in \on{Loc}(\prodS^{\NUMS\setminus\{j\}})$ the constant local system with value $k$, for all vertices $v$ of finite weight $j$ of $\Sgh$ and halfedges $a_i$ at $v$.
        \item $L_{v,a_i}\in\psP(v)$ given by the object \eqref{eq:locLv2}, with $L=L_e\in \on{Loc}(\prodS^{\NUMS})$ the constant local system with value $k$, for all vertices $v$ of infinite weight of $\Sgh$ and non-virtual halfedges $a_i$ at $v$.
    \end{itemize}
\end{proposition}

\begin{proof}
This follows from the discussion in \Cref{constr:arcsyskit}, the observation that the pullback functor $f_j^*$ maps constant local systems to constant local systems, and the fact that the twist functor of $f_j^*\dashv (f_j)_*$ acts on the constant local system with value $k$ as the delooping $[1-j]$, see \cite[Prop.~3.19]{Chr20}.
\end{proof}

\begin{lemma}\label{lem:sphfibarcsystemispositive}
The arc system kit from \Cref{prop:arcsyskitsphfib} is positive.
\end{lemma}

Before we prove \Cref{lem:sphfibarcsystemispositive}, we need to recall some aspects of the relation between singular cohomology and derived endomorphism algebras of local systems. Given a space $X$, we denote by $C^\bullet(X)\in \mathcal{D}(k)$ the singular cochain complex of $X$. It can be defined as the limit of the constant local system $\underline{k}_X\colon X\rightarrow \mathcal{D}(k)$ with value $k$. Let $\pi\colon X\to \ast$. We have $\on{lim}(\underline{k}_X)\simeq \pi_*(\underline{k}_X)$, with $\pi_*\colon \on{Loc}(X)\rightarrow \on{Loc}(\ast)$ the right adjoint of the pullback functor $\pi^*$. The derived $\infty$-category $\mathcal{D}(k)$ of the field $k$ is symmetric monoidal, and $\on{Loc}(X)$ inherits the pointwise symmetric monoidal structure, see for instance \cite[Lem.~3.15]{Chr20}. We can describe the functor $\pi^*$ in terms of these monoidal structures as the functor $\mhyphen\otimes_k \underline{k}_X$. Its right adjoint $\pi_*$ is thus described by the functor $\on{RHom}_{\on{Loc}(X)}(\underline{k}_X,\mhyphen)$. It follows that $C^\bullet(X)$ is quasi-isomorphic to the derived endomorphism algebra $\on{End}_{\on{Loc}(X)}(\underline{k}_X)$.

Given two spaces $X,Y$, we find that $C^\bullet(X\times Y)\simeq C^\bullet(X)\otimes C^\bullet(Y)\in \mathcal{D}(k)$. To see this, we note that the limit over $X\times Y$ is obtained by first taking the right Kan extension along the projection $X\times Y\rightarrow Y$, and then taking the right Kan extension along $Y\rightarrow \ast$. The outcome of the former is the constant local system on $Y$ with value $C^\bullet(X)$. This local system is equivalent to $C^\bullet(X)\otimes \underline{k}_Y$. The outcome of the latter Kan extension is thus $\on{lim}(C^\bullet(X)\otimes \underline{k}_Y)\simeq C^\bullet(X)\otimes \on{lim}(\underline{k}_Y)\simeq C^\bullet(X)\otimes C^\bullet(Y)$, as desired.

\begin{proof}[Proof of \Cref{lem:sphfibarcsystemispositive}.]
By the above discussion, we have equivalences
\[ \Ende=\on{End}_{\on{Loc}(\prodS^{\NUMS})}(L_e)\simeq C^\bullet(\prodS^{\NUMS})\simeq \bigotimes_{i\in \NUM}C^\bullet(S^{i-1})\]
and
\[
    \Endv=\on{End}_{\on{Loc}(\prodS^{\NUMS\setminus\{j\}})}(L_e)\simeq C^\bullet(\prodS^{\NUMS\setminus\{j\}})\simeq \bigotimes_{i\in \NUM, i\neq j}C^\bullet(S^{i-1})\,,
\]
with $v$ a vertex of degree $j$. Using that $C^\bullet(S^{i-1})\simeq k\oplus k[1-i]$, see for instance \cite[Lem.~3.2]{Chr22}, the positivity of the arc system kit follows.
\end{proof}

\begin{remark}
If $\sow$ has no interior singular points and $\NUM=\emptyset$, then $\psP$ has no singularities and the generic stalk is $\on{Loc}(\ast)\simeq \mathcal{D}(k)$. Hence $\Gamma(\eS,\psP)$ describes a topological Fukaya category of ${\bf S}$ as defined in \cite{DK15,HKK17}.
\end{remark}

\begin{remark}\label{rem:CYstructuresforschobers}
Assume that $\Lambda$ contains an odd integer $j$. The construction of the perverse schober $\psP$ can be slightly modified as follows (without affecting the arc system kit), such that the $\infty$-category of global sections $\glsec(\Sgh,\psP)$ admits a relative left Calabi--Yau structure of dimension $1+\prod_{i\in \Lambda}(i-1)$. Let $v$ be any vertex of degree $j$ of $\Sgh$ with incident edges $e_1,\dots,e_k$. Combining \cite[Thm.~5.7]{BD19} and \cite[Prop.~5.2]{Chr23} yields a left Calabi--Yau structure on the functor
\[
\on{Loc}(\prodS^\NUM)^{\times k}= \prod_{i=1}^k \hF_{k}(f_j^*)(e_i)\longrightarrow \hF_{k}(f_j^*)(v)
\]
obtained from $\hF_{k}(f_j^*)$ by passing to left adjoints. This relative Calabi--Yau structure restricts to an absolute left Calabi--Yau structure of dimension $\prod_{i\in \Lambda}(i-1)$ on each copy of $\hF_{k}(f_j^*)(e_i)$. Since each edge is incident to two vertices, we obtain in this way for each edge $e$ two left Calabi--Yau structures on $\hF_{k}(f_j^*)(e)$, whose corresponding Hochschild homology classes are by construction either identical or differ by sign. If the Hochschild homology
classes differ by a sign, the Calabi–Yau structures are said to be compatible. If the Calabi--Yau structures are not compatible, we choose a halfedge of $e$ lying at the vertex $v$ and change $\mathcal{F}(v\to e)$ by composition with the autoequivalence of $\on{Loc}(\prodS^\NUM)$ obtained from pulling back along the homeomorphisms $\prodS^\NUM\to \prodS^\NUM$ given by the product of $\on{id}_{S^i}$, $i\in \Lambda\backslash\{j\}$, and the orientation reversing antipodal map $S^{j-1}\to S^{j-1}$ (where $j\in \Lambda$ is the chosen odd integer). Note that this autoequivalence acts as $-\on{id}$ on the corresponding Hochschild homology class. After these changes, all Calabi--Yau structures at the edges are thus compatible and we can apply \cite[Thm.~5.7]{Chr23} to deduce the existence of a relative left Calabi--Yau structure on $\glsec(\Sgh,\psP)$.

If $\Lambda$ contains no even integer, the $\infty$-category $\glsec(\Sgh,\psP)$ admits a relative Calabi--Yau structure if the local choices of Calabi--Yau structures can be made compatible, and this amounts to a condition on the S-graph. A more detailed discussion in the case of $n$-angulated surfaces can be found in \cite[Def.~6.6]{Chr23}.
\end{remark}

\subsection{Topological Fukaya categories of infinite area flat surfaces}\label{subsec:topFukviaschobers}

Fix a weighted marked surface $\sow$ with a mixed-angulation and dual S-graph $\Sgh$. Consider the oriented real blow-up ${\bf S}^{\on{blw}}$ of $\sow$ at the set $\Delta \cup (M\cap \sow^\circ)$ of singular points and interior marked points, meaning that an open neighborhood of each interior such point is removed, yielding a new boundary circle, and each boundary singular point is expanded to a closed interval. We refer to these boundary intervals as the marked boundary.

We modify $\Sgh$ to the graph denoted $\rgraph_{\Sgh}$, which comes with an embedding into ${\bf S}^{\on{blw}}$, by moving each interior vertex (lying at a singular point) by a small amount in an arbitrary direction and adding a loop at the vertex, wrapping around the boundary circle given by the oriented blow-up of the singular point, see Figure~\ref{fig:blowup}.

\begin{figure}[ht]
\begin{tikzpicture}[scale=.5,arrow/.style={->,>=stealth}]
\clip(-9,-3)rectangle(9,3);
\begin{scope}[xscale=1,shift={(-5,0)}]
\draw[red,thick](0,0)edge(3,3)edge(3,-3)edge(-3,3)edge(-3,-3) \ww;
\end{scope}\draw(0,0)node{\Huge{$\leftrightsquigarrow$}};
\begin{scope}[xscale=1,shift={(5,0)}]
\draw[red,thick](-2,0)
    .. controls +(45:7) and +(-45:7) ..(-2,0);
\draw[thick,fill=gray!20](0,0)circle(.8);
\draw[red,thick](-2,0)edge[bend left=30](3,3)edge[bend left=-30](3,-3)
    edge(-3,3)edge(-3,-3) \ww;
\end{scope}
\end{tikzpicture}
\caption{The passage from $\Sgh$ to $\rgraph_{\Sgh}$ near an interior singular point.}
\label{fig:blowup}
\end{figure}

Recall that a graph $\rgraph$ is called a spanning graph of the marked surface ${\bf S}^{\on{blw}}$ if there is an embedding $i\colon |\rgraph|\subset {\bf S}^{\on{blw}}$ of its geometric realization, which is a homotopy equivalence such that each marked boundary interval contains the image of a unique point in $|\rgraph|$.

\begin{lemma}\label{lem:Sgraphspanning}
The graph $\rgraph_{\Sgh}$ is a spanning graph of ${\bf S}^{\on{blw}}$.
\end{lemma}

\begin{proof}
The assertion follows from the fact that the canonical inclusion $i\colon |\Sgh|\subset \sow \backslash (M\cap \sow^\circ)$ is a homotopy equivalence, which restricts to a bijection between the boundary vertices of $\Sgh$ and the singular points on the boundary $\Delta\cap \partial \sow$. This in turn follows from the observation that the inclusion $i$ restricts on each polygon of the mixed-angulation to a homotopy equivalence.
\end{proof}

As shown in \cite{DK15,HKK17}, there is an associated $\ZZ$-graded topological Fukaya category, which we will denote by $\on{Fuk}({\bf S}^{\on{blw}})$, which arises as the global sections of a constructible cosheaf of $A_\infty$-categories or dg-categories defined on the extended graph $\rg$ of any choice of spanning graph $\rgraph$ of ${\bf S}^{\on{blw}}$. We can choose a pointwise Morita-equivalent constructible cosheaf on $\rg_{\Sgh}$ which is valued in dg-categories, and apply the functor $\mathcal{D}(\mhyphen)\colon \on{dgCat}[W^{-1}]\to \on{LinCat}_k$, to obtain a cosheaf valued in $\on{LinCat}_k$. Note that this process preserves global sections, i.e.\ (homotopy) colimits. Passing to the right adjoint diagram, we obtain a constructible sheaf $\hF$ on $\rg_{\Sgh}$, which is again valued in $\on{LinCat}_k$, since if a $k$-linear functor preserves compact objects, then its right adjoint is again $k$-linear. One can check that this defines a $\rg_{\Sgh}$-parametrized perverse schober with generic stalk $\mathcal{D}(k)$. This perverse schober assigns to each $n$-valent vertex of $\rgraph$ the $\on{Ind}$-complete topological Fukaya $\infty$-category of the $n$-gon (which carries an induced grading structure). The topological Fukaya category of any $n$-gon is equivalent to the derived $\infty$-category of the $A_{n-1}$-quiver. It follows from the equivalence $\mathcal{D}(kA_{n-1})\simeq \mathcal{V}^n_{0}$, with $0\colon 0\to \mathcal{D}(k)$ the zero functor, see also above \Cref{prop:ngonschober2} for the notation, that the perverse schober has no singularities. In the following, we describe a way to obtain the $\on{Ind}$-complete topological Fukaya $\infty$-category $\on{Ind}\on{Fuk}({\bf S}^{\on{blw}})\coloneqq \mathcal{D}(\on{Fuk}({\bf S}^{\on{blw}}))\simeq \glsec(\rg_{\Sgh},\hF)$ as the global sections of a novel and related $\eS$-parametrized perverse schober $\mathcal{G}$. In contrast to $\hF$, the perverse schober $\mathcal{G}$ has singularities, one at each singular point of finite weight of $\sow$. We then show that $\mathcal{G}$ admits a positive arc system kit, and use this to describe the space of stability conditions of the topological Fukaya category.

Let $\rg_{\on{tri}}$ be the ribbon graph equipped with a contraction $c\colon \rg_{\on{tri}}\to \rg_{\Sgh}$, satisfying that $c$ is the identity at the boundary vertices, and restricts at each internal vertex to the contraction depicted in \Cref{fig:Gtri}, where the loop encloses an interior singular point in $\sow$.
\begin{figure}[!ht]
\[
\begin{tikzpicture}[scale=.7,arrow/.style={->,>=stealth}]
\begin{scope}[xscale=1,shift={(-4,0)}]
  \draw[red,thick](-1,0)edge (-1,-2) edge (-1,2)
                  (-3,0)to(1,0)\ww(-1,0)\ww
                  (1,0) coordinate (x)
                  (-1,0) node[below left]{$\cdots$} node[above left]{$\cdots$} ;
\draw[red,thick] (x) .. controls +(-45:3) and +(45:3) .. (x)\ww;
\draw (2,0) node{$\bullet$};
\end{scope}
\draw[Emerald,ultra thick, opacity=.9,arrow](-.5,0)to(.5,0);
\begin{scope}[xscale=1,shift={(4,0)}]
  \draw[red,thick](1,0)edge (1,-2) edge (1,2)
                  (-3,0)to(1,0)\ww
                  (1,0) coordinate (x)
                  (-1,0) node[below left]{$\cdots$} node[above left]{$\cdots$} ;
\draw[red,thick] (x) .. controls +(-45:3) and +(45:3) .. (x)\ww;
\draw (2,0) node{$\bullet$};
\end{scope}
\end{tikzpicture}
\]
\caption{The contraction $c\colon \rg_{\on{tri}}\to \rg_{\Sgh}$ near an internal vertex.}\label{fig:Gtri}
\end{figure}
Thus, all loops of $\rg_{\on{tri}}$ enclosing interior singular points in $\sow$ are incident to a trivalent vertex. We obtain a $\rg_{\on{tri}}$-parametrized perverse schober $c^*(\hF)$. Let $w$ be a trivalent vertex of $\rg_{\on{tri}}$ with an incident loop. A choice of two generating Lagrangians in the topological Fukaya-category of the $3$-gon centered around the vertex $w$ gives an equivalence $\hF(w)\simeq \on{Fun}(\Delta^{1},\mathcal{D}(k))$. Using this equivalence, we can find an explicit algebraic description of $c^*(\hF)$ near the loop in terms of the grading of the surface. The grading assigns the Maslov index $m\geq 1$ to the closed curve describing the loop. Recall also that the weight of the corresponding interior marked point of $\sow$ is $m-2$. The perverse schober $c^*(\hF)$ is locally equivalent to the following diagram:
\[
\begin{tikzcd}
\mathcal{D}(k) & {\on{Fun}(\Delta^1,\mathcal{D}(k))} \arrow[l, "\on{ev}_{0}"] \arrow[r, "{\on{cof}[m]}", bend left] \arrow[r, "{\on{ev}_1}"', bend right] & \mathcal{D}(k)
\end{tikzcd}
\]
Note that the perverse schober $\hF$ has clockwise monodromy $[m]$ around this loop, in the sense defined in \cite{Chr23}. The limit of the above diagram (which describes the sections with support in a neighborhood of the loop) is equivalent to the derived $\infty$-category $\mathcal{D}(k[t_{m}])$ of the graded polynomial algebra $k[t_{m}]$ with $|t_{m}|=m$ (using cohomological grading). Restriction defines a spherical functor $\xi^m_!\colon \mathcal{D}(k[t_{m}])\rightarrow \mathcal{D}(k)$, which can be described in terms of the morphism of dg-algebras $\xi^m\colon k[t_{m}]\xrightarrow{t_{m}\mapsto 0} k$. The trivial module $\underline{k}\in \mathcal{D}(k[t_{m}])$ is the image of the right adjoint $(\xi^m)^*$ of $\xi^m_!$; this is the $(1-m)$-spherical object giving rise to the spherical adjunction $\xi^m_!\dashv (\xi^m)^*$, since $(\xi^m)^*\simeq \mhyphen \otimes_k \underline{k}$. The cofiber sequence
\begin{equation}\label{eq:twist_cof_seq_ktm}
k[t_{m}][-m]\xlongrightarrow{\cdot t_{m}} k[t_{m}] \longrightarrow \underline{k}
\end{equation}
shows that
\begin{equation}\label{eq:twistktm} T_m(k[t_{m}])\simeq k[t_{m}][1-m]\end{equation} with $T_m$ the twist functor of the adjunction $\xi^m_!\dashv (\xi^m)^*$.

We cut off the loops at the trivalent vertices incident to a loop enclosing a finite-angle conical point as follows:
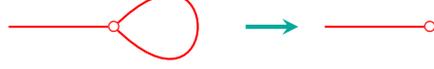
\[
\begin{tikzpicture}[scale=.7,arrow/.style={->,>=stealth}]
\clip(-9,-1)rectangle(9,1);
\begin{scope}[xscale=1,shift={(-4,0)}]
  \draw[red,thick](-1,0)to(1,0)
                  (1,0) .. controls +(-45:3) and +(45:3) .. (1,0)\ww;
\end{scope}
\draw[Emerald,ultra thick, opacity=.9,arrow](-.5,0)to(.5,0);
\begin{scope}[xscale=1,shift={(2,0)}]
  \draw[red,thick](-1,0)to(1,0)
                  (1,0) (1,0)\ww;
\end{scope}
\end{tikzpicture}
\]
We denote the resulting ribbon graph by $\rg_{\on{tri},\on{cut}}$. We let $\hF_{\on{cut}}'$ be the $\rg_{\on{tri},\on{cut}}$-parametrized perverse schober which is identical to $c^*(\hF)$ away from the new $1$-valent vertices and given by the spherical functor $\xi^m_!$ at the new $1$-valent vertices. The contraction $c$ induces a contraction $c_{\on{cut}}$ from $\rg_{\on{tri},\on{cut}}$ to the ribbon graph $\rg_{\Sgh,\on{cut}}$, obtained by removing all the loops from $\rg_{\eS}$. Note that $\rg_{\Sgh,\on{cut}}=\eS$, i.e.\ we recover the extended graph of the chosen S-graph $\eS$. We obtain the $\eS$-parametrized perverse schober $\hF_{\on{cut}}=(c_{\on{cut}})_*(\hF_{\on{cut}}')$.

\begin{proposition}
There exists an equivalence of stable $\infty$-categories
\[ \on{Ind}\on{Fuk}({\bf S}^{\on{blw}})\simeq \glsec(\eS,\hF_{\on{cut}})\,.\]
\end{proposition}

\begin{proof}
We have $\on{Ind}\on{Fuk}({\bf S}^{\on{blw}})\simeq \glsec(\rg_{\Sgh},\hF)$ and \Cref{lem:contr} implies that
\[\glsec(\rg_{\Sgh},\hF)\simeq \glsec(\rg_{\on{tri}},c^*(\hF)) \quad\text{and}\quad \glsec(\eS,\hF_{\on{cut}})\simeq \glsec(\rg_{\on{tri},\on{cut}},\hF_{\on{cut}}').\] The limit $\glsec(\rg_{\on{tri}},c^*(\hF))$ of $c^*(\hF)$ can be computed in two steps, by first computing the local limits at the loops enclosing finite-angle conical points and then computing the remaining global limit. This is a general phenomenon and  follows for instance from \cite[4.2.3.10]{HTT}. The outcome of the first step yields up to equivalence the diagram $\hF_{\on{cut}}'$, showing that their limits agree.
\end{proof}

The edges of the S-graph $\Sgh$ define a set of objects $\A_{\Sgh}=\{\A_e\}_{e\in \Sgh_1}$ in the topological Fukaya category. These can be seen as graded Lagrangians (graded arcs) in ${\bf S}^{\on{blw}}$, whose ends either lie on a marked boundary point or which wrap infinitely many times around a boundary component containing no marked points.

\begin{proposition}
There is a positive arc system kit for $\hF_{\on{cut}}$ giving rise to the objects $\A_{\Sgh}$.
\end{proposition}

\begin{proof}
We consider the objects $\{\A_e\}_{e\in \Sgh_1}$ as coCartesian sections of the Grothendieck construction of $\hF_{\on{cut}}$. It is straightforward to see that evaluating these coCartesian sections yields an arc system kit with $L_e=\A_e(e)$ and $L_{v,a}=\A_e(v)$ for each edge $e$ of $\rg$ and halfedge $a$ of $e$ lying at the vertex $v$. Note for this that at each $1$-valent vertex $v$ with halfedge $a$ of $\rg_{\on{tri},\on{cut}}$ arising from a cut-off loop, the datum $L_{v,a}$ is given by a shift of $k[t_m]\in \D(k[t_m])\simeq \mathcal{F}'_{\on{cut}}(v)$. The positivity of this arc system kit is also clear, where in the case $m=1$ we use that the equivalence \eqref{eq:twistktm} arises from the cofiber sequence \eqref{eq:twist_cof_seq_ktm}.
\end{proof}

\begin{remark}
    The stable $\infty$-category $\hC_{\eS}$ generated by $\A_{\Sgh}$ is the $\infty$-categorical version of the topological Fukaya category of the marked surface ${\bf S}^{\on{blw}}$ (without $\on{Ind}$-completing).
\end{remark}


\bibliography{biblio}
\bibliographystyle{alpha}

\end{document}